\documentclass{article}

\usepackage[french, english]{babel}

\usepackage[letterpaper,top=2cm,bottom=2cm,left=3cm,right=3cm,marginparwidth=1.75cm]{geometry}

\usepackage{amsmath}
\usepackage{graphicx}
\usepackage[colorlinks=true, allcolors=blue]{hyperref}
\usepackage{xr}
\title{The dependency digraph for irreducible finite range random walk in Free Groups}
\author{Chevalier Guillaume}

\usepackage[T1]{fontenc} 
\usepackage[utf8]{inputenc} \usepackage{amsmath,amsthm,amsfonts,amscd,amssymb,graphics,color,xspace}  \usepackage{xr}
\usepackage{cite} 
\usepackage{mathabx} 
\usepackage{hyperref}
\usepackage{tikz-cd} 
\usepackage{tikz}	 %
\usepackage{enumitem}
\usepackage{imakeidx}
\usepackage{etoolbox}
\usepackage{caption}
\usepackage{subcaption}
\usepackage{multicol}
\usepackage{dsfont} 
\usepackage{esvect}
\usepackage{pdfpages}
\usepackage{graphics}
\usepackage{graphicx}
\usepackage{setspace}

\addto\captionsfrench{}

\usetikzlibrary{tikzmark,shapes.misc}

\usepackage{geometry}

\makeindex
\begin{document}
\newtheorem{Thm}{Theorem}[section]
\newtheorem{Prop}[Thm]{Proposition} 
\newtheorem{Lem}[Thm]{Lemma} 
\newtheorem{Cor}[Thm]{Corollary} 
\newtheorem{Clm}[Thm]{Claim} 
\newtheorem{prop}[Thm]{Property} 
\newtheorem*{Thm*}{Theorem} 
\newtheorem*{Cor*}{Corollary} 
\newtheorem*{Prop*}{Proposition}
\theoremstyle{definition} 
\newtheorem{Def}[Thm]{Definition} 
\theoremstyle{remark}
\newtheorem{Rem}[Thm]{Remark} 
\newtheorem{Notation}[Thm]{Notation}
\newtheorem{Ex}[Thm]{Example}  
\newtheorem*{Rem*}{Remark}
\newtheorem*{Ex*}{Example}
\numberwithin{equation}{section}

\newcounter{nmbrexercise}
\newenvironment{Exercise}[1][ ] 
{
  \par\vspace{\baselineskip}%
 {\refstepcounter{nmbrexercise} \color{PineGreen} \noindent \textbf{Exercise~\thenmbrexercise :} \textit{#1}}%
  \par\vspace{\baselineskip}%
}%
{\vspace{\baselineskip}}

\maketitle

\begin{abstract}
This article is one of a triptych composed with \cite{Rtaub} and \cite{Asymp}, that aims at proving an asymptotic expansion to any order of the passage probability of an irreducible equivariant finite range random walk on a tree. In this text we study the analytic and geometric properties of the objects introduced in \cite{Asymp}. To do so we introduce the definition of a flooded cavern tree enabling a precise study of some dependency digraph associated with a given irreducible finite-range random walk on a free group.
\end{abstract}

\begin{center}
	\textbf{\Large Introduction}
\end{center}

The objective of the triptych (\cite{Rtaub}, \cite{Asymp} and the current text), is to prove the theorem \ref{Theoreme_central} below concerning asymptotic of passage probabilities of an irreducible equivariant finite-range random walk on some infinite tree with bounded valence that admits a group of automorphisms acting cofinitely on the set of vertices. A typical example is an irreducible finite-range random walk on the Cayley graph of a free group with at least two generators, with group of automorphisms, the free group itself (See details in \cite{Asymp} Example \ref{Asymp-Ex_ifrrw}).

\begin{Thm}[Theorem \ref{Asymp-Asymptotics_of_probability_Green_General} in \cite{Asymp}]\label{Theoreme_central}
		~
		Let $X=(X_0,X_1)$ be an infinite tree of bounded valence such that any vertex possesses at least three neighbours, $\Gamma<\operatorname{Aut}(X)$ be a subgroup of automorphisms of the tree such that $\Gamma\backslash X_0$ is finite, $p:X_0\times X_0\to [0,1]$ be a finite range $\Gamma$-invariant transition kernel, making $(X_0,p)$ an irreducible Markov chain and let $k\in\mathbb{N}$ be an integer such that 
		$$\forall x,y\in X_0,\, d(x,y)>k\Rightarrow p(x,y)=0.$$
		Denote by $d$ the period of $(X_0,p)$ and $r:X_0\times X_0\to \mathbb{Z}/d\mathbb{Z}$, the periodicity cocycle (Definition \ref{Asymp-DEF:PERIODICITY_COCYCLE} in \cite{Asymp}).
		
		For all $x,y\in X_0$, there exist constants $C>0,\, (c_l)_{l\geq 1}$ such that for any integer $K\geq 1$, we have the asymptotic expansion as $n$ goes to $\infty$:
		$$ p^{(dn+r(x,y))}(x,y) = C R^{-dn}\frac{1}{n^{3/2}}\left( 1 + \sum_{l=1}^{K-1} \frac{c_l}{n^{l}} +O\left(\frac{1}{n^K}\right)\right),$$
		and $p^{(dn+t)}(x,y;\{y\}^\complement)=0$, if $t\neq r(x,y)\,[d]$.
	\end{Thm}
	
To prove this theorem, we study the analytic properties of the generating functions associated with the sequence of probabilities $(p^{(n)}(x,y))_{n\in\mathbb{N}}$ for $x,y\in X_0$, called \textit{Green's functions}. In \cite{Asymp} we prove that the Green's functions can be written as the composition of a rational function over an affine algebraic curve, with a parametrization of this algebraic curve $\mathcal{C}$. The curve $\mathcal{C}$ is called the \textit{Lalley's curve} in the name of S. P. Lalley who studied it in \cite{Lalley_1993}. It has the particular form\footnote{To be more rigorous we should consider the Lalley's curve to be the irreducible variety of $\mathcal{C}$, which indeed is a curve (See \cite{Lalley_1993} Proposition 3.1). Besides this irreducible component coincide with $\mathcal{C}$ on the part that is of interest to us.} $\mathcal{C}=\{(z,J)\in\mathbb{C}\times E : J=z\psi(J)\}$, where $E\approx \mathbb{C}^N$ is a $\mathbb{C}$ vectorial space with finite dimension $N\geq 1$, and $\psi:E\to E$ is a polynomial function with positive coefficients (See \cite{Asymp} sections \ref{Asymp-Section_Lalley_s_curve} and \ref{Asymp-Section_Green_function_in_a_tree} for the details).
\begin{Prop}
Under the assumptions of Theorem \ref{Theoreme_central}, there exists an algebraic variety $\mathcal{C}$ of the form 
$$\mathcal{C}=\{(z,J)\in\mathbb{C}\times E : J=z\psi(J)\},$$
with $\psi:E\to E$ a polynomial function with non-negative coefficients, over a complex vector space of finite dimension (See Definition (\ref{Asymp-Def_de_psi}) in \cite{Asymp}). The Lalley's curve $\mathcal{C}$ comes naturally with a projection $\lambda:(z,J)\mapsto z$ from $\mathcal{C}$ to $\mathbb{C}$, that admits a section in a neighbourhood of the origin $u:z\mapsto (z,v_z)\in \mathbb{C}\times E$, i.e that $u$ satisfies $u(0)=(0,0)$ and $\lambda\circ u= Id$.

 For any pair $(x,y)$ of vertices of $X$, there exist rational functions $g_{x,y}, f_{x,y}:\mathbb{C}\times E \to \mathbb{C}$ that are regular near the origin and for any complex number $z$ of modulus sufficiently small, we have
		\begin{equation}\label{Eq_Green_G}
		 G_z(x,y)=g_{x,y}(z,v_z), \; \text{ and }\; F_z(x,y)=f_{x,y}(z,v_z),
		\end{equation}	
		where $z\mapsto (z,v_z)$ is a parametrization.
\end{Prop}
\begin{Rem}
	In \cite{Asymp}, admitting the result of Proposition \ref{Regularity_of_C} below, about smoothness of the Lalley's curve $\mathcal{C}$ in a neighbourhood of $\overline{\{u(z):z\in\mathbb{D}(0,R)\}}$, we've proven (Proposition \ref{Asymp-Extension_of_algebraicity_say} in \cite{Asymp}) that the rational functions $g_{x,y}$ for $x,y\in X_0$ are regular in a neighbourhood of $\overline{\{u(z):z\in\mathbb{D}(0,R)\}}$.
\end{Rem}
In the current text, we study the geometric properties of the curve $\mathcal{C}$ in the neighbourhood of $\overline{\{u(z):z\in\mathbb{D}(0,R)\}}$, in particular we prove its smoothness in Proposition \ref{Regularity_of_C}, and we prove that the projection $\lambda:(z,J)\mapsto z$ from $\mathcal{C}$ to $\mathbb{C}$ has a zero first derivative at $u(R)$ and a non-zero second derivative at $u(R)$ (Propositions \ref{Prop_derivative_of_lambda} and \ref{Prop_derive_seconde_de_lambda}). Besides we prove that the first derivative of the rational functions $g_{x,y}$ at $u(R)$ are non-zero. These last results coupled with the Tauberian theorems in \cite{Rtaub} are essential to prove Theorem \ref{Theoreme_central} (See the proof of Theorem \ref{Asymp-Asymptotics_of_probability_Green_General} in \cite{Asymp}).

The key ingredient in the inspection of the Lalley's curve is the study the \textbf{dependency digraph}, associated with the system of polynomial equations defining the Lalley's curve. This directed graph is, using the notation of the above theorem, the simple directed graph with set of vertices $\{1,...N\}$ and with an edge from a vertex $i$ to another vertex $j$, if the polynomial function $J\mapsto \psi_j(J)$ of the $j$-th coordinate of $\psi$, depends on the coordinate $J_i$. We prove that this graph has one \textbf{strong connected component}, which is a connected component of the digraph with the property that, for any vertex in the digraph there exists a path from this vertex to a vertex in the strong connected component. Alan R. Woods in \cite{Woods_1997} encountered the same structure while computing the limit distribution of the ways of colouring the root of a finite simple descending rooted tree, following some types of coloring rules, as the number of vertices goes to infinity. Actually we construct an infinite directed graph, such that its quotient with respect to some group of automorphisms is isomorphic to the dependency digraph defined above. This infinite digraph will be shown to be the limit of some family of descending rooted tree that we call \textit{flooded cavern trees}, which are some particular intervals graphs.

\section{The Dependency Digraph}\label{Section_The_Graph}
Before going further, we invite the reader to take note of the two sections \ref{Asymp-Section_Decomposition_of_admissible_paths} and \ref{Asymp-Section_Lalley_s_curve} of \cite{Asymp}.
	\begin{Def}[Directed Graph]\label{Def_Directed_Graph}\index{Directed graph}
		A \textit{directed graph}, or \textit{digraph}, is the data of a quadruple $\mathcal{G}=(G_0,G_1,s,b)$ constituted of two sets $G_0$ and $G_1$ and of two $G_0$-valued maps defined on $G_1$, $s:G_1\to G_0$ and $a:G_1\to G_0$ called \textit{source} and \textit{aim} respectively. The set $G_0$ is called set of vertices of the graph $\mathcal{G}$, and $G_1$ is called the set of edges of the graph $\mathcal{G}$.
		
		A \textit{(directed) geodesic path} in $\mathcal{G}$ is a finite sequence $(e_1,...,e_n)$ of edges of $\mathcal{G}$, such that for any integer $ 1<i\leq n$ the source $s(e_i)$ of $e_{i}$ equals the aim  $a(e_{i-1})$ of $e_{i-1}$.
		The integer $n$ is called the \textit{length} of the path, $s(e_1)$ is called the \textit{source} of the path, and $a(e_n)$ is called the \textit{aim} of the path.
		
		The directed graph $\mathcal{G}$ will be said to be \textit{half-connected} if for any pair $(\theta,\theta')$ of distinct vertices of $\mathcal{G}$, there exists a path with source $\theta$ and aim $\theta'$, or a path with source $\theta'$ and aim $\theta$.
		
		The directed graph $\mathcal{G}$ will be said to be \textit{connected} if for any pair $(\theta,\theta')$ of distinct vertices of $\mathcal{G}$, there exists a path with source $\theta$ and aim $\theta'$.
		
		A \textit{connected component}\index{Connected component of a directed graph} of the directed graph $\mathcal{G}$, is a \textit{directed subgraph}\footnote{A \textit{directed subgraph} of a directed graph $\mathcal{G}=(G_0,G_1,s,a)$ induced by a subset of vertices $C_0\subset G_0$, is the directed graph $(C_0,C_1,s',a')$, where $C_1=s^{-1}(C_0)\cap a^{-1}(C_0)$, and $s'=s|_{C_1}$, $a'=a|_{C_1}$.} induced by a subset $C\subset G_0$ of vertices of the graph, with the property that for any vertex $\theta$ in $C$, and any vertex $\theta'$ in $G_0$, if there exists a path with source $\theta$ and aim $\theta'$, and a path with source $\theta'$ and aim $\theta$, then $\theta'$ actually belongs to $C$.
		
		The graph will be said to be \textit{simple}\index{Simple graph} if there are no two distinct edges $e_1,e_2$ of the graph, having the same source and the same aim, i.e such that $s(e_1)=s(e_2)$ and $a(e_1)=a(e_2)$. In this case we can actually see the set of edges $G_1$ as a subset of pairs of vertices, $G_1\subseteq G_0\times G_0$, with the source map being the projection on the first coordinate and the aim map, the projection on the second coordinate.
	\end{Def}
	\begin{Rem}
		Edges of directed graph considered here will be denoted with decreasing indices \og $(\theta_1,\theta_0)$ \fg{}. The reason for that comes from Lemma \ref{Proposition_of_the_ecluse}, in which we will consider directed geodesic path $(\theta_n,...,\theta_0)$ starting from a vertex that is at distance $n$ from a base point vertex $\theta_0$, so that the vertex $\theta_j$, for $0\leq j \leq n$, is at distance $j$ from $\theta_0$.
	\end{Rem}
    \begin{Rem}
        In a directed graph, there does not exists path of length zero. The directed subgraph induced by a connected component is a connected directed graph. A directed graph that only have one vertex is always connected by definition. 
    \end{Rem}
We are still working under the assumptions described in the previous section, which we now recall,
 
Let $X=(X_0,X_1)$ be a tree of bounded valence with set of vertices $X_0$, and set of edges $X_1$. Let $\Gamma<\operatorname{Aut}(X)$ be a subgroup of automorphisms of $X$, that acts cofinitely on $X_0$ (i.e the set  $\Gamma\backslash X_0$ of orbits of $X_0$ under the action of $\Gamma$, is finite.) Let $p:X_0\times X_0 \to [0,1]$ be a finite range $\Gamma$-invariant transition kernel that makes the Markov chain $(X_0,p)$ irreducible. 
Let $k$ be a non-negative integer, such that for any vertices $x,y\in X_0$, if $d(x,y)>k$ then $p(x,y)=0$.

Moreover, we suppose that each vertex of $X$ has at least $3$ neighbours, that is to say, for any vertex $x\in X_0$, there exists three distinct vertices $y_1,y_2,y_3 \in X_0$ such that $(x,y_1),(x,y_2)$ and $(x,y_3)$ are edges of $X$. This last assumption, will only play a role in the proof of Proposition \ref{Prop_V_infini_is_absorbent}, which is a key proposition in the study of the Lalley's curve $\mathcal{C}$.
		
Under these assumptions, the $\Gamma$-orbits set $\Gamma\backslash\Xi$ of elements in $\Xi$, is finite (Proposition \ref{Asymp-Prop_Action_cofinite} in \cite{Asymp}), and so the vector space $E$ of $\Gamma$-invariant functions defined on $\Xi$ has finite dimension. We also defined a polynomial map $\psi:\mathcal{F}(\Xi,\mathbb{C})\to\mathcal{F}(\Xi,\mathbb{C})$ leaving stable the subspace $E=\mathcal{F}(\Xi,\mathbb{C})^\Gamma$ of complex $\Gamma$-invariant functions over $\Xi$, used to construct the Lalley's curve $\mathcal{C}$, which is parametrized by restricted Green's functions (Proposition \ref{Asymp-Green_restricted_as_parametrisation_of_C} in \cite{Asymp}). 

In this section we propose to study the iterations $(\psi^{\circ n})_n$. More precisely, for $(a,b)_{y}$ fixed in $\Xi$, we will wonder which elements $(a',b')_{y'}\in\Xi$ are involved in the definition of the polynomial function $J\mapsto\psi^{\circ n}(J)((a,b)_y)$.  Recall that for all $(a,b)_{y}$ in $\Xi$, we have
\begin{equation}\label{eq-t2}
	\begin{split}
		\psi(J)((a,b)_{y})&=p(a,b)+\sum_{c\in\mathcal{B}(y)^\complement}p(a,c)\cdot J_{[c,y]}(c,b)\\
		&=p(a,b)+\sum_{c\in\mathcal{B}(y)^\complement}p(a,c)\left(\sum_{\substack{(c_0,...,c_l)\in\Xi_{[c,y]}\\ c_0=c\,, c_l=b}} J_{c,y}(c_0,...,c_l)\right),
	\end{split}
\end{equation}
where the notations \og $J_{[c,y]}(c,b)$\fg{} and\og $J_{c,y}(c_0,...,c_l)$\fg{} have been introduced in notations \ref{Asymp-notation_J_[z,y]} in \cite{Asymp}.

For our study, we provide $\mathcal{F}(\Xi,\mathbb{C})$ with the base $\{1_\theta\}_{\theta\in\Xi}$, and $E$ with the base $\{1_{[\theta]}\}_{[\theta]\in\Gamma\backslash\Xi}$. We will always think of $E$ as being $\mathbb{C}^{\dim_\mathbb{C}(E)}$.

\begin{Def}\textit{Factors of monomials of $\psi$}\label{Def_monomials_first_order}\index{Factor of monomial}

	For $\theta_0=(a_0,b_0)_{y_0}$ and $\theta_1=(a_1,b_1)_{y_1}$ two elements in $\Xi$, we will say that $J(\theta_1)$ \textit{is a factor in one of the monomials of} $J\mapsto\psi(J)(\theta_0)$ if there exist $c\in\mathcal{B}(y)^\complement$ such that $p(a_0,c)>0$, and $(c_0,...,c_l)\in\Xi_{[c,y_0]}$ with $c_0=c$ and $c_l=b_0$ such that for some $j\in\{1,...,l\}$, we have $\theta_1=(c_{j-1},c_j)_{x_{i_j}}$, where we have denoted $(x_0,...,x_m)$ the geodesic segment $[c,y_0]$ in $X$, and $(i_1,...,i_l)=\iota_{[c,y_0]}(c_0,...,c_l)$ the crossing indices associated with $(c_0,...,c_l)$, along $[c,y_0]$ (See Lemma \ref{Asymp-Lem_2.6} in \cite{Asymp}). We will also say that such a product \og $J_{c,y_0}(c_0,...,c_l)$\fg{} \textit{is a monomial of} $J\mapsto\psi(J)(\theta_0)$ (Compare with Definition \ref{Asymp-Def_monomial} in \cite{Asymp}).

\end{Def}
\begin{Rem}\label{Rem_post_def_Factors_monomial}
Let $\theta_0=(a_0,b_0)_{y_0}$ and $\theta_1=(a_1,b_1)_{y_1}$ be two elements in $\Xi$ such that, $J(\theta_1)$ is a factor in one of the monomials of $J\mapsto \psi(J)(\theta_0)$. Then by definition, we immediately verify that necessarily $y_1$ belongs to the geodesic segment $[a_1,y_0]$.
\end{Rem}
\begin{Ex}
	In (\ref{eq-t2}) for $(c_0,...,c_l)\in\Xi_{[c,y]}$, with $c_0=c$, $c_l=b$ and $p(a,c)>0$,
	$$J_{c,y}(c_0,...,c_l)= J((c_0,c_1)_{x_{i_1}})\cdot J((c_1,c_2)_{x_{i_2}})\cdot\cdot\cdot J((c_{l-1},c_l)_{x_{i_l}}),$$
	with $(i_1,i_2,...,i_l)=\iota_{c,y}(c_0,...,c_l)$ the crossing indices associated with the crossing vertices $(c_0,...,c_l)$ along $[c,y]=(x_0,...,x_m)$. Then for all $j\in\{1,...,l\}$, $J((c_{j-1},c_{j})_{x_{i_j}})$ is a factor in one of the monomials of $J\mapsto\psi(J)((a,b)_y)$.
\end{Ex}

\begin{Def}\textit{The graph $\mathcal{G}_\Xi$, or dependency digraph}:\label{Def_of_G_Xi}\index{$\mathcal{G}_\Xi$}

	Let $\mathcal{G}_\Xi=(\Xi,E(\mathcal{G}_\Xi),s,b)$ be the simple directed graph whose set of vertices is $\Xi$, whose set of edges $E(\mathcal{G}_\Xi)$ is a subset of $\Xi\times \Xi$ defined below, and with $s$ and $b$ being respectively the projection on the first coordinate and the second coordinate.
	
	$E(\mathcal{G}_\Xi)$ is the set of pairs $(\theta_1,\theta_0)\in\Xi\times\Xi$ such that $J(\theta_1)$ is a factor in one of the monomials of $\psi(J)(\theta_0)$. 
\end{Def}

Notice that since $\psi:\mathcal{F}(\Xi,\mathbb{C})\to\mathcal{F}(\Xi,\mathbb{C})$ is $\Gamma$-equivariant, $\Gamma$ acts naturally on the graph $\mathcal{G}_\Xi$. We thus can consider the quotient graph $\Gamma\backslash\mathcal{G}_{\Xi}$, which we will note $\mathcal{V}$:
\begin{Def}\textit{Definition of $\mathcal{V}$}\label{Definition_of_V}\index{$\mathcal{V}$}
	
	$\mathcal{V}=(V,E(\mathcal{V}),s,b)$ is the directed graph whose set of vertices is $\Gamma\backslash\Xi$ and whose set of edges $E(\mathcal{V})$ is the set of edge classes from $\Gamma\backslash E(\mathcal{G}_\Xi)$. The source map $s$, and the aim map $b$ are the maps that sends an edge $[(\theta_1,\theta_0)]\in E(\mathcal{V})$ to respectively $\Gamma\theta_1\in V$, and $\Gamma \theta_0\in V$, where $(\theta_1,\theta_0)$ is some representative of $[(\theta_1,\theta_0)]$. 
\end{Def}
\begin{Rem}
	By definition, an edge exists from a vertex $[\theta_1]$ of $\mathcal{V}$ to another vertex $[\theta_0]$ of $\mathcal{V}$ if, and only if, there are representatives $(a_1,b_1)_{y_1}\in [\theta_1] $ and $(a_0,b_0)_{y_0}\in [\theta_0] $ such that $J((a_1,b_1)_{y_1})$ is a factor in one of the monomials of $J\mapsto\psi(J)((a_0,b_0)_{y_0})$.
\end{Rem}

\begin{Ex*}
	See examples in the Appendix \ref{Section_Examples}.
\end{Ex*}

 We enlarge definition \ref{Def_monomials_first_order} of factors of a monomials:
\begin{Def}\textit{Factors of monomials of $\psi$ (enlarged version)}\label{Def_Factors_Monomial_Enlarged}\index{Factor of monomial}

	Let $n\geq 1$ be an integer. For $\theta_0=(a_0,b_0)_{y_0}$ and $\theta_n=(a_n,b_n)_{y_n}$ two elements of $\Xi$, we will say that $J(\theta_n)$ \textit{is a factor in one of the monomials of} $J\mapsto\psi^{\circ n}(J)(\theta_0)$  if there exists a path in the graph $\mathcal{G}_\Xi$ of length $n$, with source $\theta_n$ and aim $\theta_0$.
\end{Def}

By their definition, the graphs $\mathcal{G}_\Xi$ and $\mathcal{V}$ carry informations about the monomials involved in the iterations $(\psi^{\circ n})_n$.
These graphs are not presented in \cite{Lalley_1993}, but underlie Steven P.Lalley's construction. In his paper, S. P. Lalley showed under the hypothesis (\ref{premiere_hypothese_de_SPLalley}) below, that the graph $\mathcal{V}$ is connected (Proposition 2.7 in \cite{Lalley_1993}).
\begin{equation}\label{premiere_hypothese_de_SPLalley}
	\forall (x,y)\in X_1,\; p(x,y)>0.
\end{equation}

For an irreducible Markov chain on $X$, it is not systematically true that the directed graph $\mathcal{V}$ is connected (see the example in the Appendix \ref{Section_Examples}). Nevertheless, it can be seen from these examples that $\mathcal{V}$ has a unique connected component $V_\infty$, such that for any vertex of the graph, there exists a path in $\mathcal{V}$ from the given vertex to this connected component $V_\infty$. Among other things, we will prove this observation in Proposition \ref{Prop_V_infini_is_absorbent}. Such structure have also been observed in \cite{Woods_1997}, he called it a "\textit{strongly connected component}".

\subsection{Flooded Cavern Tree}

Recall the definition of a directed graph, and the definition of a directed geodesic path in a directed graph (Definition \ref{Def_Directed_Graph}).
\begin{Def}\textit{Graph Morphism}\index{Graph morphism}

Let $\mathcal{G}_1=(V_1,E_1,s_1,b_1)$ and $\mathcal{G}_2=(V_2,E_2,s_2,b_2)$ be two directed graphs. A \textit{graph morphism} from $\mathcal{G}_1$ to $\mathcal{G}_2$ is the data of two maps $\varrho_V:V_1\to V_2$ (resp. $\varrho_E:E_1\to E_2$) from the set of vertices (resp. edges) of $\mathcal{G}_1$ to the set of vertices (resp. edges) of $\mathcal{G}_2$, such that:
$$\varrho_V\circ s_1=s_2\circ \varrho_E \text{ and } \varrho_V\circ b_1=b_2\circ \varrho_E.$$

Note that for simple graphs the map $\varrho_E$ is entirely determined by $\varrho_V$. In practice, for simple graphs, we will only mention of $\varrho_V$.
\end{Def}
\begin{Def}[Simple directed geodesic paths and cycle of a directed graph]~
\index{Cycle}
	A directed geodesic path in a directed graph is called a \textit{cycle} if its source equals its aim.
	
	A directed geodesic path $(e_1,...,e_n)$ in a directed graph $\mathcal{G}=(V,E,s,b)$ is said to be \textit{simple} if the source $s(e_1), ..., s(e_n)$ of each edges $e_1,...,e_n$ of the directed geodesic path are distinct\footnote{This differs from the standard definition, in which we only ask for edges to be distinct. The definitions coincide when the graph itself, is simple.}.
\end{Def}
\begin{Def}\textit{(Simple) Descending Rooted Tree \label{Directed_Rooted_Tree}(See figure \ref{Fig_descending_rooted_tree})}\index{Descending rooted tree}

	A \textit{descending rooted graph} is a directed graph $\mathcal{T}$ that does not possesses any simple cycle of positive length, and that admits a distinguished vertex $r$, called \textit{root} of the tree, such that any simple path in the tree can be extended in a simple path with aim the root $r$.
	
	Such descending rooted graph is said to be a \textit{simple descending rooted tree} if there is only one way to extend a simple path in the graph, to a simple path with aim $r$. Equivalently if there cannot be two distinct directed geodesic paths with same source and same aim.
\end{Def}

	\begin{figure}
		\centering
		\begin{subfigure}[b]{0.45\textwidth}
            \centering
            \includegraphics[width=\textwidth]{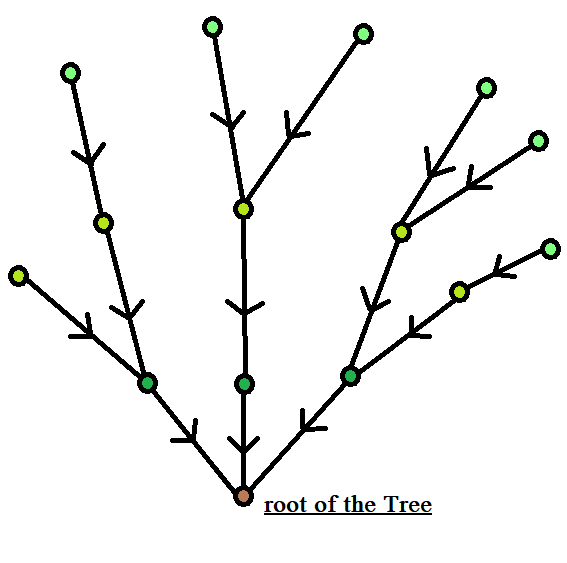}
            \caption{A typical simple descending rooted tree.}\label{Fig_descending_rooted_treea}
        \end{subfigure}
     \hfill
        \begin{subfigure}[b]{0.45\textwidth}
            \centering
            \includegraphics[width=\textwidth]{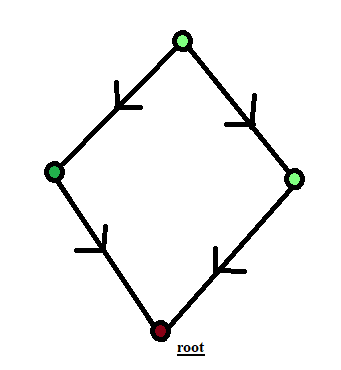}
            \caption{A descending rooted graph that is not a simple descending rooted tree.}\label{Fig_descending_rooted_tree_not_simple}
        \end{subfigure}
        \caption{Some Descending rooted graph.}
        \label{Fig_descending_rooted_tree}
	\end{figure}

We give one example of such simple descending rooted tree:
\begin{Def}\textit{Inclusion Graph associated with a function $g$ (Flooded Cavern Tree)}\label{DEF:FLOODED_CAVERN_TREE}\index{Flooded cavern tree}

Let $i<j$ be two integers, and let $g:\{i,...,j\}\to\mathbb{N}$ be a $\mathbb{N}$-valued function defined on the interval $\{i,...,j\} \subset \mathbb{N}$ such that for any integer $t$ such that $i\leq t < j$, we have:
$$ g(t)\geq g(i)>g(j).$$

We define $\mathcal{T}_g$ to be the graph whose set of vertices $\mathcal{I}$ is defined as the collection of all subintervals $I=\{a,...,b\}$ of $\{i,...,j\}$ such that $\operatorname{Card}(I)\geq 2$ and 
$$ \forall a\leq t <b,\, g(t)\geq g(a) > g(b). $$
See Figure \ref{figure the submerged mountain}. There exists an edge with source $I_1\in\mathcal{I}$ and aim $I_2\in\mathcal{I}$ if $I_1\varsubsetneq I_2$ and there does not exists any subset $I_3\in \mathcal{I}$ such that $I_1\varsubsetneq I_3 \varsubsetneq I_2$. Besides there is at most one edge between two given vertices, so the digraph is simple.
\end{Def}
	
\begin{figure}
	\centering
	\includegraphics[scale=0.4]{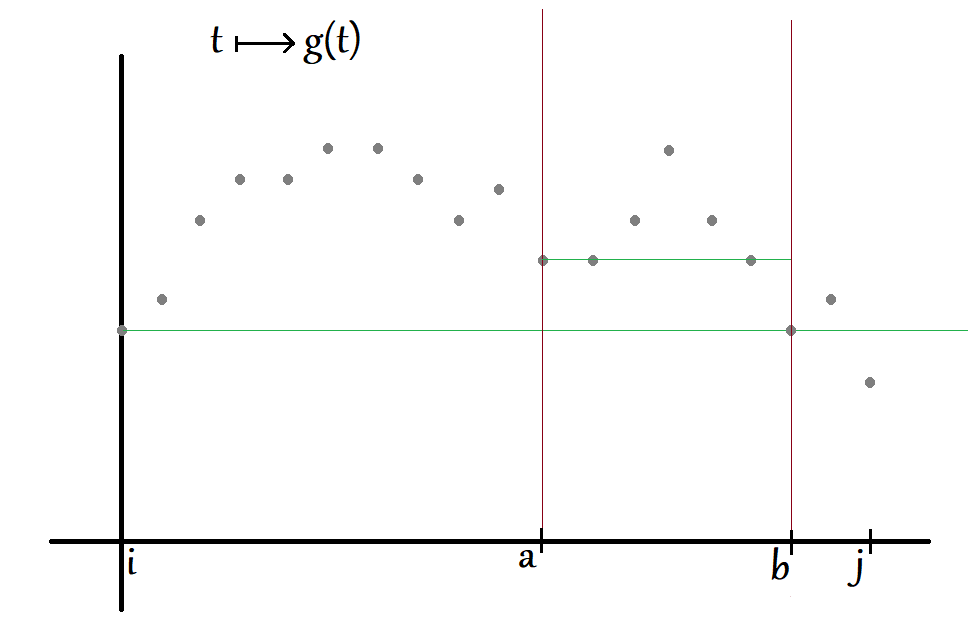}
	\caption{A vertex $[a,b]$ of the Flooded Cavern Tree $\mathcal{T}_g$}\label{figure the submerged mountain}
\end{figure}

\begin{Prop}\label{PROP:FLOODED_CAVERN_TREE_PROPERTY}
We use the notations of the previous definition \ref{DEF:FLOODED_CAVERN_TREE}:

\begin{enumerate}[label=\roman*)]
    \item For any two vertices $I_1,I_2$ in $\mathcal{I}$, there exists a directed path from $I_1$ to $I_2$ if, and only if $I_1\varsubsetneq I_2$;
    \item The leaves of $\mathcal{T}_g$ (i.e vertices such that there does not exist any edge with aim these vertices) are the intervals $I\in\mathcal{I}$ of cardinal two. Besides it corresponds to intervals $I=(t,t+1)$ with $t\in\{i,...,j-1\}$ such that $g(t)>g(t+1)$;
    \item $\mathcal{T}_g$ is a simple descending rooted tree with root $\{i,...,j\}$;
    \item For any two distinct intervals $I_1,I_2\in\mathcal{I}$, if $I_1\cap I_2\neq \emptyset$ then either $I_1\varsubsetneq I_2$, or $I_2\varsubsetneq I_1$, or $\operatorname{Card}(I_1\cap I_2)=1$ and in this last case there exists $I_3\in \mathcal{I}$ such that there are two edges with sources respectively $I_1$ and $I_2$ and with common aim $I_3$.
    \item For any interval $I=\{a,...,b\}$ in $\mathcal{I}$ of cardinal at least $2$, there exists a finite sequence of non-decreasing integers $a+1<i_1<...<i_l=b$ such that the intervals $I_1=\{a+1,...,i_1\}, I_2=\{i_1,...,i_2\}$, ..., $I_l=\{i_{l-1},...,i_l\}$ are in $\mathcal{I}$. Besides these intervals are all possible sources of edge, with aim the interval $I\in\mathcal{I}$.
    \item For any $I\in\mathcal{I}$, let $g_I$ denote the restriction of $g$ to the subinterval $I$. The simple descending rooted tree $\mathcal{T}_{g_I}$ identifies as a subgraph of $\mathcal{T}_g$ with root $I$.
\end{enumerate}
\end{Prop}
\begin{proof}
The three first points $i)$, $ii)$ and $iii)$ as the point $vi)$ are immediate consequences of the definition.

For the point $iv)$, let $I_1=\{i_1,...,j_1\}$ and $I_2=\{i_2,...,j_2\}$ be two intervals in $\mathcal{I}$ with non-empty intersection, such that $I_1\not\subseteq I_2$ and $I_2\not\subseteq I_1$. Without loss of generality, suppose 
$i_1<i_2\leq j_1 < j_2$. Then by definition we have
$$ g(i_2)\geq g(j_1) \text{,  and  } g(j_1)\geq g(i_2)>g(j_2),$$
with equality in the first inequality if, and only if, $i_2=j_1$. We thus deduce that $i_2=j_1$ and $iv)$ is proven.

For the point $v)$, let $a+1=:i_0<i_1<...<i_l=b$ be a finite non-decreasing sequence of integers recursively defined as: 

For any $s\geq 0 $, if $i_s\neq b$ let $i_{s+1}:=\min \{t > i_s : g(t) < g(i_s) \}$. By definition we have that for any $s\in\{1,...,l\}$, the interval $I_s=\{i_{s-1},...,i_s\}$ belongs to $\mathcal{I}$. Now let $s\in\{1,...,l\}$, and let $I'$ be an interval in $\mathcal{I}$ such that $I_s\varsubsetneq I'\subseteq I$. We prove that $I'$ necessarily equals $I$ so that, there exists an edge with source $I_s$ and aim $I$.
If $s=1$, then necessarily $a$ belongs to $I'$, and we thus have $$\sup(I')\geq \min \{t >a : g(a)>g(t)\}=b.$$ Hence $I\subseteq I'$. 

If $s>1$, then $I_{s-1}\cap I'\neq \emptyset$ and is of cardinal at least $2$, for if not then $\min(I')=i_{s-1}$, implying, $\max(I')=i_s$. Thus we get $I_{s-1}\varsubsetneq I'$. An immediate induction gives $I_1\subset I'$, hence $I'=I$ from the case $s=1$.
\end{proof}

\begin{Rem}
We can imagine a cavern whose ceiling height is given by the function $g:\{i, \dots, j\}\to\mathbb{N}$, into which we inject water from the point $(i+1, g(i+1))$. At this point, there is also an air/water evacuation system (in Figure \ref{Figure_of_the_ecluse}, these points are represented as small colored circles at each iteration of the process). Once a sufficient amount of water has been injected, the water level reaches an equilibrium state.

The water level varies at each index involved in the decomposition of intervals in $\mathcal{I}$, presented in assertion $v)$ of Proposition \ref{PROP:FLOODED_CAVERN_TREE_PROPERTY}. More precisely, the different water levels divide the cavern into several sections. The indices of these sections correspond to the pairs \og$(i_{s-1}, i_{s})$\fg{}, for $s\in\{1,...,l\}$, of integers defining the bounds of the intervals in the decomposition described in assertion $v)$. In Figure \ref{Figure_of_the_ecluse}, these indices are represented as small colored squares.

These sections can then be seen as new caverns, on which we can repeat the process until exhaustion.
\end{Rem}
	

	\begin{figure}
		\centering
		\includegraphics[width=\textwidth]{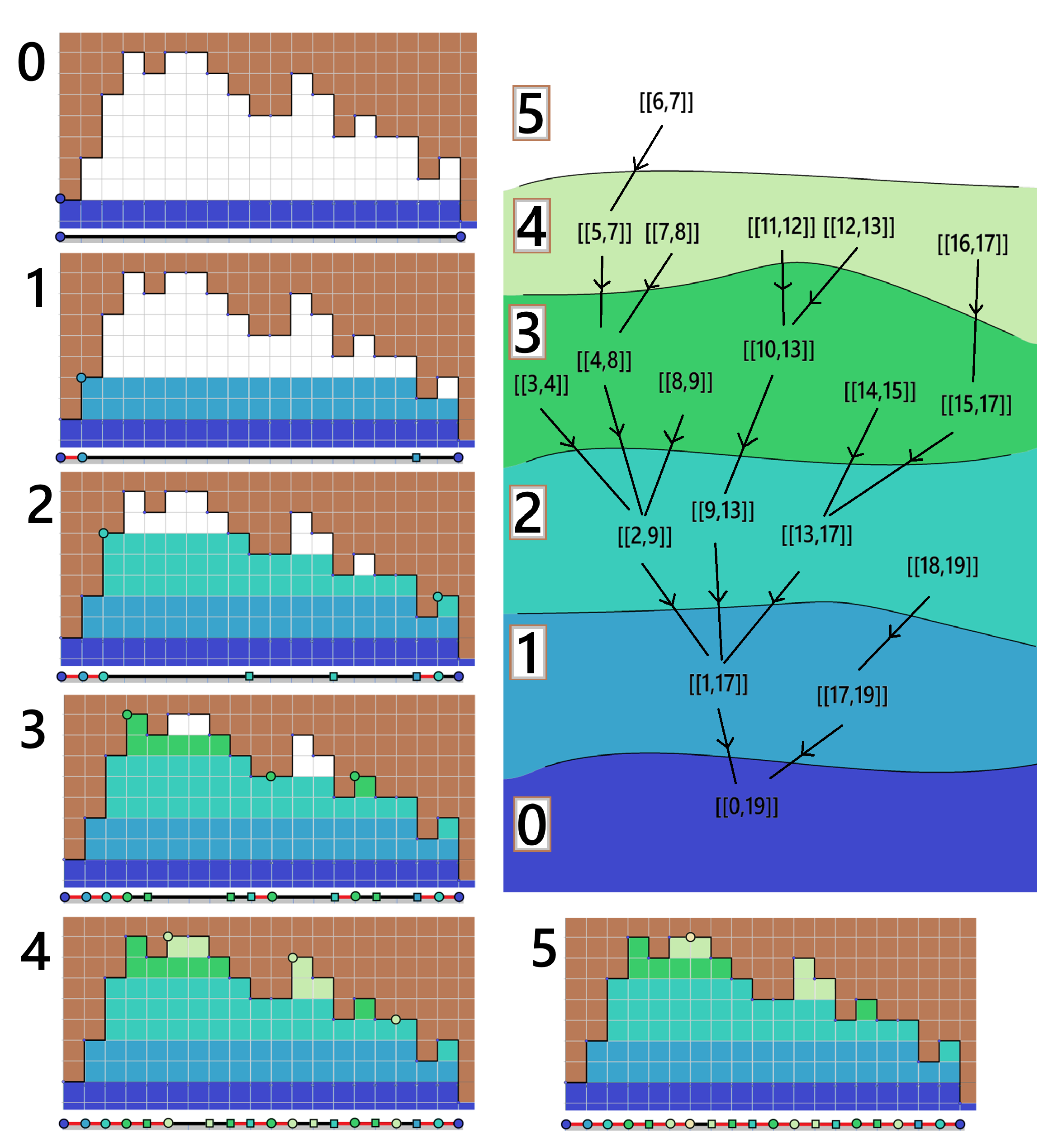}
		\caption{Illustration of the cavern lemma. }\label{Figure_of_the_ecluse}
	\end{figure}

\subsection{Characterisation of paths in \texorpdfstring{$\mathcal{G}_\Xi$}{Graphe GXi} and \texorpdfstring{$\mathcal{V}$}{Graphe V}}
\begin{Def}\textit{The tree $\mathcal{T}_\gamma$}\label{Def:The_T_gammas}

Let $(a,b)_y$ be an element in $\Xi$ and let $\gamma=(\omega_0,...,\omega_n)$ be a $p$-admissible path in $\mathcal{P}h(a,b;\mathcal{B}(y)^\complement)$. Consider the integer-valued function $g:\{0,...,n\}\to \mathbb{N}$ defined as
$$ g:i\mapsto d(\omega_i,y). $$
The simple descending rooted tree $\mathcal{T}_\gamma$ is the graph $T_g$.This graph $\mathcal{T}_\gamma$ comes together with a graph morphism $\phi_\gamma:\mathcal{T}_\gamma\to\mathcal{G}_\Xi$ defined as follow:

For $I=\{i,...,j\}$ a vertex of $\mathcal{T}_\gamma$, let $y_I$ be the only vertex of the geodesic segment $[\omega_i,y]$ at distance $k+1$ from $\omega_i$, then we define the image of $I$ under the graph morphism $\phi_\gamma$ as 
$$\phi_\gamma(I)=(\omega_i,\omega_j)_{y_I}.$$
\end{Def}
\begin{Rem}
    If $I=\{i,...,j\}$ is a vertex of $\mathcal{T}_\gamma$, consider $\gamma'$ to be the $p$-admissible path $(\omega_i,...,\omega_j)\in\mathcal{P}h(a',b';\mathcal{B}(y')^\complement)$, with $(a',b')_{y'}=\phi_\gamma(I)$. Then from assertion $vi)$ of Proposition \ref{PROP:FLOODED_CAVERN_TREE_PROPERTY}, the simple descending rooted tree $\mathcal{T}_{\gamma'}$ identifies as a subtree of $\mathcal{T}_\gamma$:
    $$\mathcal{T}_{\gamma'} \hookrightarrow \mathcal{T}_\gamma.$$
\end{Rem}
\begin{prop}\label{Prop_Phi_gamma_is_well_defined}
With the notations of the above definition \ref{Def:The_T_gammas}, the graph morphism $\phi_\gamma:\mathcal{T}_\gamma\to\mathcal{G}_\Xi$ is well defined. I.e if $I$ is a vertex in $\mathcal{T}_\gamma$, then $\phi_\gamma(I)$ is in $\Xi$, and if $I_1\to I_2$ represents an edge of $\mathcal{T}_\gamma$, then $\phi_\gamma(I_1)\to \phi_\gamma(I_2)$ is an edge in $\mathcal{G}_\Xi$.
\end{prop}
\begin{proof}
Firstly let $I=\{i,...,j\}\in\mathcal{I}$ be a vertex of $\mathcal{T}_\gamma$ and let $y_I\in X_0$ be a vertex of $X$ defined as in Definition \ref{Def:The_T_gammas}. Using the Lemma \ref{Asymp-Shadow_Lemma} in \cite{Asymp} and by consideration of the distance we get:
$$\forall i\leq t < j , \, d(\omega_t,y_I)\geq d(\omega_i,y_I)=k+1\,>\, d(\omega_j,y_I).$$
Hence the $p$-admissible path $\gamma'=(\omega_i,...,\omega_j)$ belongs to $\mathcal{P}h(\omega_i,\omega_j;\mathcal{B}(y_I)^\complement)$, and $\phi_\gamma=(\omega_i,\omega_j)_{y_I}$ is in $\Xi$.

All it remains to prove is that edges in $\mathcal{T}_\gamma$ give edges in $\mathcal{G}_\Xi$, through the graph morphism $\phi_\gamma$.

Let $I=\{i,...,j\}\in\mathcal{I}$ be a vertex of $\mathcal{T}_\gamma$, that is not a leaf (that is such that $\operatorname{Card}(I)\geq 2$). Denote by $(a,b)_{y}=\phi_\gamma(I)$, the image of the interval $I$ through the morphism $\phi_\gamma$ and consider the path decomposition of the $p$-admissible path $(\omega_{i+1},...,\omega_{j})$ along the geodesic segment $[\omega_{i+1},y]$ with respect to $\mathcal{B}$ (Proposition \ref{Asymp-Prop_decomposition_of_paths_on_a_tree} and Corollary \ref{Asymp-corollary_of_decomposition_of_paths_on_a_tree}) in \cite{Asymp}, about admissible paths decompositions in a tree:
$$\gamma_1\ast...\ast\gamma_l=(\omega_{i+1},...,\omega_j).$$ 
If we write $\gamma_1=(\omega_{i_0},...,\omega_{i_1})$, ...,  $\gamma_l=(\omega_{i_{l-1}}, ..., \omega_{i_l})$ then $i_0=i+1<i_1<...<i_l=j$, and the tuple of crossing vertices associated with the path decomposition is given by the $(l+1)-$tuple:
$(\omega_{i_0},...,\omega_{i_l})\in\Xi_{[\omega_{i_0},y]}$, we also have $p(\omega_i,\omega_{i_0})>0$.

Considering the distance at $y$, we get that the intervals $I_1=\{i_0,..., i_1\}$, ..., $I_l=\{i_{l-1},..., i_l\}$ are all the intervals in $\mathcal{I}$ involved in assertion $v)$ from Proposition \ref{PROP:FLOODED_CAVERN_TREE_PROPERTY}, and are in particular all the possible source of edges with aim $I$ in $\mathcal{T}_\gamma$.

Thus by definition of the graph morphism $\phi_\gamma$ and by definition of the edges of the simple directed graph $\mathcal{G}_\Xi$ (Definition \ref{Def_of_G_Xi}), we have that for any $s\in\{1,...,l\}$, there is an edge in $\mathcal{G}_\Xi$ with source $\phi_\gamma(I_s)$ and aim $\phi_\gamma(I)$. 
\end{proof}
From the above proof we in particular get the Proposition \ref{Proposition_of_the_ecluse} below:

\begin{Prop}\label{Proposition_of_the_ecluse} ~
	Let $(a,b)_{x,y}$ be in $\Xi$, let $\gamma$ be a $p$-admissible path in $\mathcal{P}h(a,b;\mathcal{B}(y)^\complement)$ and $\mathcal{T}_\gamma$ be the flooded cavern tree defined in Definition \ref{Def:The_T_gammas}.

	For $m\in\mathbb{N}$, let $\Theta_m=\Theta_m(\gamma)$ be the set of sources of path of length $m$ in $\mathcal{T}_\gamma$, whose aim is the root\footnote{Note that the image of this root under $\phi_\gamma$ is $(a,b)_y$.} of $\mathcal{T}_\gamma$. 
	
	Then for any integer $m\geq 1$ such that $\Theta_m\neq\emptyset$, the product
	$$\prod_{I\in\Theta_m} J(\phi_\gamma(I))$$
	is a monomial of $J\mapsto\psi^{\circ m}(J)((a,b)_{y})$ (See Definition \ref{Asymp-Def_monomial} in \cite{Asymp}).
\end{Prop}
\begin{proof}
We prove this result by induction over $m\geq 1$.
Denote by $(\omega_0,...,\omega_n)=\gamma$ the vertices of $\gamma$ and $I_\gamma=\{0,...,n\}$ the root of $\mathcal{T}_\gamma$. Denote by $I_1,...,I_l$ the elements of $\Theta_1(\gamma_I)$. It correspond to all the intervals involved in assertion $v)$ of Proposition \ref{PROP:FLOODED_CAVERN_TREE_PROPERTY}. From the above proof we have that the product 
$$\prod_{j=1}^l J(\phi_\gamma(I_j)),$$
is a monomial of $J\mapsto \psi(J)((a,b)_y)$.

Now fix $m\geq 1$ be such that the result is proven for any rank $\leq m$. We prove it for rank $m+1$.

Let $I_1,...,I_\kappa$ be all the intervals in $\Theta_{m}$. By the induction hypothesis, there exists a positive constant $C_m>0$, and a polynomial function with non-negative coefficients $Q$, such that for any vector $J\in \mathcal{F}(\Xi,\mathbb{C})$, we have
\begin{equation*}
 \psi^{\circ m}(J)((a,b)_y)=C J(\phi_\gamma(I_1))\cdots J(\phi_\gamma(I_\kappa)) + Q(J),
\end{equation*}
thus:
\begin{equation}\label{eq:3666}
 \psi^{\circ m+1}(J)((a,b)_y)=C \psi(J)(\phi_\gamma(I_1))\cdots \psi(J)(\phi_\gamma(I_\kappa)) + Q(\psi(J)).
\end{equation}
Besides for any integers $j\in\{1,...,\kappa\}$, there exists constants $C_j>0$ and a polynomial function with non-negative coefficients $Q_j$.

\begin{equation}\label{eq:3670}
\psi(J)(\phi_\gamma(I_j))=C_j\left(\prod_{I'\in\Theta_1(\gamma_{I_j})} J(\phi_\gamma(I')) \right) + Q_j(J).
\end{equation}

So that we get, by injecting (\ref{eq:3670}) into (\ref{eq:3666}), the existence of a non-negative polynomial function $Q$ and a positive constant $C>0$ such that for any $J\in\mathcal{F}(\Xi,\mathbb{C})$, we have
$$ \psi^{\circ (m+1)}(J)= C \prod_{i=1}^\kappa \left(\prod_{I'\in\Theta_1(\gamma_{I_j})} J(\phi_\gamma(I'))\right) + Q(J). $$ 

Thus the finite product $\prod_{i=1}^\kappa \left(\prod_{I'\in\Theta_1(\gamma_{I_j})} J(\phi_\gamma(I'))\right)$ equals the product $\left(\prod_{I'\in\Theta_{m+1}(\gamma)} J(\phi_\gamma(I'))\right)$, thanks to assertion $vi)$ of Proposition \ref{PROP:FLOODED_CAVERN_TREE_PROPERTY}, and is a monomial of $J\mapsto \psi^{\circ (m+1)}(J)$.
\end{proof}
 \begin{prop}\label{Prop_G_Xi_is_limit_of_T_gamma}
Let $(a,b)_y$ and $(a',b')_{y'}$ in $\Xi$ be two vertices of $\mathcal{G}_\Xi$. Any directed path in $\mathcal{G}_\Xi$ from $(a',b')_{y'}$ to $(a,b)_y$ is the image, under the graph morphism $\phi_\gamma$, of a path in $\mathcal{T}_\gamma$, with aim the root of $\mathcal{T}_\gamma$, for some admissible path $\gamma\in\mathcal{P}h(a,b;\mathcal{B}(y)^\complement)$.
 \end{prop}
 \begin{proof}
    We prove this theorem by induction on the directed path length.
    
    For $m=1$, consider an edge $(a',b')_{y'} \to (a,b)_y$ in $\mathcal{G}_\Xi$ with source $(a',b')_{y'}$ and aim $(a,b)_y$. By definition of the edges of $\mathcal{G}_\Xi$, let $c\in\mathcal{B}(y)^\complement$ be such that $p(a,c)>0$ and $(c_0,...,c_l)\in\Xi_{[c,y]}$ be such that $c_0=c$, $c_l=b$ and for some integer $j\in \{1,...,l\}$, we have 
    $$ (c_{j-1},c_{j})_{x_{i_j}} = (a',b')_{y'}, $$
    where we have denoted $(x_0,...,x_m)$ the geodesic segment $[c,y]$, and $(i_1,...,i_l)=\iota_{[c,y]}(c_0,...,c_l)$ the associated crossing indices (See Lemma \ref{Asymp-Lem_2.6} in \cite{Asymp}).
    For any $s\in\{1,...,l\}$ such that $s\neq j$, set $\gamma_s$ to be an arbitrary $p-$admissible path in $\mathcal{P}h(c_{s-1},c_s;\mathcal{B}(x_{i_s})^\complement)$, then for any $p$-admissible path $\gamma_j\in\mathcal{P}h(a',b';\mathcal{B}(y')^\complement)$, from Corollary \ref{Asymp-corollary_of_decomposition_of_paths_on_a_tree} in \cite{Asymp}, the $p-$admissible path 
    \begin{equation}\label{eq:3696}
        \gamma'=\gamma_1\ast\cdots\ast\gamma_{l}
    \end{equation}
    is in $\mathcal{P}h(c,b;\mathcal{B}(y)^\complement)$, and the above formula (\ref{eq:3696}) is the path decomposition of $\gamma'$ along the geodesic segment $[c,y]$ with respect to $\mathcal{B}$.
    
    In particular if we denote by $\gamma_a\in\mathcal{P}h(a,a';\mathcal{B}(y)^\complement)$ and $\gamma_b'\in\mathcal{P}h(b',b;\mathcal{B}(y)^\complement)$ the $p$-admissible paths:
    \begin{equation*}
        \gamma_a:=(a,c)\ast\gamma_1\ast\cdots\ast\gamma_{j-1} \quad \text{and}\quad \gamma_b':=\gamma_{j+1}\ast\cdots\ast\gamma_j,
    \end{equation*}
    then for any $p$-admissible path $\gamma^{(1)}\in\mathcal{P}h(a',b';\mathcal{B}(y)^\complement)$ and for the $p$-admissible path $\gamma$ defined as
    \begin{equation}\label{eq:3706}
    \gamma:=\gamma_a \ast\gamma^{(1)}\ast\gamma_b',
    \end{equation}
    the simple descending rooted tree $\mathcal{T}_{\gamma^{(1)}}$ identifies as a subgraph of $\mathcal{T}_\gamma$ (See the sixth assertion in Proposition \ref{PROP:FLOODED_CAVERN_TREE_PROPERTY}) and the edge with source the root of the subgraph $\mathcal{T}_{\gamma^{(1)}}$, in $\mathcal{T}_\gamma$, has its aim equal to the root of $\mathcal{T}_\gamma$. Besides the image of this edge through the graph morphism $\phi_\gamma$ is the edge $(a',b')_{y'}\to (a,b)_y$ in $\mathcal{G}_\Xi$, with source $(a',b')_{y'}$ and aim $(a,b)_{y}$.
    
    Suppose the result verified for $m\geq 1$. Let
    $(a_{m+1},b_{m+1})_{y_{m+1}}\to (a_m,b_m)_{y_m} \to \dots \to  (a_0,b_0)_{y_0}$ represents a directed path in $\mathcal{G}_\Xi$. By induction hypothesis, there exists a $p$-admissible path $\gamma^{(1)}\in\mathcal{P}h(a_1,b_1;\mathcal{B}(y_1)^\complement)$ such that, the directed path $(a_{m+1},b_{m+1})_{y_{m+1}}\to\dots\to (a_1,b_1)_{y_1}$ is the image under the graph morphism $\phi_{\gamma^{(1)}}$, of some directed path in $\mathcal{T}_{\gamma^{(1)}}$ with aim the root of $\mathcal{T}_{\gamma^{(1)}}$.
    
    Using the notation of the previous case $m=1$,  taking $(a',b')_{y'}=(a_1,b_1)_{y_1}$, $(a,b)_y=(a_0,b_0)_{y_0}$ and the $p$-admissible path $\gamma^{(1)}\in\mathcal{P}h(a',b';\mathcal{B}(y')^\complement)$ from the previous paragraph, we obtain the result with $\gamma\in\mathcal{P}h(a,b;\mathcal{B}(y)^\complement)$ defined as in (\ref{eq:3706}).
 \end{proof}
 \begin{Rem}
    In particular the above implies that the graph $\mathcal{G}_\Xi$ is the inductive limit of the simple descending rooted trees $\mathcal{T}_\gamma$, where $\gamma$ runs over all $p$-admissible paths in $\bigcup_{(a,b)_y\in\Xi} \mathcal{P}h(a,b;\mathcal{B}(y)^\complement)$.
 \end{Rem}
\vspace{\baselineskip}

Along this section, we actually proved:
\begin{Cor}\label{Lemma_of_the_ecluse}~

	Let $(a',b')_{y'}$ and $(a,b)_{y}$ be two vertices of $\mathcal{G}_\Xi$. The following assertions are equivalent:
	\begin{enumerate}[label=\roman*)]
	\item There exists a directed geodesic path in $\mathcal{G}_\Xi$ from $(a,b)_y$ to $(a',b')_{y'}$;
	\item The vertex  $y'$ is in the geodesic segment $[a',y]$ and there exist two $p$-admissible paths $\gamma_1\in\mathcal{P}h(a,a';\mathcal{B}(y)^\complement)$, $\gamma_2\in\mathcal{P}h(b',b;\mathcal{B}(y)^\complement)$ and there exists a $p$-admissible path $\gamma'\in\mathcal{P}h(a',b';\mathcal{B}(y')^\complement)$, such that
	$$\gamma_1\ast\gamma'\ast\gamma_2\in \mathcal{P}h(a,b;\mathcal{B}(y)^\complement);$$
	\item  The vertex  $y'$ is in the geodesic segment $[a',y]$ and there exist two $p$-admissible paths $\gamma_1\in\mathcal{P}h(a,a';\mathcal{B}(y)^\complement)$, $\gamma_2\in\mathcal{P}h(b',b;\mathcal{B}(y)^\complement)$ such that, for all $p$-admissible path $\gamma'\in\mathcal{P}h(a',b';\mathcal{B}(y')^\complement)$, we have
	$$\gamma_1\ast\gamma'\ast\gamma_2\in \mathcal{P}h(a,b;\mathcal{B}(y)^\complement).$$
	\end{enumerate}
\end{Cor}
\begin{Rem}\label{Rem_10.10}
	In the assertions $ii)$ and $iii)$ the information $\gamma_1\ast\gamma'\ast\gamma_2\in \mathcal{P}h(a,b;\mathcal{B}(y)^\complement)$ is implied by $y'\in [a',y]$ and the membership of $\gamma'$ to $\mathcal{P}h(a',b',\mathcal{B}(y')^\complement)$. Indeed, if $y'\in[a',y]$ then the vertices of the $p$-admissible path $\gamma'$ lie in the shadow $\Omega_{y'}^y$, of $y'$ under $y$. We therefore deduce that $\gamma' \in\mathcal{P}h(a',b';\mathcal{B}(y')^\complement\cap\mathcal{B}(y)^\complement)$, and so the $p$-admissible path $\gamma_1\ast\gamma'\ast\gamma_2$ is in $\mathcal{P}h(a,b;\mathcal{B}(y)^\complement)$. See also remark \ref{Rem_post_def_Factors_monomial}.
\end{Rem}
\begin{proof}[Proof of Corollary \ref{Lemma_of_the_ecluse}]

\og $i)\Rightarrow iii)$ \fg{} : 
    This was proven in the induction used to prove \ref{Prop_G_Xi_is_limit_of_T_gamma}, see (\ref{eq:3706}).
\og $iii)\Rightarrow ii)$ \fg{}: Since $\mathcal{P}h(a',b';\mathcal{B}(y')^\complement)$ is non-empty, we immediately get the implication. 
\og $ii)\Rightarrow i)$\fg{}: Take $\gamma:=\gamma_1\ast\gamma'\ast\gamma_2\in\mathcal{P}h(a,b;\mathcal{B}(y)^\complement)$, as in the assertion $ii)$. By considering the distance from $y$, we get that $\mathcal{T}_{\gamma'}$ identifies as a subtree of $\mathcal{T}_\gamma$. The result then follows from the good definition of the graph morphism $\phi_\gamma:\mathcal{T}_\gamma\to\mathcal{G}_\Xi$ (Proposition \ref{Prop_Phi_gamma_is_well_defined}).
\end{proof}
\begin{Notation}\label{notation_post_lemma_ecluse}
	The remark \ref{Rem_post_def_Factors_monomial} coupled with Lemma \ref{Asymp-Shadow_Lemma} from \cite{Asymp} tells us that if there exists a path with source $(a',b')_{y'}$ and aim $(a,b)_{y}$ in $\mathcal{G}_\Xi$, then the shadow $\Omega_{x'}^{y'}$ is included in the shadow $\Omega_x^y$, where $x$ (resp. $x'$) is the only neighbour of $y$ (resp. $y'$) in the geodesic segment $[a,y]$ (resp. $[a',y']$).
	
	To keep this constraint in mind, we introduce the notation \og $(a,b)_{x,y}$ \fg{} where $(a,b)_y\in\Xi$ and $x$ is the only neighbour of $y$ in the geodesic segment $[a,y]$. we will then write \og $(a,b)_{x,y}\in\Xi$ \fg{} (See remark \ref{Asymp-Rem_Xi_y_as_a_union_of_Xi_x_y} in \cite{Asymp}). For concision's sake, we will retake the notation $(a,b)_y$ after this section.
\end{Notation}

\subsubsection{Another characterisation of paths in \texorpdfstring{$\mathcal{V}$}{Graphe V}}~

We formulate the analogue of Corollary \ref{Lemma_of_the_ecluse} in the quotient graph $\mathcal{V}$:
\begin{Cor}\label{Cor_de_l_ecluse_for_V}~
	Let $[\theta] $ and $[ \theta'] $ be two vertices of $\mathcal{V}$. There exists a path in $\mathcal{V}$ with source $[ \theta'] $ and aim $[\theta]$ if, and only if, there exist representatives $(a,b)_{x,y}$ of $[\theta]$ ( $\Gamma\cdot (a,b)_{x,y}=[\theta]$ ), and $(a',b')_{x',y'}$ of $[ \theta'] $, and $p$-admissible paths $\gamma_1\in\mathcal{P}h(a,a';\mathcal{B}(y)^\complement),\tilde{\gamma}\in\mathcal{P}h(a',b';\mathcal{B}(y')^\complement)$ and $\gamma_2\in\mathcal{P}h(b',b;\mathcal{B}(y)^\complement)$, such that 
	$$y'\in [a',y] \quad \text{ and } \quad \gamma_1\ast\tilde{\gamma}\ast\gamma_2\in \mathcal{P}h(a,b;\mathcal{B}(y)^\complement).$$
\end{Cor}
 Recall that $\psi:E\to E$ is a polynomial map over $E$ with non-negative coefficients and is, in particular, infinitely differentiable on $E\approx\mathbb{C}^{\Gamma\backslash\Xi}$. We have the following characterisation of directed geodesic paths in $\mathcal{V}$:

\begin{Prop}\label{Prop_Chemin_et_derivee}~

 For any non-negative integer $n$, and any vertices $\alpha,\beta$ of $\mathcal{V}$, we have

$i)$ There exists a path $\vartheta$ of length $n$ in $\mathcal{V}$ with source $\alpha$, and aim $\beta$;\\
if, and only if,

$ii)$ There exists $J\in E$ with non-negative coordinates such that $(D_J \psi)^{\circ n}(1_{\alpha})
 	(\beta)>0$,
 	where $D_J\psi$ denotes the differential  the polynomial map $\psi$ at $J$.\footnote{Here and in the following proof we provide $E$ with the basis $\{1_{[\theta]}\}_{[\theta]\in\Gamma\backslash\Xi}$, viewing $E$ has $\mathbb{C}^{dim(E)}$. Thereby, the operator $D_J\psi:T_J E\to T_{\psi(J)} E$, is considered has a linear operator from $\mathbb{C}^{dim(E)}$ to $\mathbb{C}^{dim(E)}$. This Proposition will be used in the next section to justify some Perron-irreducibility arguments (See Lemma \ref{Lem_7.12})}
 \end{Prop}
 \begin{proof} 
 	 The non-negativity of the coefficients of $\psi$ implies that, for any $J\geq 0$, the differential $D_J\psi$ is a non-negative operator (Cf. formula (\ref{Asymp-Def_de_psi}) in \cite{Asymp} defining $\psi$). For $\vartheta$ a path in $\mathcal{V}$, denote $\alpha=\alpha_\vartheta$ its source, $\beta=\beta_\vartheta$ its aim and $l(\vartheta)$ its length. We prove the result by induction on $l(\vartheta)=n\in\mathbb{N}$, by proving for each $n$, the implication $i)\Rightarrow ii)$ then its converse $ii) \Rightarrow i)$.
 	
	$\bullet$ For $n=0$, there is nothing to prove. We therefore start the induction at $n=1$:
	Let $\alpha,\beta$ be two vertices of $\mathcal{V}$. By hypothesis, for all $J\in E=\mathcal{F}(\Xi,\mathbb{C})^\Gamma\approx\mathbb{C}^{\Gamma\backslash\Xi}$,
 	$$\psi(J)(\beta)=J(\alpha)\cdot\phi(J) +\chi(J),$$ 
 	with $\phi=\phi_{\alpha,\beta}:E\to\mathbb{C}$ a polynomial function with non-negative coefficients, and $\chi=\chi_{\alpha,\beta}:E\to\mathbb{C}$ a polynomial function with non-negative coefficients, none of whose monomials admits $J(\alpha)$ as a factor. Moreover, $\phi$ is non-zero if, and only if $(\alpha,\beta)$ is an edge of $\mathcal{V}$.
	Therefore, by differentiating the previous equality, for all $J$ in $E$, and evaluating it at $\dot{J}=1_\alpha$, we get:
	$$D_J\psi(1_\alpha)(\beta)=\phi(J)+J(\alpha)\cdot D_J\phi(1_\alpha) + D_J\chi(1_\alpha),$$
	but for $J\in E$ that is strictly positive, since $\phi$ and $\chi$ have non-negative coefficients, and $\chi$ does not depend on the values taken by $J$ at $\alpha$:
	\begin{gather*}
		J(\alpha)\cdot D_J\phi(1_\alpha)\geq 0\\ 
		D_J\chi(1_\alpha)= 0\\
		\text{ and } \\
		\phi(J)\geq 0,\text{ is non-zero, iff }(\alpha,\beta) \text{ is an edge of }\mathcal{V};
	\end{gather*}
	Thus if $(\alpha,\beta)$ is an edge of $\mathcal{V}$, then $D_J\psi^1(1_\alpha)(\beta)>0$. Conversely, if $D_J\psi(1_\alpha)(\beta)>0$, then necessarily $\phi(J)$ (or $D_J\phi(1_\alpha)$ is non-zero, so $\phi(J)>\phi(0)=0$, thus $\phi(J)$) is non-zero, so $(\alpha,\beta)$ is an edge of $\mathcal{V}$.

	$\bullet$ Assuming the result is verified for paths of length $n\geq 1$, let us show that it is also the case for paths of length $n+1$. Let $\vartheta$ be a path of length $n+1$ in $\mathcal{V}$, again let $\alpha=\alpha_\vartheta$ be the source of the path $\vartheta$, and $\beta=\beta_\vartheta$ be its aim. Let $\beta'\in\Gamma\backslash\Xi$ be a vertex of $\mathcal{V}$ and $\vartheta'$ be another path, such that $\beta'=\beta_{\vartheta'}$ and $\vartheta$ is equal to the path $\vartheta'$ to which we have added the edge $(\beta',\beta_\vartheta)$ at the end, i.e $\vartheta=\vartheta'\ast (\beta',\beta)$.
	By definition, 
	$$(D_J\psi)^{n+1}(1_\alpha)(\beta)=\sum_{\delta\in\Gamma\backslash\Xi}\left((D_J\psi)^n(1_\alpha)(\delta)\right) D_J\psi(1_{\delta})(\beta).$$
	For $J\in E$ that is strictly positive and for all $\delta\in\Gamma\backslash\Xi$, by positivity of $D_J\psi$, we have
	$$(D_J\psi)^{n+1}(1_\alpha)(\beta)\geq \left((D_J\psi)^n(1_\alpha)(\delta)\right) D_J\psi(1_{\delta})(\beta),$$
	which, for $\delta=\beta '$ and by applying the induction hypothesis to the paths $\vartheta '$ and $(\beta ',\beta)$, is strictly positive. 

	Conversely, if $(D_J\psi)^{n+1}(1_\alpha)(\beta)$ is a strictly positive real number, then there exists 
	$\delta\in\Gamma\backslash\Xi$ such that $\left((D_J\psi)^n(1_\alpha)(\delta)\right) D_J\psi(1_{\delta})(\beta)$ is strictly positive. Thus by the induction hypothesis, $(\delta,\beta)$ is an edge of $\mathcal{V}$ and there exists $\vartheta$ a path in $\mathcal{V}$ of length $n$, with source $\alpha$ and aim $\delta$. The path $\vartheta\ast (\delta,\beta)$ is a path in $\mathcal{V}$ of length $n+1$, with source $\alpha$ and aim $\beta$. This is what we were looking for.
	The induction is complete;
\end{proof}
\begin{Rem}
	Note that in the above proof we used a vector $J$ with positive coordinates, but the non-negativity of the coefficients of $\psi$ implies that for any vectors $J_1,J_2$ in $E$ such that $J_2\geq J_1\geq 0$, then for any non-negative integer $n$, and any vertices $\alpha,\beta$ of $\mathcal{V}$, we have $(D_{J_2} \psi)^{\circ n}(1_{\alpha})(\beta)\geq (D_{J_1} \psi)^{\circ n}(1_{\alpha})
 	(\beta)$. Hence if $(D_{J_1} \psi)^{\circ n}(1_{\alpha})(\beta)$ is non-zero, so is its analogous with $J_2$. Hence the statement with \og There exists $J\in E$ with non-negative coordinates ...\fg{}.
\end{Rem}

\subsection{Graduation of \texorpdfstring{$\mathcal{V}$}{Graphe V}}

Now that we know more about the directed geodesic paths of $\mathcal{G}_\Xi$, we can precise the structure of the quotient graph $\mathcal{V}$. We separate the vertices of $\mathcal{G}_\Xi$ (and thus the vertices of $\mathcal{V}$) into $2$ categories, the elements in $\Xi_{<\infty}$, and the ones in $\Xi_\infty$ (Introduced in section \ref{Asymp-Section_Asymptotics} in \cite{Asymp} and recalled below in Definition \ref{Def_of_V_N_and_V_infinity}).

We will show by means of Proposition \ref{Lemma_of_the_ecluse} that the subgraph of $\mathcal{V}$ induced by the vertices of this second category is connected.

\begin{Rem}\label{Remark_1_Lalley}
	Following Steven P.Lalley's work in \cite{Lalley_1993}, if we assume that for any pair of neighbours $(x,y)$ in the tree $X$, we have $p(x,y)>0$, then any element $(a,b)_{x,y}\in\Xi$ is in this second category $\Xi_\infty$.
\end{Rem}
\begin{Def}\label{Def_of_V_N_and_V_infinity}\textit{The sets $\Xi_{<\infty}$ and $\Xi_\infty$, and the directed subgraphs $\mathcal{V}_{<\infty}$ and $\mathcal{V}_\infty$,}\index{$\Xi_{<\infty},\Xi_\infty$}\index{$V_N, V_\infty$}\index{$\mathcal{V}_{<\infty},\mathcal{V}_\infty$}

	For $N\in\mathbb{N}$, we define the subset $\Xi_N$ of vertices $(a,b)_{x,y}$ of $\Xi$ for which the $p$-admissible paths of $\mathcal{P}h(a,b;\mathcal{B}(y)^\complement)$ have their vertices at most at distance $N$ away from $y$. I.e $\forall\gamma=(\omega_0,...,\omega_n)\in\mathcal{P}h(a,b;\mathcal{B}(y)^\complement)$, $$\max_i(d(\omega_i,y))\leq N.$$

	We then define $V_N$ as the subset of vertices of $\mathcal{V}$ consisting of the vertices $[(a,b)_{x,y}]$ in $\Gamma\backslash\Xi $ admitting a representative $(a,b)_{x,y}$ in $\Xi_N$.

	Then let $\Xi_{<\infty}:=\displaystyle\bigcup_{N\in\mathbb{N}} \Xi_N$, and $V_{<\infty}:=\displaystyle\bigcup_{N\in\mathbb{N}} V_N=\Gamma\backslash\Xi_N$.

	Finally, we define the subset $\Xi_\infty:=\Xi\setminus \Xi_{<\infty}$ consisting of the vertices $(a,b)_{x,y}\in\Xi$ such that for any integer $N>0$, there exists a $p$-admissible path $\gamma=(\omega_0,...,\omega_n)$ in $\mathcal{P}h(a,b;\mathcal{B}(y)^\complement)$ such that $\max_i(d(\omega_i,y))\geq N$. We will denote $V_\infty$ the subset of vertices of $\mathcal{V}$ consisting of vertices $[(a,b)_{x,y}]\in\Gamma\backslash\Xi$ admitting a representative $(a,b)_{x,y}$ in $\Xi_\infty$.

	For $N\in\mathbb{N}$, let $\mathcal{V}_N$ be the subgraph of $\mathcal{V}$ induced by the set of vertices $V_N$, also note $\mathcal{V}_{<\infty}$ the subgraph induced by $V_{<\infty}$, and $\mathcal{V}_\infty$ the subgraph induced by $V_\infty$.
\end{Def}
\begin{Rem}
Let $g\in\Gamma$ be an automorphism that preserves the transition kernel $p$, let $(a,b)_y$ be an element in $\Xi$ and let $\gamma=(\omega_0,...,\omega_n)$ be a $p$-admissible path in $\mathcal{P}h(a,b;\mathcal{B}(y)^\complement)$. 
The $p$-admissible path $g\gamma=(g\omega_0,...,g\omega_n)$ belongs to $\mathcal{P}h(ga,gb;\mathcal{B}(gy)^\complement)$, besides, since $g$ is an automorphism of $X$, for any integer $i\in\{0,...,n\}$, we have $d(\omega_i,y)=d(g\omega_i,gy)$. We deduce that if $g (a,b)_y$ belongs to $\Xi_N$, for some non-negative integer $N$, then so do $(a,b)_y$. And by extension, for any non-negative integer $N$, the subset $\Xi_N$ is a $\Gamma$-invariant subset of $\Xi$. And so is $\Xi_\infty$.
\end{Rem}
\begin{Lem}\label{Graduation_of_Xi}
Let $N$ be a non-negative integer. We have the following result:
\begin{itemize}
	\item Any edge with aim a vertex in $\Xi_N$ takes its source in $\Xi_N$.
	
	\item Any edge with source a vertex of $\Xi_\infty$ has its aim in $\Xi_\infty$.
\end{itemize}
	In other words, for any integers $M>N\geq 0$, there is no edge with source in $\Xi_M\setminus \Xi_{M-1}$ and aim in $\Xi_N$, nor any edge with source in $\Xi_\infty$ and aim in $\Xi_N$.
\end{Lem}
Passing to the quotient modulo $\Gamma$ we obtain:
\begin{Lem}\label{Graduation_de_V}~
Let $N$ be a non-negative integer. We have the following results:
\begin{itemize}
	\item Any edge with aim a vertex in $V_N$ takes its source in $V_N$.
	
	\item Any edge with source a vertex of $V_\infty$ has its aim in $V_\infty$.
\end{itemize}
	In other words, for any integers $M>N\geq 0$, there is no edge with source in $V_M\setminus V_{M-1}$ and aim in $V_N$, nor any edge with source in $V_\infty$ and aim in $V_N$.
\end{Lem}
\begin{proof}[Proof of Lemma \ref{Graduation_of_Xi}]
 	Consider an edge of $\mathcal{G}_\Xi$ with source $(a,b)_{x,y}\in\Xi$ and aim $(a_0,b_0)_{x_0,y_0}\in\Xi$. By Proposition \ref{Lemma_of_the_ecluse} (More precisely the implication $i)\Rightarrow ii)$ and $iii)$), $y$ belongs to the geodesic segment $[a,y_0]$, so do $x$ and $x_0$ by definition (See notation \ref{notation_post_lemma_ecluse}), and there exist two $p$-admissible paths $\gamma_s\in\mathcal{P}h(a_0,a;\mathcal{B}(y_0)^\complement)$, $\gamma_b\in\mathcal{P}h(b,b_0;\mathcal{B}(y_0)^\complement)$ such that for any $p$-admissible path $\gamma\in\mathcal{P}h(a,b;\mathcal{B}(y)^\complement)$,
	$$\gamma_s\ast\gamma\ast\gamma_b\in\mathcal{P}h(a_0,b_0;\mathcal{B}(y_0)^\complement).$$

	If $(a_0,b_0)_{x_0,y_0}\in \Xi_N$, then any vertex of the previous $p$-admissible path $\gamma_s\ast\gamma\ast\gamma_b$ is at distance at most $N$ from $y_0$. In particular, any $p$-admissible path $\gamma\in\mathcal{P}h(a,b;\mathcal{B}(y)^\complement)$ is at distance at most $N$ from $y_0$.
	Moreover, since $y,x_0\in[a,y_0]$ and $x\in [a,y]$, we have the inclusion of shadows: $$\Omega_x^y\subset\Omega_{x_0}^{y_0}.$$
	Remember from Lemma \ref{Asymp-Shadow_Lemma} in \cite{Asymp}, that all vertices of $\gamma$ distinct from $b$ are in the shadow $\Omega_x^y$ and thus at distance at most $N-d(y_0,y)\leq N$ from $y$. Hence, by definition $(a,b)_{x,y}\in \Xi_N$.

	The second point follows from the first by contraposition.
	\end{proof}

\begin{figure}
	\centering
	\includegraphics[scale=0.45]{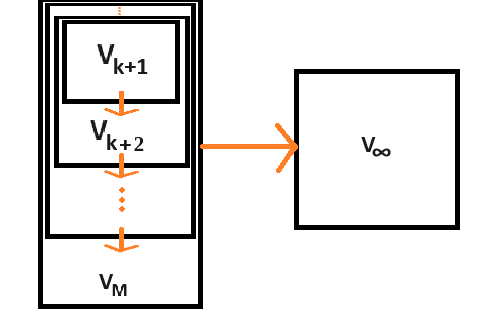}
	\caption{Structure of $\mathcal{V}$}
\end{figure}
 
\begin{Rem}\label{Connexite_et_palliers}
	Since for all non-negative integer $N\in\mathbb{N}$, there are no edge from a vertex of $\Xi_N\setminus \Xi_{N-1}$ to a vertex of $\Xi_{N-1}$, the connected components of $\mathcal{G}_\Xi$ take all their vertices in $\Xi_\infty$ or take all their vertices in a \og\textit{level set}\fg{} $\Xi_{N}\setminus \Xi_{N-1}$, with $N$ depending on the related component.
\end{Rem}
\begin{Cor}\label{Cor_pre_Prop_linearity}
 Let $N$ be a non-negative integer and let $\theta=(a,b)_{x,y}\in \Xi_N\setminus \Xi_{N-1}$ for some integer $N$ and 
let $c\in\mathcal{B}(y)^\complement$ and $(c_0,c_1,...,c_l)\in \Xi_{[c,y]}$, with $c_0=c$ and $c_l=b$ be such that $p(a,c)>0$. And for $j\in\{1,...,l\}$ denote by $\theta_j\in\Xi$ the element 
$$\theta_j=(c_{j-1},c_j)_{x_{i_j}},$$
where $(i_1,...,i_l)=\iota_{[c,y]}(c_0,...,c_l)$ are the crossing indices associated with $(c_0,...,c_l)$, along the geodesic segment $[c,y]=(x_0,...,x_m)$. 
 
 If for some $1\leq j \leq l$, the vertex $\theta_j$ is in $\Xi_N\setminus \Xi_{N-1}$ then necessarily $j=l$. Equivalently, we always have that $\theta_1$, ...,$\theta_{l-1}$ are in $\Xi_{N-1}$.
\end{Cor}
\begin{proof}
With the above notation, for any $1\leq j \leq l$, let $x_{i_j}^*$ be the only neighbour of $x_{i_j}$ in the geodesic segment $[c_{j-1},x_{i_j}]$, so that the vertex $\theta_j$ is $(c_{j-1},c_j)_{x_{i_j}^*,x_{i_j}}\in \Xi$. By the previous Lemma \ref{Graduation_of_Xi}, the vertices $\theta_1,...,\theta_l$ of $\mathcal{G}_\Xi$ necessarily belong to $\Xi_N$.

For any $1\leq j \leq l$, let $N_j\leq N $ be a non-negative integer such that $\theta_j$ belongs to $\Xi_{N_j}\setminus\Xi_{N_j -1}$ and let $\gamma_j$ be a $p$-admissible path of $\mathcal{P}h(c_{j-1},c_j;\mathcal{B}(x_{i_j})^\complement)$, such that one of its vertices is at a distance exactly $N_j$ from $x_{i_j}$. Reasoning as in the proof of Lemma \ref{Graduation_of_Xi}, the $p$-admissible path $\gamma_j$ cannot go at distance more than  $N - d(y,x_{i_j})$ from $x_{i_j}$. Thus if $x_{i_j}$ is different from $y$, $N_j\leq N - d(y,x_{i_j})\leq N-1$, and $\theta_j$ is in $\Xi_{N-1}$.
The result follows.

\end{proof}
\begin{Def}\textit{Reachable vertex}

	Let $(a,b)_{x,y}\in\Xi$, we will say that a vertex $\omega\in X_0$ is reachable for $(a,b)_{x,y}$, if there exists a $p$-admissible path $\gamma=(\omega_0,...,\omega_n)\in\mathcal{P}h(a,b;\mathcal{B}(y)^\complement)$ such that $\omega=\omega_i$ for some $i\in\{0,...,n\}$.
\end{Def}

\begin{Rem}
	For $N\in\mathbb{N}$, and $(a,b)_{x,y}\in\Xi_N$, the set of vertices reachable for $(a,b)_{x,y}$ is included in the ball $\mathcal{B}_N(y)$. If $(a,b)_{x,y}\in\Xi_\infty$ then there are reachable vertices for $(a,b)_{x,y}$ at arbitrarily large distance away from $y$.
\end{Rem}

\begin{Prop}\label{Stagnation_of_V_N}~

	There exists $M>0$ such that 
	$$\forall N \geq M,\; V_N=V_M.$$
	In other words, there exists $M>0$ such that for any $(a,b)_{x,y}\in\Xi$, if there exists a $p$-admissible path $\gamma$ in $\mathcal{P}h(a,b;\mathcal{B}(y)^\complement)$ that goes at distance at least $M+1$ from $y$, then there are $p$-admissible paths in $\mathcal{P}h(a,b;\mathcal{B}(y)^\complement)$ going arbitrarily far away from $y$.

	Furthermore, if $M$ is large enough then for any vertex $\omega'$ of $X$ reachable for $(a,b)_{x,y}$ and at distance $d(\omega',y)>M$ from $y$, we have that all vertices in the shadow $\Omega^y_{\omega'}$ of $\omega'$ under $y$, are also reachable for $(a,b)_{x,y}$.
\end{Prop}

\begin{proof}
	We give ourselves a finite system of representatives $S_1$ of $\Gamma\backslash X_1$, the set of edge classes of $X$ under the action of $\Gamma$. For each pair of representatives $(x,y)\in S_1$, choose a $p$-admissible path $\gamma_{(x,y)}$ in $\mathcal{P}h(x,y)$, which exists by irreducibility of the Markov chain $(X_0,p)$. And set:
	\begin{gather*}
		M_{(x,y)}:=\max_{\omega\in\gamma_{(x,y)}} d(\omega,y)\\
		M:=\max_{(x,y)\in S_1} (M_{(x,y)})+k;
	\end{gather*}
	Where $k$ is, as you may recall, a non-negative integer such that 
	$$\forall x,y\in X_0,\, d(x,y)>k\Rightarrow p(x,y)=0 .$$
	Therefore, by finiteness of the set $\Gamma\backslash X_1$ and $\Gamma$-invariance of the transition kernel $p:X_0\times X_0\to[0,1]$, for any pair of neighbours $(x,y)$ in $X_1$, there exists a $p$-admissible path in $\mathcal{P}h(x,y)$ whose vertices are at most at distance $M-k$ away from $y$.

	We fix an element $(a,b)_{x,y}$ of $\Xi$ for which there exists a $p$-admissible path $\gamma\in\mathcal{P}h(a,b;\mathcal{B}(y)^\complement)$, one of whose vertices, say $\omega'$, satisfies 
	$$d(\omega',y)\geq M+1.$$
	we will show that $(a,b)_{x,y}$ belongs to $\Xi_\infty$.

	We decompose $\gamma$ by one of its passage at $\omega'$: Let $\gamma_a\in\mathcal{P}h(a,\omega';\mathcal{B}(y)^\complement)$ and $\gamma_b'\in\mathcal{P}h(\omega',b;\mathcal{B}(y)^\complement)$ such that 
	$$\gamma=\gamma_a\ast\gamma_b'.$$

	Let $\omega$ be a vertex of $X$ in the shadow $\Omega_{\omega'}^{y}$ of $\omega'$ under $y$. The geodesic segment $[\omega',\omega]=:(x_0,...,x_m)$ is included in this same shadow. For each $s\in\{1,...,m\}$ let $\gamma_s$ and $\gamma_s'$ be two $p$-admissible paths in $\mathcal{P}h(x_{s-1},x_s)$ and $\mathcal{P}h(x_s,x_{s-1})$ respectively, whose vertices are at most at distance $M-k$ away from $x_{s-1}$ and $x_s$ respectively.
	By considering the distances to $y$, we check that the $p$-admissible path
	$$\gamma_a\ast\gamma_1\ast\cdot\cdot\cdot\ast\gamma_m
	\ast\gamma_m'\ast\cdot\cdot\cdot\ast\gamma_1'\ast\gamma_b'$$
	is in $\mathcal{P}h(a,b;\mathcal{B}(y)^\complement)$ and passes through the vertex $\omega=x_m$.

	We just have shown that if $(a,b)_{x,y}\in\Xi$ is such that there exists a $p$-admissible path of the set $\mathcal{P}h(a,b;\mathcal{B}(y)^\complement)$, passing through a vertex $\omega'$ sufficiently far away from $y$ then, all elements $\omega$ in the shadow $\Omega_{\omega'}^y$, of $\omega'$ under $y$, are reachable for $(a,b)_{x,y}$ and in particular $(a,b)_{x,y}$ belongs to $\Xi_\infty$.
\end{proof}

A little extension of the preceding proof shows that:
\begin{Cor}\label{Stagnation_of_VN_upgraded}
	With the notation of the (previous) proof of Proposition \ref{Stagnation_of_V_N}, given $(\varpi_1,...,\varpi_l)$ a tuple of vertices of $X$ that belong to the shadow $\Omega_{\omega'}^y$, there exists a $p$-admissible path $\gamma=(\omega_0,...,\omega_n)$ in $\mathcal{P}h(a,b;\mathcal{B}(y)^\complement)$ such that for some non-decreasing sequence of indices $i_1\leq...\leq i_l$, we have:
	$$\omega_{i_1}=\varpi_1, \, ...,\, \omega_{i_l}=\varpi_l.$$
\end{Cor}
\begin{prop}~

	The set $\Xi_\infty$ (thus also the set $V_\infty$) is non-empty.
\end{prop}
\begin{proof}
	We will show that for any non-negative integer $N$, the set $\Xi\setminus \Xi_N$ is non-empty, which together with Proposition \ref{Stagnation_of_V_N} implies the result.
	
	Let $x_0\in X_0$. By irreducibility of the Markov chain $(X_0,p)$, for any integer $N\geq 1$ there is a $p$-admissible path $\gamma_N$ of $\mathcal{P}h(x_0,x_0)$ going at distance at least $N$ away from $x_0$. For each $\gamma_N$ denote by $\omega_N$ the first vertex visited by $\gamma_N$ at distance at least $N$ from $x_0$.

	For each $N$, let $a_N\in\mathcal{B}(x_0)\setminus\{x_0\}$, be the last vertex in $\mathcal{B}(x_0)$ visited by $\gamma_N$ before visiting $\omega_N$ (such an $a_N$ exists since $\omega_N\neq x_0$ and for all $c\notin\mathcal{B}(x_0)$, we have $p(x_0,c)=0$).
	
	Next, let $x_N\in X_0$ be such that $d(x_N,a_N)=k$ and $x_0\in [x_N,a_N]$, then let $y_N$ be a neighbour of $x_N$ such that $x_N$ belongs to the geodesic segment $[y_N,a_N]$. Note that $d(y_N,x_0)=d(y_N,a_N)-d(a_N,x_0)\leq (k+1)-1\leq k$.

	Finally, let $b_N$ be the first vertex of $\mathcal{B}(y)$ visited by $\gamma_N$ after $\omega_N$.
	By construction $(a_N,b_N)_{x_N,y_N}$ is in $\Xi\setminus\Xi_{N-1}$. In particular for $N>M$, where $M$ comes from Proposition \ref{Stagnation_of_V_N}, we have $(a_N,b_N)_{x_N,y_N}\in\Xi_\infty$. Thus $\Xi_\infty$ (and also $V_\infty$) is not empty.
\end{proof}

\subsection{Connectedness of \texorpdfstring{$\mathcal{V}_\infty$}{Graphe Vinfini}}
~ 

We now prove that $\mathcal{V}_\infty$ is connected and that from any vertex $[\theta]$ of $\mathcal{V}$, there exists a path in $\mathcal{V}$ with source $[\theta]$ and with aim a vertex of $\mathcal{V}_\infty$.

\begin{Def}\textit{Visible edge from another}\label{Def_visibility}

	Let $(x_0,y_0), (x,y)\in X_1$ be two edges of the tree $X$. The edge $(x,y)$ is said to be \textit{ visible} from $(x_0,y_0)$ if the vertices $x_0$ and $y$ lie within the geodesic segment $[x,y_0]$.
	Equivalently, if the shadow of $x$ under $y$ is included in the shadow of $x_0$ under $y_0$, $$\Omega_x^y\subseteq\Omega_{x_0}^{y_0}.$$
\end{Def}

In the following lemma, and only in this one, we use the assumption that every vertex of $X$ possesses at least $3$ neighbours.
\begin{Lem}\label{Lem_of_visible_edges}
	For any pair of edges $(x,y)\in X_1$, $(x_0,y_0)\in X_1$, there exists $g\in\Gamma$ such that $(gx,gy)$ is visible from $(x_0,y_0)$, i.e. such that 
	$$\Omega_{gx}^{gy}\subseteq\Omega_{x_0}^{y_0}.$$
\end{Lem}
 This result comes from \cite{Quint_Trees_2} Corollary 5.5. In this corollary, the result is stated under the additional assumption that the action of $\Gamma$ on $X$ is proper. But actually, this is automatically satisfied for $\Gamma$ is a closed subset of the compact\footnote{Note that the natural topology to put on $\operatorname{Aut}(X)$ is the topology of the pointwise convergence, which makes $\operatorname{Aut}(X)$ a compact Hausdorff set.} $\operatorname{Aut}(X)$.
	
\begin{Prop}\label{Prop_V_infini_is_absorbent}~

	Let $[\theta_0]$ be a vertex of $\mathcal{V}$ that is in $V_\infty$, then for any vertex $[\theta]$ of $\mathcal{V}$, there exists a path in $\mathcal{V}$ with source $[\theta] $ and aim $[\theta_0]$.
\end{Prop}
\begin{proof}
	Fix two vertices, $[(a_0,b_0)_{x_0,y_0}]$ from $\mathcal{V}_\infty$, $[(a,b)_{x,y}]$ from $\mathcal{V}$. Let us show that there is a path in $\mathcal{V}$ from $[(a,b)_{x,y}] $ to $[(a_0,b_0)_{x_0,y_0}]$.

	According to Proposition \ref{Lemma_of_the_ecluse}, we only need to find $(a',b')_{x',y'}\in [(a,b)_{x,y}] $ such that $y'\in [a',y_0]$, and two $p$-admissible paths $\gamma_1\in\mathcal{P}h(a_0,a';\mathcal{B}(y_0)^\complement)$ and  $\gamma_2\in\mathcal{P}h(b',b_0;\mathcal{B}(y_0)^\complement)$ (See remark \ref{Rem_10.10}).  
 
 	Let $M>0$ coming from the Proposition \ref{Stagnation_of_V_N} and let $\omega_0$ be reachable for $(a_0,b_0)_{x_0,y_0}$ at distance $d(\omega_0,y)>M$. So all vertices of the shadow $\Omega_{\omega_0}^{y_0}$ are reachable for $(a_0,b_0)_{x_0,y_0}$. By the previous Lemma \ref{Lem_of_visible_edges}, there exists an edge $(x',y')\in\Gamma\cdot (x,y)\subset X_1$ visible for $(x_0,y_0)$ such that the vertices $x'$ and $y'$ of $X$ are in the shadow of $\omega_0$ under $y_0$, and in particular are reachable for $(a_0,b_0)_{x_0,y_0}$. Then let $a',b'\in\Omega_{x'}^{y'}$ be such that $(a',b')_{x',y'}\in [(a,b)_{x,y}]$. These two vertices $a'$ and $b'$ are also reachable vertices for $(a_0,b_0)_{x_0,y_0}$ (given the inclusion $\Omega_{x'}^{y'}\subset\Omega_{\omega_0}^{y_0}$).

	Then, by the definition of a reachable vertex, there exist two $p$-admissible paths $\gamma_{1}$ in the set $\mathcal{P}h(a_0,a';\mathcal{B}(y_0)^\complement)$ and $\gamma_{2}$ in $\mathcal{P}h(b',b_0;\mathcal{B}(y_0)^\complement)$. This is what we wanted.

	So, by Proposition \ref{Lemma_of_the_ecluse}, there is a path in $\mathcal{V}$ with source $[(a,b)_{x,y}] $ and aim $[(a_0,b_0)_{x_0,y_0}] $. The result follows by arbitrariness of the vertices $[(a,b)_{x,y}]$ of $\mathcal{V}$ and $[(a_0,b_0)_{x_0,y_0}]$ of $\mathcal{V}_\infty$.
\end{proof}

\begin{Cor}\label{Cor_connexite_de_V_infini}~

	The graph $\mathcal{V}_\infty$ is connected.
\end{Cor}
\begin{proof}
	This is an immediate consequence of the previous proposition and Lemma \ref{Graduation_de_V} stating that any edge with a source in $\mathcal{V}_\infty$ has its aim in $\mathcal{V}_\infty$.
\end{proof}

\begin{Rem}
	Under the additional technical assumption of nearest-neighbour walking ($(x,y)\in X_1 \Rightarrow p(x,y)>0$), S. P. Lalley in the Proposition 2.7 of \cite{Lalley_1993}, proved that the graph $\mathcal{V}=\mathcal{V}_\infty$ is connected (See Remark \ref{Remark_1_Lalley}).
\end{Rem}

\subsection{Degree of \texorpdfstring{$\psi_E$}{PsiE}}
~

Before studying the structure of the Lalley's curve $\mathcal{C}$, we state one last result concerning the degree of the polynomial functions $J\mapsto\psi_E(J)(\theta)$ for $\theta\in\Xi_\infty$.

	We provide the vector space $E$ identified with $\mathcal{F}(\Gamma\backslash\Xi,\mathbb{C})$, with the basis $\{1_{[\theta]}\}_{[\theta]\in\Gamma\backslash\Xi}$. 
We now show using the Proposition \ref{Proposition_of_the_ecluse} that $\psi_E$ is of degree at least $2$ with respect to the variables $(J([\theta]))_{[\theta]\in V_\infty}$, in the following sense:

There exist elements $[\theta] ,[\theta_1] , [\theta_2]$ in $V_\infty$, polynomial functions with non-negative coefficients $\phi$ and $\chi$, with $\phi$ non-zero such that:
$$\psi_E(J)([\theta])=J([\theta_1] )J([\theta_2] )\phi(J) +\chi(J).$$
We will say that $J\mapsto\psi_E(J)([\theta] )$ admits $J([\theta_1]) \cdot J([\theta_2] )$ as a factor of one of its monomials. 
Said differently, if we denote by $(a,b)_y=\theta$ a representative of the orbit $[\theta]=\Gamma\theta$, if there exist a vertex $c\in\mathcal{B}(y)^\complement$, a tuple $(c_0,...,c_l)\in\Xi_{[c,y]}$ such that $p(a,c)>0$, $c_0=c$ and $c_l=b$, then denoting by $\theta_j:=(c_{j-1},c_j)_{x_{i_j}}$ for any $j\in\{1,...,l\}$ and $(i_1,...,i_l)=\iota_{[c,y]}(c_0,...,c_l)$, we will say that $J(\theta_i)J(\theta_j)$ is a \textit{factor of one of the monomials} of $J\mapsto \psi(J)(\theta)$, for any distinct integers $i\neq j$ in $\{1,...,l\}$. And passing to the quotient by $\Gamma$, that $J([\theta_i])J([\theta_j])$ is a \textit{factor of one of the monomials} of $J\mapsto \psi_E(J)([\theta])$.
We aim to find vertices $\theta,\theta_1,\theta_2$ in $\Xi_\infty$ such that 
$J(\theta_i)J(\theta_j)$ is a factor of one of the monomials of $J\mapsto \psi(J)(\theta)$ (Compare with Corollary \ref{Cor_pre_Prop_linearity}).

Given the Proposition \ref{Graduation_de_V} and Proposition \ref{Proposition_of_the_ecluse}, we wish to prove that there exists a $p$-admissible path $\gamma$ and a positive integer $n$ such that the set $\Theta_n(\gamma)$ contains two elements whose image by the graph morphism $\mathcal{T}_\gamma\to\mathcal{G}_\Xi$ are in $\Xi_\infty$. 

\begin{Prop}\label{degre_infini_de_psi}
The polynomial map $\psi_E$ is of degree at least $2$ with respect to the variables $(J(\theta))_{\theta\in V_\infty}$
\end{Prop}
\begin{proof}
    We start by proving that for some $(a,b)_y\in\Xi_\infty$ and some positive integer $N\geq 1$, the polynomial function $J\mapsto\psi^{\circ N}(J)(a,b)_y)$ is of degree at least $2$ with respect to the variables $(J(\theta))_{\theta\in\Xi_\infty}$.
    
    Fix an arbitrary element $(a,b)_y$ in $\Xi_\infty$. Using Proposition \ref{Proposition_of_the_ecluse} and Proposition \ref{Graduation_of_Xi} it is sufficient to prove that there exists a $p$-admissible path $\gamma\in\mathcal{P}h(a,b;\mathcal{B}(y)^\complement)$ such that there exists a positive integer $N\geq 1$, with $\Theta_N(\gamma)$ containing two intervals/vertices of the simple descending rooted tree $\mathcal{T}_\gamma$ whose images under the graph morphism $\phi_\gamma:\mathcal{T}_\gamma\to\mathcal{G}_\Xi$ belong to $\Xi_\infty$. 
    
    Actually, thanks to Proposition \ref{Graduation_of_Xi}, since given a vertex $I$ of $\mathcal{T}_\gamma$, whose image by $\phi_\gamma$ is in $\Xi_\infty$ has all its descendent in $\Xi_\infty$ (That is vertices $I'$ of $\mathcal{T}_\gamma$ strictly containing the interval $I$, see the first assertion in Proposition \ref{PROP:FLOODED_CAVERN_TREE_PROPERTY}), all we need to prove is that there exist two intervals/vertices $I_1, I_2$ of $\mathcal{T}_\gamma$ with $I_1\not\subset I_2$ and $I_2\not\subset I_1$, such that $\phi_\gamma(I_1)$ and $\phi_\gamma(I_2)$ are in $\Xi_\infty$.
    
    We construct such $p$-admissible path $\gamma$. Let $M\geq 1$ be as in the Proposition \ref{Stagnation_of_V_N}. Let $[(a,b)_{x,y}]$ be in $V_\infty$ and let $\omega$ in $X_0$ be reachable for $(a,b)_{x,y}$, such that $d(\omega,y)>M$. Then let $\varpi_1$, $\varpi_2$ be in the shadow $\Omega_\omega^y$ of $\omega$ under $y$, such that $d(\varpi_1,\omega)>M+k$ and $d(\varpi_2,\omega)>M+k$. Set $\gamma=(\omega_0,...,\omega_n)$ to be a $p$-admissible path in $\mathcal{P}h(a,b;\mathcal{B}(y)^\complement)$ passing through the vertices $\varpi_1$, $\omega$, $\varpi_2$ in this order at least once (See figure \ref{Mountain_School_Figure}). Such $p$-admissible path $\gamma$ exists thanks to the extension of Proposition \ref{Stagnation_of_V_N}, see Corollary \ref{Stagnation_of_VN_upgraded}.
    
    Denote by $l_1<l<l_2$ three integers in $\{0,...,n\}$, such that $\omega_{l_1}=\varpi_1$, $\omega_l=\omega$ and $\omega_{l_2}=\varpi_2$. Then set 
    $i_1:=\max\{ i : 1\leq i <l_1 : d(y,\omega) \leq d(y,\omega_i) < d(y,\omega) + k \}$ and 
    $i_2:=\max\{ i : l\leq i <l_2 : d(y,\omega) \leq d(y,\omega_i) < d(y,\omega) + k \}$. So that if we denote $j_1$ (resp $j_2$) the integer $\min\{ t> i_1 : d(y,\omega_t)<d(y,\omega_{i_1}) \}$ (resp. $\min\{ t> i_1 : d(y,\omega_t)<d(y,\omega_{i_2}) \}$) then the intervals $I_1$ and $I_2$ are in $\mathcal{T}_\gamma$. Now by construction $l_1\in I_1$ and $l_2\in I_2$, and we also have $d(y,\varpi_1)-d(\omega_{i_1},\varpi_1)>M$ and $d(y,\varpi_2)-d(\omega_{i_2},\varpi_2)>M$, thus we deduce, via the Proposition \ref{Stagnation_of_V_N}, that $\phi_\gamma(I_1)$ and $\phi_\gamma(I_2)$ belong to $\Xi_\infty$.
    
    Finally by contradiction, suppose that $\psi_E$ is at most of degree $1$ with respect to the variables $(J([\theta] ))_{[\theta]\in V_\infty}$ then because of the Lemma \ref{Graduation_de_V}, necessarily for any integer $n\geq 1$, the polynomial map $\psi_E^{\circ n}$ would also have degree at most $1$ with respect to the variables $(J([\theta] ))_{[\theta]\in V_\infty}$. But this is not the case for $n=N$. Thus necessarily $\psi_E$ is of degree at least $2$ with respect to the variables $(J([\theta]))_{[\theta]\in V_\infty}$.
    \end{proof}
    
	\begin{figure}
		\centering
		\includegraphics[scale=0.4]{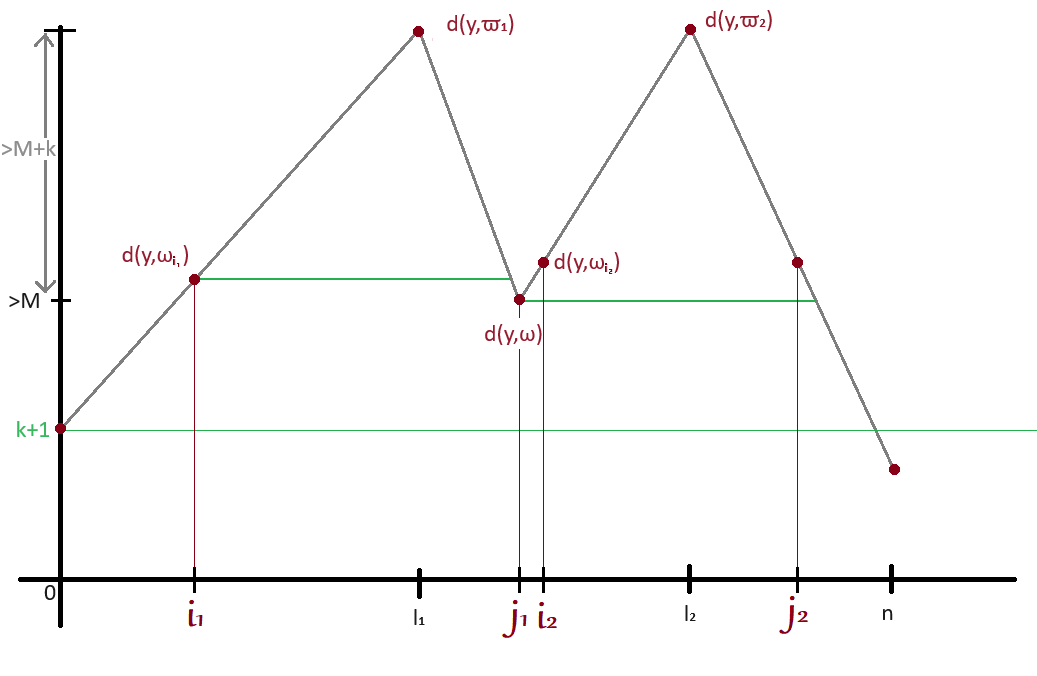}
		\caption{$i\mapsto d(\omega_i,y)$}
		\label{Mountain_School_Figure}
	\end{figure}
	
We have a result analogous to the Proposition \ref{Prop_Chemin_et_derivee}:
\begin{Prop}~

	We identify $E$ as $\mathcal{F}(\Gamma\backslash\Xi,\mathbb{C})\approx\mathbb{C}^{\Gamma\backslash\Xi}$. With the previous assumptions and notations, for three vertices $[\theta] ,[\theta_1] ,[\theta_2]$ of $\mathcal{V}$, the polynomial function $J\mapsto\psi(J)([\theta] )$ admits $J([\theta_1]) \cdot J([\theta_2] )$ as a factor of one of its monomials if, and only if the second differential of $\psi_E$ at $[\theta] $ evaluated at $(1_{[\theta_1] },1_{[\theta_2] })\in E\times E$ is strictly positive: 
	$$D^2 \psi_E(1_{[\theta_1] },1_{[\theta_2] })([\theta] )>0.$$
\end{Prop}
\begin{proof}
	The proof follows the same reasoning as in the proof of Proposition \ref{Prop_Chemin_et_derivee}.
\end{proof}

\section{The structure of \texorpdfstring{$\mathcal{C}$}{Lalley's Curve} and critical point of \texorpdfstring{$u$}{the function u}}\label{Section Structure of C}

The aim of this section is to refine our understanding of the Lalley's curve $\mathcal{C}$. We prove in this section that $\mathcal{C}$ is a smooth complex curve in the neighbourhood of $\overline{\{u(z):z\in\mathbb{D}(0,R)\}}$ (Proposition \ref{Regularity_of_C}), and compute the tangent space of $\mathcal{C}$ at $u(R)$ (Proposition \ref{Prop_Spectrum_belonging_C}). To do so, we will need to study the spectrum of the differentials $D_{v_z}\psi$ where $v:\mathbb{D}(0,R)\to E, z\mapsto v_z$ is the map in $E=\mathcal{F}(\Xi,\mathbb{C})^\Gamma$ defined by:
\begin{equation}\label{DEF:V_GREEN2}
		v_z:(a,b)_{y}\mapsto G_z(a,b;\mathcal{B}(y)^\complement).
	\end{equation}
 For that we will use our study of the iterations $(\psi^{\circ n})_n$  from the previous section and an order relation on the connected components of the dependency digraph $\mathcal{V}$. Let us fix some notations:
\subsection{The functions \texorpdfstring{$v_z^{(\alpha)}:\Xi_\alpha\to \mathbb{C}$}{function vzalpha}}~

We begin by enumerating each connected component $(C_\alpha)_{\alpha\in\Lambda}$\index{$C_\alpha$} of $\mathcal{V}$, where $\Lambda$ is a set in bijection with the set of connected components of $\mathcal{V}$. Let $\preceq$\index{$\preceq$} be the partial order relation on $\Lambda$ defined by: for $\alpha,\beta\in\Lambda$, $\alpha\preceq\beta $ if there is a path in $\mathcal{V}$ with source a vertex in the connected component $C_\alpha$, and aim a vertex in the connected component $C_\beta$. We will write $\alpha\prec\beta$ for $\alpha\preceq\beta$ and $\alpha\neq\beta$. Note that this order relation is well defined and its structure is partially described by Proposition \ref{Graduation_de_V} (See remark \ref{Connexite_et_palliers}).

By the Proposition \ref{Prop_V_infini_is_absorbent}, there exists a greatest element for this order relation:
$$\alpha_{max}\in\Lambda\text{ such that } C_{\alpha_{max}}=V_\infty.$$
we will note \og $\infty$ \fg{} this greatest element. Beside, since $\mathcal{V}$ is finite, $\Lambda$ is also finite and therefore, admits minimal elements.

\begin{Def}\label{Def_sets_Xi_alpha_spaces_E_alpha}

	For $\alpha\in\Lambda$, set:
	\begin{equation}\index{$\Xi_\alpha$}
		\Xi_\alpha:=\Big\{(a,b)_{y}\in\Xi: [(a,b)_{y}]\in C_\alpha\Big\},
	\end{equation}
	\begin{equation}\index{$E_\alpha$}
	\text{ and }	E_\alpha:=\mathcal{F}(\Xi_\alpha,\mathbb{C})^\Gamma
	\end{equation}
	the vector space of $\Gamma$-invariant complex-valued functions defined on $\Xi_\alpha$. $E_\alpha$ identifies as a vector subspace of $E$, by extending maps in $E_\alpha$ by $0$ outside $\Xi_\alpha$. 

Analogously set $\Xi_{\prec\alpha}:=\bigcup_{\beta\prec\alpha}\Xi_\beta$ and $E_{\prec\alpha}:=(\bigoplus_{\beta\prec\alpha} E_\beta)$ which is isomorphic to $\mathcal{F}(\Xi_{\prec\alpha},\mathbb{C})^\Gamma$.\index{$\Xi_{\prec\alpha}$}\index{$\Xi_{\prec\alpha}$}
	Let $\pi_\alpha:E\to E_\alpha$ \index{$\pi_\alpha$}be the projection of $E$ onto $E_\alpha$ with respect to the base $(1_{[\theta]})_{[\theta]\in\Gamma\backslash\Xi}$.
\end{Def}

\begin{Rem}
	Note that $\bigcup_{\alpha\prec\infty} C_\alpha=V_{<\infty}$. As a result, we find it consistent to denote $\Xi_{<\infty}$ the union $\bigcup_{\alpha\prec\infty}\Xi_\alpha$ and $E_{<\infty}$ the complex vector space $\mathcal{F}(\Xi_{<\infty},\mathbb{C})^\Gamma$.
\end{Rem}

\begin{Def}\label{Def_v_alpha_R_alpha}\textit{The maps $v^{(\alpha)}$ and $v^{(\prec\alpha)}$}\index{$v^{(\alpha)}$}\index{$v^{(\prec\alpha)}$}

Recall the definition of the map $v:\mathbb{D}(0,R)\to E$ from (\ref{DEF:V_GREEN2}). Let $\alpha$ be an element of $\Lambda$.
	We will write $v^{(\alpha)}$, respectively $v^{(\prec\alpha)}$ for the germs of holomorphic map with values in $E_\alpha$, respectively $E_{\prec\alpha}$, given by the composition $v^{(\alpha)}=\pi_{\alpha}\circ v$, respectively $v^{(\prec\alpha)}=\pi_{\prec\alpha}\circ v$. Note that if $\alpha$ is minimal, then $E_{\prec\alpha}=\{0\}$ by convention and thus $v^{(\prec\alpha)}\equiv 0$.

Let $R_\alpha>0$ \index{$R_\alpha$}be the possibly infinite radius of the largest open disk $\mathbb{D}(0,R_\alpha)$ on which $v^{(\alpha)}$ admits a holomorphic extension.
Note that by the interrelations between the maps \og $v^{(\alpha)}$ \fg{}, from Proposition \ref{Graduation_de_V} (See Lemma \ref{Comparison_Lemma_alpha_beta} below) we deduce that the sequence $(R_\alpha)_\alpha$ is non-increasing with respect to the order relation $\preceq$:$$\forall\beta\prec\alpha, \, R_\alpha\leq R_\beta.$$

Denote $R_{\prec \alpha}:=\min_{\beta\prec \alpha} (R_\beta) \in [1,+\infty]$ (with $\min (\emptyset) = \infty$ )\index{$R_{\prec\alpha}$}, so, by definition $v^{(\prec\alpha)}$ admits a holomorphic extension over the disk $\mathbb{D}(0,R_{\prec\alpha})$.
\end{Def}
\begin{Rem}\label{Rem_R_infini_is_min}
	By monotonicity of $(R_\alpha)_\alpha$ with respect to $\preceq$, $$R_\infty=\min_\alpha (R_\alpha)=R_G\geq R.$$ 
	Admitting Proposition \ref{Regularity_of_C}, about smoothness of $\mathcal{C}$ in a neighbourhood of $\overline{\{u(z): z\in\mathbb{D}(0,R)\}}$, we proved the equality $R_G=R$, see (\ref{Asymp-equality_radius}) in \cite{Asymp} for a justification. 
\end{Rem}

For $\alpha\in\Lambda$, by definition of the edges of $\mathcal{V}$, for any $\theta\in\Xi_\alpha$, the polynomial function $J\mapsto\psi(J)(\theta)$ depends only on the variables $(J(\theta'))_{\theta'\in\Xi_{ \preceq \alpha}}$. For $z\in \mathbb{D}(0,R_{\prec\alpha})$, let $$s_\alpha(z):E_\alpha\to E$$ be the section that extends any map in $E_\alpha$ to the whole set $\Xi$, by giving it the same values as $v^{(\prec\alpha)}(z)$ on $\Xi_{\prec\alpha}$, and extending it by $0$ on $\Xi\setminus (\Xi_\alpha\cup\Xi_{\prec\alpha})$.
For $J^{(\alpha)}\in E_\alpha$, we will write
$$s_\alpha(z)(J^{(\alpha)})=J^{(\alpha)} + v_z^{(\prec\alpha)}\in E.$$

We then define the map $\psi_\alpha: \mathbb{D}(0,R_{\prec_\alpha})\times E_\alpha\to E_\alpha$ by:
\begin{equation} \label{Def_de_psi_alpha_induction}\index{$\psi_\alpha$}
	\psi_\alpha:(z,J^{(\alpha)})\mapsto\pi_{\alpha}(\psi(J^{(\alpha)}+v^{(\prec\alpha)}(z))).
\end{equation}

At last let $\mathcal{C}_\alpha$ be the set of zero of the map \index{$W_\alpha$}
$$W_\alpha:(z,J^{(\alpha)})\mapsto J^{(\alpha)}-z\cdot\psi_\alpha(z,J^{(\alpha)}),$$
\begin{equation}\label{Def_of_C_alpha}
	\mathcal{C}_\alpha:=W_\alpha^{-1}(\{0_\alpha\})=\Big\{(z,J^{(\alpha)}): J^{(\alpha)} = z \psi_\alpha(z,J^{(\alpha)})\Big\}.
\end{equation}

\begin{Rem}
	At $(0,0_\alpha)\in \mathbb{D}(0,R_{\prec\alpha})\times E_\alpha$, we have $W_\alpha(0,0_\alpha)=0_\alpha$ and $\dfrac{\partial W_\alpha}{\partial J^{(\alpha)}}(0,0_\alpha)=Id_{E_\alpha}$. In particular, $(0,0_\alpha)$ is a regular point of $W_\alpha$ and therefore $\mathcal{C}_\alpha$ is a smooth complex curve in the neighbourhood of $(0,0_\alpha)$. 
\end{Rem}
\begin{Rem}
Since $v^{(\prec\alpha)}$ is a rational map (Lemma \ref{Asymp-Rationality_of_g_xi_<_infty} in \cite{Asymp}) so is $s_\alpha$. Besides $\psi:E\to E$ is polynomial and $\pi_\alpha$ is linear, thus the previous map $\psi_\alpha$ is also a rational map from $\mathbb{C}\times E_\alpha$ to $E_\alpha$. And thus $\mathcal{C}_\alpha$ is an algebraic curve in $\mathbb{C}\times E_{\alpha}$, that is smooth in the neighbourhood of $(0,0_\alpha)$.
\end{Rem}

\begin{Lem}\textit{Comparison Lemma}\label{Comparison_Lemma_alpha_beta}

	We identify $E$ with the vector space $\mathcal{F}(\Gamma\backslash\Xi,\mathbb{C})$. 
	Let $\alpha,\beta\in\Lambda$, $\beta\preceq\alpha$, for any $[\theta_s]\in C_\beta$, $[\theta_a]\in C_\alpha$, there exists positive integers $\kappa, e \in\mathbb{N}$ and a positive constant $C>0$ such that:

	For any $J^{(\alpha)}\in E_{\alpha}$ with non-negative coordinates, and $r\geq 0$ such that $J^{(\alpha)}=r\psi_\alpha (r,J^{(\alpha)})$, we have:
	\begin{equation}
		J^{(\alpha)}([\theta_a] )\geq C r^e  s_\alpha(J^{(\alpha)})([\theta_s] ) \left(\min_{\Gamma\backslash\Xi_{\preceq \alpha}}(s_\alpha(J^{(\alpha)}))\right)^{\kappa}
	\end{equation}
\end{Lem}
\begin{proof}
	First, let us assume that there is an edge with source $[\theta_s]$ and aim $[\theta_a]$ in $\mathcal{V}$.
	By definition of the edges of $\mathcal{V}$, there exist a polynomial function $Q$ with non-negative coefficients, elements $\theta_1,...,\theta_l \in \Xi$ and $C>0$ such that for some $i\in\{1,...,l\}$, $\theta_i$ is a representative of $[\theta_s]\in\Gamma\backslash\Xi$ and for any $J\in E$,
	$$C J(\theta_1)\cdot\cdot\cdot J(\theta_l)+Q(J)= \psi(J)(\theta_a).$$
	
	By definition of the order relation $\preceq$, the $\Gamma$-orbits $[\theta_1],...,[\theta_l]$ belong to $\Gamma\backslash\Xi_{\preceq\alpha}$. Then for all real vector $J^{(\alpha)}\in E_{\alpha}$ with non-negative coordinates, and all non-negative real number $r\geq 0$, such that $J^{(\alpha)}=r\psi_\alpha(r,J^{(\alpha)})$, we have for any $\theta\in\Xi_{\preceq \alpha}$, by definition $J^{(\alpha)}([\theta_a] )=r\psi(J^{(\alpha)}+v^{(\prec\alpha)}(r))([\theta_a])$, thus
	\begin{align*}
		J^{(\alpha)}([\theta_a])=r \psi(J^{(\alpha)}+v^{(\prec\alpha)}(r))([\theta_a])&\geq r C s_\alpha(J^{(\alpha)})([\theta_1])\cdot\cdot\cdot s_\alpha(J^{(\alpha)})([\theta_l])\\
		&\geq r C s_\alpha(J^{(\alpha)})([\theta_s] ) \left(\min_{\Gamma\backslash\Xi_{\preceq \alpha}}(s_\alpha(J^{(\alpha)}))\right)^{l-1}.
	\end{align*}

	If there is no edge with source $[\theta_s]$ and aim $[\theta_a]$, then by definition of $\preceq$ there must exists a path with source $[\theta_s]$ and aim $[\theta_a]$. An immediate induction on the path length gives the desired result.
\end{proof}

\begin{Cor}
	Let $\beta\preceq\alpha$ be elements of $\Lambda$. Taking $J=\pi_{\alpha}(v_r)= v_r^{(\alpha)}+v_r^{(\prec \alpha)}$, in the previous lemma we get for any $0<r\leq R_\alpha$
	\begin{equation}\label{The_previous_equality_1}
		v_r([\theta_a])\geq r^e C v_r([\theta_s])\left(\min_{\Gamma\backslash\Xi_{\preceq \alpha}}\{v_r\}\right)^{\kappa}.
	\end{equation}
\end{Cor}
\begin{Rem}
In particular, if $\alpha \prec \infty$ is an element of $\Lambda$, with the notations of the previous corollary, for any $\beta\preceq \alpha$, since the maps $v^{(\beta)}$ are rational maps (Lemma \ref{Asymp-Rationality_of_g_xi_<_infty} in \cite{Asymp}), singular points of $v^{(\beta)}$ are poles, we have that $\lim_{r\to R_\beta^-} v^{(\beta)}(r)=\infty$. We find again through the previous inequality (\ref{The_previous_equality_1}) that we must have $$R_\alpha\leq R_\beta.$$
\end{Rem}

\begin{Prop}~
	For any $\alpha\in\Lambda$ such that $\alpha\prec\infty$, and for any complex number $z\in \mathbb{D}(0,R_\alpha)$, the map $J^{(\alpha)}\mapsto\psi_\alpha(z,J^{(\alpha)})$ is affine on $E_\alpha$. We denote by $\vec{\psi_\alpha}(z)$ its linear part.
	
	We have the following formula for the rational map $z\mapsto v_z^{(\alpha)}$:
\begin{equation}\label{forme_fractionelle_de_v_alpha}
	v^{(\alpha)}_z=\left(I-z\vec{\psi}_\alpha(z)\right)^{-1}(z \psi_\alpha(z,0_\alpha)).
\end{equation}
\end{Prop}
\begin{proof}
	Fix $\theta$ in $\Xi_\alpha$. By Corollary \ref{Cor_pre_Prop_linearity},  we have that $J\mapsto\psi(J)(\theta)$ is affine with respect to the variables $\{J(\theta_\alpha)\}_{\theta_\alpha\in\Xi_\alpha}$. It follows by definition of $\psi_\alpha$ (defined in (\ref{Def_de_psi_alpha_induction})), that $\psi_\alpha$ is also affine with respect to the variables $\{J([\theta_\alpha])\}_{[\theta_\alpha]\in\Gamma\backslash \Xi_\alpha}$.

Now since $(I-0\cdot\vec{\psi}_\alpha(0))=I$, we have in a neighbourhood of $z=0$, the formula \ref{forme_fractionelle_de_v_alpha}, which generalizes to any complex number $z\in\mathbb{C}$.
\end{proof}

\subsection{Spectrum of the differentials of \texorpdfstring{$\psi$}{Psi}}
	In this subsection we study the spectrum of the differential $D_{v_z}\psi$ of $\psi$ at $v_z$, for $z\in\overline{\mathbb{D}}(0,R)$. Note that from Proposition \ref{Graduation_de_V} and definition of $\preceq$ the differential of $\psi$ at any vector $J\in E$ is block triangular, that is, it has the matrix form:  
$$
	D_{v_z}\psi=\begin{pmatrix}
\vec{\psi}_{\alpha_1}(z) &0&\dots&0\\
\ast  &\ddots&\ddots&\vdots\\
\vdots&\ddots&\vec{\psi}_{\alpha_N}(z)&0\\
\ast  &\dots &\ast&\dfrac{\partial \psi_\infty}{\partial J^{(\infty)}}(z,v_z)
\end{pmatrix}.$$

	We prove in Corollary \ref{Cor_Spectrum_alpha}, that any non-zero complex number $z\in\mathbb{D}(0,R_{<\infty})$ satisfies:
$$\frac{1}{z}\notin Spec(\vec{\psi_\alpha}(z)),$$
	for any $\alpha\prec\infty$.
	So that if $z\in\overline{\mathbb{D}}(0,R)$, $z\neq 0$ has its inverse $\frac{1}{z}$ in the Spectrum of $D_{v_z}\psi$, then $\frac{1}{z}$ must belong to $Spec\left(\frac{\partial\psi_\infty}{\partial J^{(\infty)}}(v_z^{(\infty)})\right)$.
	
	We also characterize in Proposition \ref{Prop_Spectrum_belonging_A}, the non-zero complex number $z\in\overline{\mathbb{D}}(0,R)$, such that its inverse $\frac{1}{z}$ is in the spectrum of $\frac{\partial\psi_\infty}{\partial J^{(\infty)}}(v_z^{(\infty)})$.
	
These points will be useful to prove the regularity of $\mathcal{C}$ over $\{u(z): z \in \overline{\mathbb{D}(0,R)}\}$.
	
	We use the notations from the definitions \ref{Def_sets_Xi_alpha_spaces_E_alpha} and \ref{Def_v_alpha_R_alpha}. So we have
	$$R_{<\infty}:=\min_{\alpha\prec\infty}(R_\alpha).$$
	From Lemma \ref{Asymp-Rationality_of_g_xi_<_infty} in \cite{Asymp} the map $v^{(<\infty)}$ is rational over $\mathbb{C}$ with values in the complex finite dimension vector space $E_{<\infty}$. And for $z \neq R_{<\infty}$, 
	$$\lim_{z\to R_{<\infty}} \sup_{\Xi_{<\infty}}|v^{(<\infty)}_z|=\infty.$$
	
	Recall that by definition, for any $z\in \mathbb{D}(0,R_{<\infty})$, 
	$$\forall (a,b)_y\in\Xi_{<\infty},\, v^{(<\infty)}_z((a,b)_y)=G_z(a,b;\mathcal{B}(y)^\complement).$$
	
	Thus by finiteness of the Green's functions at $z=R$ (See \cite{Woess_2000} Chapter 2) and since $R_\infty \geq R$ (see remark \ref{Rem_R_infini_is_min}), we get the inequalities:
	\begin{equation}
		R\leq R_\infty < R_{<\infty}.
	\end{equation}
In particular we have the following characterisation of the radius of convergence of restricted Green's functions, that we state as a remark and will not be used in the sequel:
\begin{Prop}\label{Charac_radius}
Let $(a,b)_y$ be in $\Xi$ and denote by $R(a,b;\mathcal{B}(y)^\complement)$ the radius of convergence of the restricted Green's function $z\mapsto G_z(a,b;\mathcal{B}(y)^\complement)$. We have the strict inequality $R(a,b;\mathcal{B}(y)^\complement)>R$ if, and only if, there exists a non-negative integer $N$ such that, any $p$-admissible path $\gamma=(\omega_0,...,\omega_n)$ in $\mathcal{P}h(a,b;\mathcal{B}(y)^\complement)$ verifies:

$$ \max_{0\leq i \leq n} d(\omega_i,y)\leq N.$$
\end{Prop}
\begin{proof}
We've just proven the reciprocal direction ($(a,b)_y\in\Xi_{<\infty}\Rightarrow R(a,b;\mathcal{B}(y)^\complement)>R$). The direct one will be a consequence of the equalities $R_\infty=R(a,b;\mathcal{B}(y)^\complement)$ for any $(a,b)_y\in\Xi_\infty$, and of the equality $R=R_\infty$.
\end{proof}

For any $\alpha\prec\infty$, from formula (\ref{forme_fractionelle_de_v_alpha}), we can say that any singularity $z\neq 0$ of the rational map $v^{(\alpha)}$ is either a singularity of $v^{(\beta)}$ for some $\beta\prec\alpha$, or has its inverse in the spectrum of the linear map $\vec{\psi}_\alpha(z)$,
\begin{equation}\label{belong_spectrum_alpha}
\frac{1}{z}\in Spec \left(\vec{\psi}_\alpha(z) \right).
\end{equation}

We deduce that if $z\in\mathbb{C}\setminus\{0\}$ is a singularity of $v^{(<\infty)}$, then its inverse $\frac{1}{z}$ is in the spectrum of some $\vec{\psi}_\alpha(z)$, for some $\alpha\prec\infty$. 

In particular, by the definition of $R_{<\infty}$, we know that $v^{(<\infty)}$ does not possess any pole in $\mathbb{D}(0,R_{<\infty})$. We would like to assert that this implies that, $\frac{1}{z}\notin Spec(\vec{\psi}_\alpha(z))$ for any non-zero $z\in\mathbb{D}(0,R_{<\infty})$. However, this may not be true in general. What allows us to conclude is the Lemma:

\begin{Lem}\label{Lem_7.12}
	Let $\alpha\in\Lambda$ be such that $\alpha\prec\infty$. If $\vec{\psi}_\alpha$ is not identically $0$, then for any real number $r\in ]0,R_{\prec\alpha}[$, the operator $\vec{\psi}_\alpha(r): E_\alpha\to E_\alpha$ is Perron irreducible.
\end{Lem}
\begin{Rem}
Suppose that for any complex number $z\in\mathbb{D}(0,R_\alpha)$, the linear operator $\vec{\psi}_\alpha(z)$ on $E_\alpha$ is zero. By definition, there cannot be any edges in $\mathcal{V}$ between vertices of $C_\alpha$, otherwise $\vec{\psi}_\alpha(z)$ would not be zero. Since $C_\alpha$ is a connected component, it possesses one, and only one element, that does not admit a loop edge (i.e an edge in $\mathcal{V}$ with its source equals to its aim). In this case $v^{(\alpha)}$ is a polynomial function of $z$ and $v^{(\beta)}$ for $\beta\prec\alpha$.
\end{Rem}
\begin{proof}
	We identify $E_\alpha$ with the finite-dimensional vector space $\mathcal{F}(\Gamma\backslash\Xi_\alpha,\mathbb{C})$, endowed with the basis $(1_{[\theta] })_{[\theta]\in\Gamma\backslash\Xi_\alpha}$. Recall that $\vec{\psi}_\alpha(r)$ is an irreducible Perron operator if there exists an integer $N$ such that for any $[\theta_s] ,[\theta_a]\in C_\alpha$, we have  (see Definition \ref{Asymp-Def_Perron_operator} in the Appendix of \cite{Asymp}):
	\begin{equation}\label{equation_temp_2b}
		\sum_{n=0}^N \vec{\psi}_\alpha^{\circ n}(z)(1_{[\theta_s] })([\theta_a] ) >0.
	\end{equation}

	We fix $J^{(\alpha)}\in E_\alpha$ strictly positive over $\Xi_\alpha$. Since $\vec{\psi}_\alpha(z)$ is linear, we have that $D_{J^{(\alpha)}}(\vec{\psi}_\alpha(z)):T_{J^{(\alpha)}}E \to T_{\psi_\alpha(z,J^{(\alpha)})}E$ equals $\vec{\psi}_\alpha(z)$, when we identify the tangent spaces of $E_\alpha$ at a point to $E_\alpha$ itself.
	Hence:
\begin{equation}\label{eqts1}
	 \vec{\psi}_\alpha(z)=D_{J^{(\alpha)}}(\vec{\psi}_\alpha(z))=\pi_\alpha \circ D_{s_\alpha(z)(J^{(\alpha)})}\psi \circ D_z s_\alpha(z),
 \end{equation}
 	with $D_z s_\alpha$ identified with the inclusion $E_\alpha \hookrightarrow E$.
	For all $([\theta_s] ,[\theta_a]) \in C_\alpha\times C_\alpha$:
	\begin{equation}\label{eq-t4}
		\vec{\psi}_\alpha(z)(1_{[\theta_s] })([\theta_a] )=(\pi_\alpha \circ D_{s_\alpha(z)(J^{(\alpha)})}\psi)(1_{[\theta_s] })([\theta_a] )=D_{s_\alpha(z)(J^{(\alpha)})}\psi(1_{[\theta_s] })([\theta_a] ).
	\end{equation}
	Finally, for each pair $([\theta_s] ,[\theta_a]) \in C_\alpha\times C_\alpha$, by connectedness of $C_\alpha$, let $N([\theta_s] ,[\theta_a] )$ be such that there exists a path of length $N([\theta_s] ,[\theta_a] )$ in $\mathcal{V}$ from $[\theta_s] $ to $[\theta_a] $. Then set $N:=\max_{([\theta_s] ,[\theta_a] )}(N([\theta_s] ,[\theta_a] ))$.
	Given the above equalities (\ref{eq-t4}), Proposition \ref{Prop_Chemin_et_derivee} immediately gives (\ref{equation_temp_2b}). Hence $\vec{\psi}_\alpha(r)$ is an irreducible Perron operator.
\end{proof}
Note that the previous proof adapts to the case $\alpha=\infty$, taking $\frac{\partial\psi_\infty}{\partial J^{(\infty)}}(z,J^{(\infty)})$ instead of $\vec{\psi}_\alpha$ and replacing equality (\ref{eqts1}), by its analogue, in this case: 
\begin{equation}
\frac{\partial\psi_\infty}{\partial J^{(\infty)}}(z,J^{(\infty)}) =\pi_\infty\circ D_{s_\infty(z)(J^{(\infty)})}\psi \circ D_z s_\infty(z).
\end{equation}
We thus have
\begin{Cor}\label{Cor_partial_psi_infini_is_Perron_Irred}
For any real number $0<r<R_{<\infty}$, and any positive function $J^{(\infty)}\in E_\infty$, the operator $\frac{\partial\psi_\infty}{\partial J^{(\infty)}}(r,J^{(\infty)})$ is an irreducible Perron operator. 
\end{Cor}

\begin{Cor}\label{Cor_Spectrum_alpha}
	For any $\alpha\prec \infty$ in $\Lambda$, for any non-zero complex number $z\in\mathbb{D}(0,R_\alpha)$, its inverse $\frac{1}{z}$ is not in $Spec\left(\vec{\psi}_\alpha(z)\right)$.
\end{Cor}
\begin{proof}
	We prove this result by induction on $\alpha\prec \infty$, with respect to $\preceq$. 
	
	Let $\alpha$ be in $\Lambda$ and suppose that the result is verified for any $\beta \prec \alpha$. By quantifying on the null set, this property is trivially satisfied when $\alpha$ is a minimal element of $\Lambda$.
	
	If $\vec{\psi}_\alpha\equiv 0$ is the null function, then the result is trivially true. Else, $\vec{\psi}_\alpha$ is not identically zero, and in particular $r\mapsto \rho\left(\vec{\psi}_\alpha(r)\right)$ is a non-decreasing function of $r\in [0,R_{\alpha}[$ (See Lemma \ref{inequality_spectral_radius} in the Appendix).
	Suppose there is some non-zero complex number $z\in\overline{\mathbb{D}(0,R_\alpha)}$ such that its inverse $\frac{1}{z}$ is in $Spec(\vec{\psi}_\alpha(z))$, we prove that necessarily $|z|=R_\alpha$.  By assumption, on $z$, we have
	$$\frac{1}{|z|}\leq \rho\left(\vec{\psi}_\alpha(z)\right)\leq \rho\left(\vec{\psi}_\alpha(|z|)\right).$$ 
	Using the intermediate values theorem, we deduce that there must exist a positive real number $0<r_0\leq|z|$ such that $\frac{1}{r_0}=\rho\left(\vec{\psi}_\alpha(r_0)\right)$. 
Note that we actually need the continuity of $r\mapsto \rho(\vec{\psi}_\alpha(r))$. To prove it
	Let $n\in\mathbb{N}$ be a non-negative integer that is fixed. The map that associates to any complex polynomial of degree $n$, its set of zeros is a continuous map in the coefficients of the polynomial (See \cite{Whitney_1972} App. V.4, page 363)\footnote{We actually don't need such a strong result, since the spectral radius of $\vec{\psi}_\alpha (r)$, for $r>0$, is a simple root of the characteristic polynomial associated with the operator, a variant of Rouché's Theorem gives the continuity of the function $r\mapsto \rho(\vec{\psi}_\alpha(r))$.}. Since the map $r\mapsto \vec{\psi}_\alpha(r)$ is clearly continuous, the map $r\mapsto \rho(\vec{\psi}_\alpha(r))$ is also continuous.

	For any real number $0<r<R_\alpha$, the vector $r\psi_\alpha(r,0_\alpha)$ is a non-negative function in $E_\alpha$, and $\vec{\psi}_\alpha(r)$ is Perron irreducible, thus we must have
	$$\lim_{r\to r_0}v^{(\alpha)}(r)=\lim_{r\to r_0}\left(I-r\vec{\psi}_\alpha(r)\right)^{-1}(r \psi_\alpha(r,0_\alpha))=\infty.$$
	Indeed, to obtain the above limit, it suffices to check that $r_0 \psi_\alpha(r_0,0_\alpha)$ is not in the kernel of the adjugate matrix of $\left(I-r_0\vec{\psi}_\alpha(r_0)\right)$. It can be verified\footnote{
	
Since $\left(I-r_0\vec{\psi}_\alpha(r_0)\right)$ is a matrix with real coefficients, its adjugate also has real coefficients. From the formula 
$$\left(I-r_0\vec{\psi}_\alpha(r_0)\right)~^t Adj\left(I-r_0\vec{\psi}_\alpha(r_0)\right)
=~^t Adj\left(I-r_0\vec{\psi}_\alpha(r_0)\right)\left(I-r_0\vec{\psi}_\alpha(r_0)\right)=0,$$
we deduce that any row of the adjugate matrix $Adj\left(I-r_0\vec{\psi}_\alpha(r_0)\right)$ is either zero, or the transpose of a real eigenvector of $r_0\vec{\psi}_\alpha(r_0)$, associated with the eigenvalue $1$. Since $r_0\vec{\psi}_\alpha(r_0)$ is an irreducible Perron operator, its real eigenvectors associated with the eigenvalue $1$ have all its coefficients positive, or all its coefficients negative. The same treatment can be applied to the columns. Thus, if we prove that the adjugate matrix possesses at least one positive (resp. negative) coefficient, then all its coefficients are positive (resp. negative). This last point is immediate since the rank of the matrix $\left(I-r_0\vec{\psi}_\alpha(r_0)\right)$, is $n-1$, so the rank of its adjugate must be $1$, and thus $Adj\left(I-r_0\vec{\psi}_\alpha(r_0)\right)$ can not be the null matrix. 
} that the adjugate matrix of $\left(I-r_0\vec{\psi}_\alpha(r_0)\right)$ has all its coefficients positive, or all its coefficients negative. Since for any $r_0>0$, the vector $r_0 \psi_\alpha(r_0,0_\alpha)$ in $E_\alpha$ has positive coordinates, it cannot be in the kernel of the adjugate matrix of $\left(I-r_0\vec{\psi}_\alpha(r_0)\right)$.
	
	It follows that necessarily $r_0=R_\alpha\leq |z|$. The result follows for $\alpha$. The induction is complete and the Corollary follows.
\end{proof}

\begin{Prop}\label{Prop_Spectrum_belonging_B}
	For any non-zero complex number $z\in\overline{\mathbb{D}(0,R)}$, its inverse $1/z$ is in $Spec\left(D_{v_z}\psi\right)$ if, and only if $1/z$ is in $Spec\left(\frac{\partial\psi_\infty}{\partial J^{(\infty)}}(z,v_z^{(\infty)})\right)$.
	
	In this case if $\Lambda_{1/z}$ is the associated generalized eigenspace associated with the eigenvalue $\frac{1}{z}$ of the operator $D_{v_z}\psi$, and $\Lambda_{1/z}(\infty)$ is the one associated with $\frac{\partial \psi_\infty}{\partial J^{(\infty)}}(z,v_z^{(\infty)})$, then
		$$ \Lambda_{1/z}=\{0_{<\infty}\}\times \Lambda_{1/z}(\infty).$$
\end{Prop}
\begin{proof}
From Proposition \ref{Graduation_de_V}, we infer that the differential $D_{v_z}\psi$ is a block triangular matrix, of the form 
	\begin{equation}\label{Form_of_D_J_psi}
	D_{v_z}\psi=\begin{pmatrix}
\vec{\psi}_{\alpha_1}(z) &0&\dots&0\\
\ast  &\ddots&\ddots&\vdots\\
\vdots&\ddots&\vec{\psi}_{\alpha_N}(z)&0\\
\ast  &\dots &\ast&\dfrac{\partial \psi_\infty}{\partial J^{(\infty)}}(z,v_z)
\end{pmatrix},
	\end{equation}
	where if $\alpha_i\prec\alpha_j$, then $i<j$. From Corollary \ref{Cor_Spectrum_alpha}, for any $\alpha\prec\infty$, and any non-zero complex number $z\in\mathbb{D}(0,R_{<\infty})$, its inverse $1/z$ is not in the spectrum $Spec\left(\vec{\psi}_\alpha(z)\right)$, we deduce that $1/z$ is in $Spec\left(D_{v_z}\psi\right)$ if, and only if $1/z$ is in $Spec\left(\frac{\partial\psi_\infty}{\partial J^{(\infty)}}(z,v_z^{(\infty)})\right)$. And in this case, if $\dot{J}=(\dot{J}^{(<\infty)},\dot{J}^{(\infty)})\in E_{<\infty}\times E_\infty=E$ is an associated generalized eigenvector of $D_{v_z}\psi$ for the eigenvalue $1/z$, then $\dot{J}^{(<\infty)}=0$ and $\dot{J}^{(\infty)}$ is a generalized eigenvector of $\frac{\partial\psi_\infty}{\partial J^{(\infty)}}(z,v_z^{(\infty)})$. Hence the proposition.
\end{proof}

\begin{Cor}\label{1surR=rho_D_v_R_psi}
    We have the equality:
        $$\frac{1}{R}=\rho(D_{v_R}\psi).$$
\end{Cor}
\begin{proof}
Note that if $\frac{1}{R}$ is not in the spectrum $Spec\left( D_{v_R}\psi \right)$, then since $\frac{\partial W}{\partial J}(R,v_R)=I- R D_{v_R}\psi$, where $W:(z,J)\mapsto J-z\psi(J)$ is the polynomial map defining the Lalley's curve, using the complex Implicit Functions Theorem, the function $z\mapsto v_z$ admits a holomorphic extension in the neighbourhood of $z=R$, which cannot be. Thus $\frac{1}{R}$ belongs to $Spec\left( D_{v_R}\psi \right)$. In particular $\frac{1}{R}\leq \rho(D_{v_R}\psi )$. Then using the intermediate values theorem there must exists $0<r_0\leq R$, such that $1/r_0=\rho(D_{v_{r_0}}\psi)$.

From Proposition \ref{Prop_Spectrum_belonging_B}, we have $1/r_0=\rho\left(\frac{\partial\psi_\infty}{\partial J^{(\infty)}}(z,v_z^{(\infty)})\right)$

Differentiating the equation $z\psi_\infty(z,v_z^{(\infty)})=v_z^{(\infty)}$ with respect to $z$, we get for any real number $0<r<r_0$:
\begin{equation}\label{eq4251}
\frac{d v_r^{(\infty)}}{dz}=\left(I-r\left(\frac{\partial\psi_\infty}{\partial J^{(\infty)}}(r,v_r^{(\infty)})\right)\right)^{-1}\underbrace{\left(\psi_\infty(r,v_r^{(\infty)})+r\left(\frac{\partial\psi_\infty}{\partial z}(r,v_r^{(\infty)})\right)\right)}_{=:q(r)}.
\end{equation}

Using the non-decreasing of $\psi_\infty$ with respect to $z\in[0,R]$, we have that $q(r)\in E$ is a positive vector with positive coordinates in $E_\infty$. Thus, since $\left(\frac{\partial\psi_\infty}{\partial J^{(\infty)}}(z,v_z^{(\infty)})\right)$ is an irreducible Perron operator, we obtain\footnote{Using the same method as in the proof of Corollary \ref{Cor_Spectrum_alpha}} that the right hand side of the above equality (\ref{eq4251}) goes to $+\infty$, as $r$ goes to $r_0$. We deduce that $r_0=R$.
\end{proof}

\begin{Prop}\label{Prop_Spectrum_belonging_A}
	Let $z\in\overline{\mathbb{D}(0,R)}$ be a non-zero complex number, we have the equivalence:
	\begin{equation}\label{Implications_Prop_Spectrum_belonging_A}
	 \frac{1}{z}\in Spec(D_{v_z}\psi )\Leftrightarrow z^d=R^d,
	 \end{equation}
	 where $d=\gcd\{n: p^{(n)}(x,x)\}$ is the period of the Markov chain $(X_0,p)$.
	 
	Besides, in this case we have $1/R=\rho(D_{v_z}\psi)$ and $\Lambda_{1/(\zeta_d z)}=A\Lambda_{1/z}$ where $\Lambda_{1/z}$ correspond to the characteristic subspace associated with $D_{v_z}\psi$ and the eigenvalue $1/z$, $A$ is a diagonal operator of order $d$, defined in (\ref{Asymp-Def_of_A}) in \cite{Asymp} and $\zeta_d=\exp\left(\frac{2i\pi}{d}\right)$.
\end{Prop}
To prove this Proposition, we will need the two following Lemmas:
\begin{Lem}\label{eqts8}
For any non-zero complex number $z$ in $\mathbb{D}(0,R)$, we have
\begin{equation}
	\forall \theta\in\Xi_\infty, |v_z(\theta)|\leq v_{|z|}(\theta),
\end{equation}
with equality if, and only if, $z^d=|z|^d$.
\end{Lem}
\begin{proof}
From Corollary \ref{Asymp-Cor_Admissible_Path_Lenght} in the Appendix of \cite{Asymp}, the length of $p$-admissible paths in the set $\mathcal{P}h(a,b;\mathcal{B}(y)^\complement)$ are in a set $r(a,b) + d\mathbb{N}$, where $r:X_0\times X_0\to\mathbb{Z}/d\mathbb{Z}$ is the periodicity cocycle associated with the Markov chain $(X_0,p)$; denote it $r(a,b)+\mathcal{N}_\theta$, for $\theta=(a,b)_y\in\Xi_\infty$. Let $d_\theta:=gcd(\mathcal{N}_\theta)$ and let $\sum_n a_n z^n$ be the power series expansion in a neighbourhood of $z=0$, of the function $z\mapsto v_z(\theta)$. Using the triangular inequality we have the equivalence:
$$|v_z(\theta)|=\left|\sum_{n\in\, r+\mathcal{N}_\theta} \underbrace{a_n}_{> 0} z^n \right|= \sum_{n\in\, r+\mathcal{N}_\theta} a_n |z^n|=v(|z|)(\theta),$$
if, and only if $z^{d_\theta}=|z|^{d_\theta}$. To conclude, we need to show that $d_\theta=d$.

Let $\gamma,\gamma'$ be two $p$-admissible paths from $\mathcal{P}h(y,y)$ such that $l(\gamma)=d+l(\gamma')$. Such $p$-admissible paths exists for $\{l(\gamma) : \gamma \in \mathcal{P}h(y,y)\}$ is an additive subset of $\mathbb{N}$, with gcd equals to $d$. Let $M>0$ be such that the vertices of $\gamma$ and $\gamma'$ are in the ball $\mathcal{B}_M(y)$, and let $y'=gy$ for some $g\in\Gamma$, be a reachable vertex for $(a,b)_y$ at distance $>M+k$ from $y$. For $\gamma_1\in\mathcal{P}h(a,y';\mathcal{B}(y)^\complement)$, $\gamma_2\in\mathcal{P}h(y',b;\mathcal{B}(y)^\complement)$, the length of the two concatenated $p$-admissible paths $\gamma_1\ast g\gamma\ast\gamma_2$ and $\gamma_1\ast g\gamma'\ast\gamma_2$, are in the set $r+\mathcal{N}_\theta$. In particular $d_\theta$ divides $l(\gamma_1)+l(\gamma)+l(\gamma_1)-(l(\gamma_1)+l(\gamma')+l(\gamma_2))=d$. Hence $d_\theta=d$.
\end{proof}
\begin{Lem}\label{The_second_inequality}
	For any non-zero complex number $z\in\overline{\mathbb{D}}(0,R)$, we have the following inequality
	$$	\rho \left( \frac{\partial \psi_\infty}{\partial J^{(\infty)}}(|z|,|v_z^{(\infty)}|)\right)\leq
		\rho \left( \frac{\partial \psi_\infty}{\partial J^{(\infty)}}(|z|,v_{|z|}^{(\infty)})\right),
	$$
	with equality if, and only if, $z^d=|z|^d$.
\end{Lem}
\begin{proof}
Let $z\in\overline{\mathbb{D}}(0,R)$ be a non-zero complex number. Note that the matrix representation of the two operators $ \frac{\partial \psi_\infty}{\partial J^{(\infty)}}(|z|,|v_z^{(\infty)}|)$ and $\frac{\partial \psi_\infty}{\partial J^{(\infty)}}(|z|,v_{|z|}^{(\infty)})$ are matrix with non-negative coefficients, denote the first one $A$ and the second one $B$. From Proposition \ref{degre_infini_de_psi}, for any fixed non-zero complex number $z\in\overline{\mathbb{D}}(0,R)$, the map $J^{(\infty)}\mapsto \psi_\infty(|z|,J^{(\infty)})$ is a polynomial map of degree at least $2$ with non-negative coefficients.
In particular, for any $J^{(\infty)}\in E_\infty$ with positive coefficients, all the partial derivatives of order two at $J^{(\infty)}$, of this polynomial map are non-negative, and some of its partial derivatives are non-zero since $J^{(\infty)}\mapsto \psi_\infty(|z|,J^{(\infty)})$ is not affine.

As a consequence, we get the strict increasing of certain (not necessarily all) coefficients of the matrix representation of the partial derivative $\frac{\partial \psi_\infty}{\partial J^{(\infty)}}(|z|,J^{(\infty)})$, in $J^{(\infty)}$. In particular from Lemma \ref{eqts8}, some of the coefficients $A_{i,j}$ in $A$ are strictly lower than their homologous $B_{i,j}$ in $B$.
Then applying Lemma \ref{inequality_spectral_radius_Perron} from the Appendix, we get the result.
\end{proof}
\begin{proof}[Proof of Proposition \ref{Prop_Spectrum_belonging_A}]
From Proposition \ref{1surR=rho_D_v_R_psi} the real number $\frac{1}{R}$ belongs to $Spec\left( D_{v_R}\psi \right)$.
 Recall formula:
$$\zeta_d \cdot  \psi\circ A = A \circ \psi,$$
already seen in \cite{Asymp} equation (\ref{Asymp-Relation_psi_Z_d}). Differentiating the above, we get 
\begin{equation}
\zeta_d D_{v_R}\psi \circ A = A \circ D_{v_R}\psi,
\end{equation}
hence for any $l\in\{1,...,d-1\}$, $\frac{1}{(\zeta_d^l R)}\in Spec\left( D_{v_{(\zeta^l R)}}\psi \right) $ and $\Lambda_{1/(\zeta_d^l R)}=A^l\Lambda_{1/R}$. We have proven the indirect implication of (\ref{Implications_Prop_Spectrum_belonging_A}).

We now prove the direct implication. Since for any complex number $z\in\overline{\mathbb{D}(0,R)}$, for any $\theta\in\Xi$ we have $|v_z(\theta)|\leq v_{|z|}(\theta)$ and since $\psi$ is a polynomial map over $E$ with non-negative coefficients, we have that the matrix representation of $D_{v_z}\psi$ and $D_{v_{|z|}}\psi$ satisfy coefficientwise, in module,
$D_{v_z}\psi \leq D_{v_{|z|}}\psi$. Thus using Lemma \ref{inequality_spectral_radius}, we get:
$$\rho \left( D_{v_z}\psi\right)\leq \rho\left(D_{v_{|z|}}\psi \right).$$

Besides the non-negative function $r\mapsto \rho \left( D_{v_r}\psi\right)$ is non-decreasing, since the functions $r\mapsto v_r(\theta)$ for $\theta\in\Xi$ are non-decreasing. So by Corollary \ref{1surR=rho_D_v_R_psi},
$$ \forall z\in\overline{\mathbb{D}(0,R)}, \, 
\rho \left( D_{v_z}\psi\right)\leq \rho \left( D_{v_R}\psi\right) = \frac{1}{R}.
$$
We deduce that for any complex number $z$ of modulus $|z|<R$, the spectral radius $\rho \left( D_{v_z}\psi\right)$ is $\leq\frac{1}{R}<\frac{1}{|z|}$, thus $\frac{1}{z}$ cannot be in the spectrum $Spec\left( D_{v_z}\psi\right)$.

Lastly for $z\in\mathbb{C}$ such that $|z|=R$ and $z^d\neq R^d$ then using Lemma \ref{eqts8} and Lemma \ref{The_second_inequality}, we have
	$$\rho \left( \frac{\partial \psi_\infty}{\partial J^{(\infty)}}(z,v_z^{(\infty)})\right)\leq
	 \rho \left( \frac{\partial \psi_\infty}{\partial J^{(\infty)}}(R,|v_z^{(\infty)}|)\right)<
	 \rho \left( \frac{\partial \psi_\infty}{\partial J^{(\infty)}}(R,v_{R}^{(\infty)})\right),
	 $$
	 where the first inequality is a consequence of the non-negativity of the coefficients of the polynomial map $\psi$ and the inequalities $|v_z(\theta)|\leq v_{R}(\theta)$ for any $\theta\in\Xi$.
\end{proof}

\begin{Prop}\label{Spectrum_u(R)}
		The real number $\frac{1}{R}$ belongs to the spectrum of the operators $D_{v_R}\psi$ and $\frac{\partial \psi_\infty}{\partial J^{(\infty)}}(v_R)$.
	Denote by $\Lambda_{1/R}$ and $\Lambda_{1/R}(\infty)$ their respective generalized eigenspace associated with the eigenvalue $1/R$. Then we have 		
		 $$\Lambda_{1/R}=\{0_{<\infty}\}\times\Lambda_{1/R}(\infty)\text{ and }\Lambda_{1/R}(\infty)=\mathbb{C}\cdot \nu_\infty,$$
where $\nu_\infty$ is a strictly positive real vector in $E_\infty$.
\end{Prop}
\begin{proof}
The first part of the lemma is a consequence of propositions \ref{Prop_Spectrum_belonging_B} and \ref{Prop_Spectrum_belonging_A}. Also, by Proposition \ref{Prop_Spectrum_belonging_B}, to compute the dimension of $\Lambda_{1/R}$,  we just need to compute the one of $\Lambda_{1/R}(\infty)$. It has dimension $1$ and contains a strictly positive real vector $\nu_\infty$, since $\frac{\partial \psi_\infty}{\partial J^{(\infty)}}(R,v_R^{(\infty)})$ is an irreducible Perron operator (Corollary \ref{Cor_partial_psi_infini_is_Perron_Irred}).
\end{proof}

\subsection{Regularity of \texorpdfstring{$\mathcal{C}$}{the Lalley's curve}}\label{subsection_regularity_of_C}

We prove the following proposition
\begin{Prop}\label{Regularity_of_C}
	The Lalley's curve $\mathcal{C}$ is a smooth analytic curve in a neighbourhood of the closed subset $\overline{\{(z,v_z): z\in\mathbb{D}(0,R)\}}\subset \mathcal{C}$. 	
\end{Prop}

We start by stating two lemmas:
\begin{Lem}\label{Regularity_Lemma_A}
	Let $(\tau,J)\in\mathcal{C}$ with $\tau\neq 0$. 
	
	If $\frac{1}{\tau}$ is not in the spectrum $Spec \left(D_{J}\psi\right)$ of the differential $D_{J}\psi$ of $\psi$ at $J$, then $\mathcal{C}$ is smooth in the neighbourhood of $(\tau,J)\in\mathcal{C}$ and the map $\lambda:\mathcal{C}\to\mathbb{C}$ admits a holomorphic inverse in a neighbourhood of $(\tau,J)$.
\end{Lem}
\begin{proof}
Recall that the Lalley's curve equals $\mathcal{C}=W^{-1}(0_E)$, where $W:\mathbb{C}\times E \to E $ is the map defined by 
$$W:(z,J)\mapsto J-z\psi(J). $$
At $(\tau,J)$, we have $W(\tau,J)=0$ and the result is a direct consequence of the complex Implicit Functions Theorem using the fact that the partial differential of $W$ at $(\tau,J)$ with respect to $J$ is
\begin{equation*}
\frac{\partial W}{\partial J}(\tau,J)= I-\tau D_{J}\psi ,
\end{equation*}
and is invertible by assumption.
\end{proof}

\begin{Lem}\label{Regularity_Lemma_B}
	Let $(\tau,J)\in\mathcal{C}$ with $\tau\neq 0$.
	
	If $\frac{1}{\tau}$ belongs to the spectrum $Spec \left(D_{J}\psi\right)$, if the associated characteristic subspace $\Lambda_{1/\tau}$ is of dimension $1$ and if $\psi(J)$ is not in the range $\left( I - \tau D_J\psi \right)(E)$ of the linear operator $\left( I - \tau D_J\psi \right)$ on $E$, then $\mathcal{C}$ is smooth in the neighbourhood of $(\tau,J)\in\mathcal{C}$ and its tangent space at $(\tau,J)$ is
	\begin{equation}\label{eqts3}
	    T_{(\tau,J)}\mathcal{C}=\{0\}\times \Lambda_{1/z}.
	\end{equation}
\end{Lem}
\begin{proof}
		Recall that the Lalley's curve $\mathcal{C}=W^{-1}(0_E)$ where $W:\mathbb{C}\times E \to E $ is the map defined by 
$$W:(z,J)\mapsto J-z\psi(J). $$
At $(\tau,J)$, we have $W(\tau,J)=0$ and we have that the differential of $W$ at $(\tau,J)$ is the linear map from $\mathbb{C}\times E$ to $E$, that associates to any pair $(\dot{z},\dot{J})\in\mathbb{C}\times E$,
$$ D_{(\tau,J)}W(\dot{z},\dot{J})=\left( I-\tau D_{J}\psi \right)(\dot{J})-\dot{z}\psi (J). $$

By the local submersion theorem (See \cite{Lee_2013} Rank Theorem 4.12), if $D_{(\tau,J)}W$ is surjective, then $\mathcal{C}$ is regular in a neighbourhood of $(\tau,J)$. This is equivalent to prove that the kernel of $D_{(\tau,J)}W$ is of dimension $1$. But $D_{(\tau,J)}W(\dot{z},\dot{J})=0$ if, and only if, the following equation is satisfied
\begin{equation}\label{eqts2}
	\left( I-\tau D_{J}\psi \right)(\dot{J})=\dot{z}\psi (J).
\end{equation}
Since $\frac{1}{\tau}\in Spec( D_{J}\psi)$ and $\psi(J)$ is not in the range of $(I-\tau D_J\psi)$, if the above (\ref{eqts2}) is satisfied, then we must have $\dot{z}=0$. Thus $\dot{J}$ must be an eigenvector of $D_{J}\psi $ associated with $1/\tau$. We deduce that 
$$ Ker(D_{(\tau,J)}W)=\{0\}\times \Lambda_{1/\tau},$$
which is indeed of dimension $1$, by assumption on the dimension of $\Lambda_{1/\tau}$. Hence $\mathcal{C}$ is regular on a neighbourhood of $(\tau,J)$. Besides its tangent space at $(\tau,J)$ equals the kernel of $D_{(\tau,J)}W$, hence (\ref{eqts3}) holds.
\end{proof}

\begin{Prop}\label{Prop_Spectrum_belonging_C}	
	The Lalley's curve $\mathcal{C}$ is smooth in a neighbourhood of $u(R)$ and its tangent space at $u(R)$ is
	$$T_{u(R)}\mathcal{C}=\{0\}\times \{0_{<\infty}\}\times \mathbb{C}\cdot \nu_\infty\subset \mathbb{C}\times E_{<\infty}\times E_\infty,$$
where $\nu_\infty$ is a real vector in $E_\infty$, with positive coefficients.
\end{Prop}
\begin{proof}

	From Proposition \ref{Prop_Spectrum_belonging_A}, $\frac{1}{R}$ is in $Spec \left(\frac{\partial \psi_\infty}{\partial J^{(\infty)}}(R,v_R^{(\infty)})\right)$. Since $\frac{\partial \psi_\infty}{\partial J^{(\infty)}}(R,v_R^{(\infty)})$ is an irreducible Perron operator, the generalized eigenspace $\Lambda_{1/z}(\infty)$ associated with the eigenvalue $\frac{1}{R}$ of $\frac{\partial \psi_\infty}{\partial J^{(\infty)}}(R,v_R^{(\infty)})$, has dimension $1$ (See Corollary \ref{Cor_partial_psi_infini_is_Perron_Irred} and Theorem \ref{Asymp-Thm_Perron_irred} in the Appendix of \cite{Asymp} about irreducible Perron operator). Besides the range of the operator$\left(I-R\frac{\partial \psi_\infty}{\partial J^{(\infty)}}(R,v_R^{(\infty)})\right)$ doesn't contain any non-zero real vector with non-negative coordinates, and in particular doesn't contain $\psi_\infty(R,v_R^{(\infty)})$, see again Theorem \ref{Asymp-Thm_Perron_irred} in the Appendix of \cite{Asymp}. 
	Then combining Lemma \ref{Prop_Spectrum_belonging_B} with the previous Lemma \ref{Regularity_Lemma_B}, we get that $\mathcal{C}$ is smooth at $u(R)=(R,v_R)$. 
	The result is then a consequence of the Proposition \ref{Spectrum_u(R)}.
\end{proof}
\begin{proof}[Proof of Proposition \ref{Regularity_of_C}]
	Combining Proposition \ref{Prop_Spectrum_belonging_A} with Lemma \ref{Regularity_Lemma_A} we get that $\mathcal{C}$ is smooth in a neighbourhood of any point $(z,v_z)\in\mathcal{C}$ such that $z\in\overline{\mathbb{D}(0,R)}$, $z^d\neq R^d$. 
	
    From The above Proposition \ref{Prop_Spectrum_belonging_C}, the Lalley's curve $\mathcal{C}$ is smooth at $u(R)=(R,v_R)$. 
    
	Now from subsection \ref{Asymp-Subsection_ZdZ_action} in \cite{Asymp}, the automorphism of order $d$, $A\in Aut(\mathcal{C})$ sends smooth point of $\mathcal{C}$ to smooth points of $\mathcal{C}$. Since the $A$-orbit of the smooth point $u(R)$, is $(z,v_z)$ such that $z$ is in $\{ R, Re^{2i\pi/d},...,R e^{2i \pi(d-1)/d}\}$, we get that $\mathcal{C}$ is smooth at these points. Insofar as the set of smooth points of $\mathcal{C}$ is open in $\mathcal{C}$ (Cf.\cite{Lee_2013} Theorem 4.12), $\mathcal{C}$ is smooth in a neighbourhood of $\{u(z): z\in\overline{\mathbb{D}(0,R)}\}$ as expected.
\end{proof}

Actually the first part of this proof gives:
\begin{Prop}\label{Prop_Holomorphic_Extension_of_u}
	The map $u:z\mapsto(z,v_z)$ admits a holomorphic extension in the neighbourhood of any complex number $z\in\partial\mathbb{D}(0,R)$ such that $z^d\neq R^d$.
\end{Prop}
\begin{Rem}
Donald I. Cartwright proved in 1992 \cite{Cartwright_1992}, under some hypotheses - that apply in our context - that the only singular points of the green's functions $z\mapsto G_z(x,y)$ on the boundary $\partial\mathbb{D}(0,R)$ are exactly the complex numbers $z$ such that $z^d=R^d$.
\end{Rem}

\vspace{\baselineskip}

\section{Derivatives of \texorpdfstring{$\lambda$}{lambda}, \texorpdfstring{$f_{x,y}$}{fxy} and \texorpdfstring{$g_{x,y}$}{gxy} at \texorpdfstring{$u(R)$}{le point u(R)}}\label{Section_derivatives}

The objectives of this section are to compute the differential, at $u(R)$, of the functions $\lambda:(z,v)\mapsto z$, $f_{x,y}$ and $g_{x,y}$ from respectively Proposition \ref{Asymp-algebraicity_of_F} and Proposition \ref{Asymp-Cor_algbraicity_of_Greens_function} in \cite{Asymp}, whom we recall the definition below. 

In Proposition \ref{Prop_derivative_of_lambda} we prove that the first derivative of the function $\lambda:(z,J)\mapsto z$, over $T_{u(R)}\mathcal{C}$ is zero. In Proposition \ref{Prop_derive_seconde_de_lambda}, we prove that the second derivative of $\lambda$ at $u(R)$ is non-zero, and we prove in Proposition \ref{Prop_derivative_of_g} that the first derivative of the functions $g_{x,y}$, at $u(R)$, for any $x,y\in X_0$, are non-zero. We also prove that for some functions $f_{x,y}$, the derivative of $f_{x,y}$ at $u(R)$ is also non-zero. We actually give, in Proposition \ref{Corollary_PSeries_A}, a characterisation of the non-nullity of $D_{u(R)}f_{x,y}$ in term of properties of the $p$-admissible paths in $\mathcal{P}h(x,y;\{y\}^\complement)$. 

As said, we start by briefly recalling the definitions of the functions $\lambda$, $g_{x,y}$ and $f_{x,y}$, for $x,y$ being vertices of the tree $X$:

The function $\lambda:\mathbb{C}\times E\to\mathbb{C}$ is the projection on the first coordinate: 
$$\lambda:(z,J)\mapsto z.$$

To any vertex $y$ of $X$, we associate a real vector $\bar{p}=\bar{p}(y)$ indexed by $\mathcal{B}(y)\setminus\{y\}\subset X_0$, and a map $M$ from the vector space $E$ to the space of square matrices indexed by the finite set $\mathcal{B}(y)\setminus\{y\}$: 

The vector denoted $\bar{p}$ has coefficients 
	\begin{equation}\label{coefficients_of_p}
	 	\forall d\in \mathcal{B}(y)\setminus\{y\},\; \bar{p}_d=p(d,y).
	\end{equation}

And the map $M$ associates to any vector $J\in E$ the matrix $M_J$ whose coefficients are given by
\begin{equation}\label{coefficients_of_M_J}
	\forall (c,d)\in\mathcal{B}(y)\setminus\{y\},\; M_J(c,d):= p(c,d) + \sum_{c_1\notin\mathcal{B}(y)} p(c,c_1)(J)_{[c_1,y]}(c_1,d).
\end{equation}

Remember that for any element $J\in E$, the complex number $J_{[c',y]}(c',d')\in \mathbb{C}$ is only defined (See formula (\ref{Asymp-Formula_J_[z,y]}) in \cite{Asymp}) if $d'$ belongs to $\mathcal{B}(y)$ and if $c'$ is not in  $\mathcal{B}(y)$.
We extend this notation by setting $J_{[c',y]}(c',d'):=1$ when $c'=d'\in\mathcal{B}(y)$.

Remark that $J\mapsto M_J$ is a polynomial map, so is the map $J\mapsto M_J^n$ where $n\in\mathbb{N}$ is a non-negative integer. Besides we have the formula: 
		\begin{equation}\label{coefficients_of_M_Jn}
			M_J^n(c,d)=\sum_{c_0,...,c_{n}}\sum_{\substack{c_1',...,c_{n}'\\ c_i'\in \{c_i\}\cup\mathcal{B}(y)^\complement}}\left(\prod_{i=1}^{n} p(c_{i-1},c_i')(J_{[c_i',y]}(c_i',c_i))\right),
		\end{equation}
where the first sum relates to all $c_0,...,c_{n}\in\mathcal{B}(y)\setminus\{y\}$ such that $c_0=c$ and $c_n=d$. 

\underline{Definition of $f_{x,y}$}

Let $x,y$ be vertices of $X$, the function $f_{x,y}$ is defined by:
\begin{itemize}
	\item If $x=y$, then $f_{x,y}$ is the constant function equal to $1$; 
	\item If $x$ is in $\mathcal{B}(y)\setminus\{y\}$, then for any $(z,J)\in\mathbb{C}\times E$, that is not a pole, $\left(I-z M_{J}\right)^{-1}(\bar{p})$ is a complex vector in $\mathbb{C}^{(\mathcal{B}(y)\setminus\{y\})}$, and $f_{x,y}(z,J)$ is its coefficient indexed by $x$:
	\begin{equation}\label{Def_of_f_x,y_when_x_in_Ball_y}
		\forall (z,J)\in\mathbb{C}\times E, f_{x,y}(z,J)=\underbrace{\left[\left(I-z M_{J}\right)^{-1}(\bar{p})\right]}_{\displaystyle \in \mathbb{C}^{\mathcal{B}(y)\setminus\{y\}}}(x),
	\end{equation}
	with $\bar{p}$ and $J\mapsto M_J$ be respectively the vector and the matrix map associated to $y$ with coefficients given by (\ref{coefficients_of_p}) and (\ref{coefficients_of_M_J});
	\item If $x\notin\mathcal{B}(y)$, $f_{x,y}$ is the rational function
	\begin{equation}\label{Def_of_f_x,y_when_x_not_in_Ball_y}
	     f_{x,y}:(z,J)\mapsto\sum_{c\,\in\mathcal{B}(y)\setminus\{y\}} (J)_{[x,y]}(x,c) f_{c,y}(z,J).
	\end{equation}
\end{itemize}

\underline{Definition of $g_{x,y}$}
	\begin{itemize}
		\item If $x=y$, $g_{x,y}$ is the rational function given by
		\begin{equation}\label{Def_of_g_x,y_x_equal_y}
		 	g_{y,y}=\left(1-\sum_{c}zp(y,c)f_{c,y}\right)^{-1};
		\end{equation}
		\item If $x\neq y$, $g_{x,y}$ is the rational function
		\begin{equation}\label{Def_of_g_x,y_x_notequal_y}
		 	g_{x,y}=f_{x,y}\, g_{y,y}.
		\end{equation}
	\end{itemize}

Recall from Proposition \ref{Asymp-Extension_of_algebraicity_say} in \cite{Asymp}, that the functions $f_{x,y}$ and $g_{x,y}$, as rational functions over $\mathcal{C}$, are regular on a neighbourhood of $\overline{\{u(z):z\in\mathbb{D}(0,R)\}}$. Using the positivity of the coefficients in the definition of the polynomial maps $J\mapsto M_J$ (that depend of the given vertex $y$), we will characterize pairs of vertices of the tree $(x,y)$, such that $D_{u(R)}f_{x,y}\notequiv 0$ on the tangent space of $\mathcal{C}$ at $u(R)$, $T_{u(R)}\mathcal{C}$ (See Proposition \ref{Corollary_PSeries_A}).

\subsection{The projection \texorpdfstring{$\pi:\mathcal{C}\to\mathcal{C}_{<\infty}$}{projection pi} and Derivatives of \texorpdfstring{$\lambda$}{lambda}}\label{subsection_derivative_of_lambda} ~

\begin{Prop}\label{Prop_derivative_of_lambda}
	The derivative of $\lambda$ at $u(R)$ is zero on the tangent space $T_{u(R)}\mathcal{C}$.
\end{Prop}
\begin{proof}
$\lambda$ is a linear map thus its differential at any point identifies as $\lambda$ itself. From Proposition \ref{Prop_Spectrum_belonging_C} the tangent space $T_{u(R)}\mathcal{C}$ is $\{0\}\times \{0<\infty\}\times \mathbb{C}\cdot \nu_\infty\subset E$, where $\nu_\infty$ is a real vector in $E_\infty$, with all its coefficients positive. The result follows.
\end{proof}

\begin{Prop}\label{Prop_derive_seconde_de_lambda}
				The second derivative of $\lambda$ at $u(R)$ is non-zero in the following sense:

			For any parametrization of $\mathcal{C}$ in the neighbourhood of $u(R)$, 
			$$\eta:\tau\mapsto (r(\tau),J(\tau)),$$ 
			such that $\eta(0)=u(R)$, the map $\tau\mapsto\lambda\circ\eta (\tau) =r(\tau)$ is holomorphic and $r'(0)=0$, $r''(0)\neq 0$.
		\end{Prop}
		\begin{proof}
			Let $\eta:\tau\mapsto (r(\tau),J(\tau))$ be an arbitrary parametrization of $\mathcal{C}$ in the neighbourhood of $u(R)$, such that $\eta(0)=u(R)$. From the previous proposition, we have $r'(0)=0$ and from Proposition \ref{Prop_Spectrum_belonging_C} we furthermore have that, $J'(0)$ belongs to $\Lambda_{ 1/R}=\{0_{<\infty}\}\times \Lambda_{1/R}(\infty)$, the generalized eigenspace of $D_{v_R}\psi$ associated with the eigenvalue $1/R$. In particular there exists $c\in\mathbb{C}\setminus\{0\}$, such that $J'(0)=c\cdot \nu_\infty$, where $\nu_\infty$  is a vector in $\{0_{<\infty}\}\times E_\infty\subset E$, with all its coordinates in $E_\infty$ positive. Since $\eta$ parametrizes $\mathcal{C}$, for any $\tau$,
			$$J(\tau)=r(\tau)\psi(J(\tau)).$$
Differentiating two times this equation and evaluating at $\tau=0$ we get:
\begin{equation*}
    J''(0)= r''(0)\psi(J(0))+r(0)D^2_{J(0)}\psi(J'(0),J'(0))+r(0)D_{J(0)}\psi(J''(0)).
\end{equation*}

By contradiction suppose that $r''(0)=0$, then, after replacing $r(0)$ by $R$, its value, and $J'(0)$ by $c\cdot\nu_\infty$, in the above equality we get
\begin{equation}\label{eqts9}
   \left(I-RD_{J(0)}\psi\right)\dfrac{J''(0)}{R c^2}= D^2_{J(0)}\psi(\nu_\infty,\nu_\infty).
\end{equation}
The projection of $\nu_\infty$ onto $E_{< \infty}$, parallel to $\mathbb{C}\times\{0\}\times E_\infty$ is zero; from the form (\ref{Form_of_D_J_psi}) of the differential $D_{v_z}\psi$, we infer that the projection to $E_{<\infty}$ of the right hand side of the above is also zero. 
Besides $\frac{1}{R}$ is not on the spectrum of any $\vec{\psi}_\alpha(R)$ for $\alpha\prec \infty$ (Proposition \ref{belong_spectrum_alpha}), thus $ \left(I-RD_{J(0)}\psi\right)$ restricted to $E_{<\infty}$ is invertible. Projecting in $E_{<\infty}$ the equality (\ref{eqts9}) above, we get that all the coordinates in $E_{<\infty}$ of the vector $J''(0)$ are zero. Thus $J''(0)$ belongs to $\{0_{<\infty}\}\times E_\infty$.

By Proposition \ref{degre_infini_de_psi} the polynomial map with non-negative coefficients $\psi:E\to E$ has degree at least two with respect to the variables in $E_\infty$. Since $u(R)$ is a vector with positive coefficients, the bilinear map $D^2_{J(0)}\psi$ is non-negative and non-zero over $E_\infty\times E_\infty$, that is the quadratic map $J^{(\infty)} \mapsto D^2_{J(0)}\psi(J^{(\infty)},J^{(\infty)})$ is non-zero, and sends non-negative vectors of $E_\infty$ to non-negative vectors in $E_\infty$. This observation in mind, since $\nu_\infty$ is a real vector of $E_\infty$ with positive coordinates, the projection over $E_\infty$ of the right-hand-side of (\ref{eqts9}) is a non-zero, non-negative real vector of the vector space $E_\infty$.

But $J''(0)$ is in $\{0_{<\infty}\}\times E_\infty$, thus the image of $\dfrac{J''(0)}{R c^2}$ by $\left(I-RD_{J(0)}\psi\right)$ is the image of its component  $\dfrac{J_{\infty}''(0)}{Rc^2}$ in $E_\infty$,  by $\left(I-R\dfrac{\partial \psi_\infty}{\partial J^{(\infty)}}(R,v_R)\right)$, where we see $E_\infty$ has a vector subspace of $E$.

Now recall that $\dfrac{\partial \psi_\infty}{\partial J^{(\infty)}}(R,v_R)$ is an irreducible Perron operator (Corollary \ref{Cor_partial_psi_infini_is_Perron_Irred}) with spectral radius equals to $1/R$. From theorem \ref{Asymp-Thm_Perron_irred} in the Appendix of \cite{Asymp}, the left hand-side of (\ref{eqts9}) can not be non-negative and non-zero. That contradicts the previous paragraph. Hence it was absurd to assume $r''(0)=0$.
\end{proof}

\subsection{Derivatives of \texorpdfstring{$f_{x,y}$}{fxy} and \texorpdfstring{$g_{x,y}$}{gxy}}

\begin{Lem}\label{Lem_der_de_pi}
	The projection $\pi_{<\infty}:\mathcal{C}\to\mathcal{C}_{<\infty}$, that sends $(z,J^{(<\infty)},J^{(\infty)})\in\mathcal{C}$ to $(z,J^{(<\infty)})\in\mathcal{C}_{<\infty}$, has zero derivative at $u(R)$:
	$$D_{u(R)}\pi_{<\infty}\equiv 0, \quad \text{ over } T_{u(R)}\mathcal{C}.$$
\end{Lem}
\begin{proof} The map $\pi_{<\infty}$ is the restriction to $\mathcal{C}$, of the linear projection from $\mathbb{C}\times E$ into $\mathbb{C}\times E_{<\infty}$ along $E_\infty$, that we also denote by $\pi_{<\infty}$. In particular $D_{u(R)}\pi_{<\infty}$ identifies as $\pi_{<\infty}$. From Proposition \ref{Prop_Spectrum_belonging_C}, we have
	$T_{u(R)}\mathcal{C}=\{0\}\times \{0_{<\infty}\}\times \mathbb{C}\cdot \nu_\infty$, thus $D_{u(R)}\pi_{<\infty}$ is zero over $T_{u(R)}\mathcal{C}$.
\end{proof}

\begin{Prop}\label{Prop_PSeries_expansion}
Let $x,y$ be two vertices of $X$, and $f_{x,y}$ and $g_{x,y}$ be the rational functions over $\mathbb{C}\times E\approx\mathbb{C}\times\mathbb{C}^N$, from Proposition \ref{Asymp-algebraicity_of_F} and  Corollary \ref{Asymp-Cor_algbraicity_of_Greens_function} in \cite{Asymp} respectively. These rational functions admit a convergent power series expansion in the neighbourhood of the origin, with non-negative coefficients, that is a power series:
\begin{equation}\label{power_expansion_1}
	\sum_{n,k_1,...,k_N\in\mathbb{N}} a_{n,k_1,...,k_N} \tau^n J_1^{k_1}\cdots J_N^{k_N}, \quad \text{ with }a_{n,k_1,...,k_N}\geq 0,
\end{equation}
converging on a polydisk 
\begin{equation}\label{polydisc_Q}
 	Q=\left\lbrace (\tau,J)\in\mathbb{C}\times E :
 	 |\tau|\leq R + \varepsilon, \forall [\theta]\in\Gamma\backslash\Xi, |J([\theta ])|\leq v_R({[\theta]})+\varepsilon 
 	 \right\rbrace,
\end{equation}
for some $\varepsilon>0$.
\end{Prop}
\begin{proof}
	Fix $y$ an arbitrary vertex of $X$.
	In a neighbourhood of the origin $(z,J)=(0,0_E)$, the operator $\left( I-zM_J \right)^{-1}$ equals the series of matrices
	\begin{equation}\label{power_series_of_matrices}
		\sum_{n\geq 0} z^n M_J^n
	\end{equation}
	And thus each coefficient of the matrix $\left( I-zM_J \right)^{-1}$ is a power series with non-negative coefficients of the form (\ref{power_expansion_1}). Since $\bar{p}$ is a real vector with non-negative coefficients we infer from the definitions of the functions \og $f_{x,y}$ \fg{} , that for any vertex $x$, the function $f_{x,y}$ admits a power series expansion in the neighbourhood of $(z,J)=(0,0_E)$ of the form (\ref{power_expansion_1}). 
	
	Lastly in a neighbourhood of $(z,J)=(0,0_E)$, from formula (\ref{Def_of_g_x,y_x_equal_y}), $g_{y,y}$ satisfies  
	$$g_{y,y}(z,J)=\sum_n \left(\sum_c z p(y,c)f_{c,y}(z,J)\right)^n.$$
	Then, from the result for the $f_{c,y}$, we deduce that $g_{x,y}$ also admits a power series expansion with non-negative coefficients of the form (\ref{power_expansion_1}).
	
	The convergence of such power series expansions on some polydisk $Q$ defined like (\ref{polydisc_Q}) for some $\varepsilon>0$ is a consequence of Proposition \ref{Prop_rational_extension} in the Appendix, with $n=1+dim_{\mathbb{C}}(E)$, with $F=f_{x,y}$ or $g_{x,y}$, $Z=u$ the function $z\mapsto(z,v_z)$ defined on $\overline{\mathbb{D}(0,R)}$ (parameterizing $\mathcal{C}$ in a neighbourhood of the origin, see Lemma \ref{Asymp-Lem_Green_parametrization_of_C} in \cite{Asymp})  and $f=F|_{\mathcal{C}}$ the restriction of the rational map $F$ to the Lalley's curve $\mathcal{C}$, which is regular thanks to Proposition \ref{Asymp-Extension_of_algebraicity_say} in \cite{Asymp}.
\end{proof}
\begin{Prop}\label{Corollary_PSeries_A}
	We take the same notation as in the previous Proposition \ref{Prop_PSeries_expansion}, and consider (\ref{power_expansion_1}) to be the power series expansion of $f_{x,y}$. For convenience, we suppose that the isomorphism $E\approx \mathbb{C}^N$ sends the subspace $E_{<\infty}$ to $\mathbb{C}^L\times \{0_{N-L}\}\subset \mathbb{C}^N$ and $E_\infty$ to $\{0_L\}\times \mathbb{C}^{N-L}$, where $L:=\dim(E_{<\infty})$.
	
		Let $x,y$ be two distinct vertices of $X$, the followings are equivalent:
	\begin{enumerate}[label=\roman*)]
		\item $f_{x,y}$ factors through $\pi_{<\infty}$;
		\item $D_{u(R)}f_{x,y}\equiv 0$ on $T_{u(R)}\mathcal{C}$;
		\item $R_F(x,y)> R$;
		\item For any $(n,k_1,...,k_N)\in\mathbb{N}^{1+N}$,  if $k_j>0$ for some $j>L$ then the coefficient of the product $\tau^n J_1^{k_1}\cdots J_N^{k_N}$ in (\ref{power_expansion_1}) is zero, $$a_{n,k_1,...,k_N}=0.$$
		\item There exists an integer $M$ such that for any $p$-admissible path $\gamma=(\omega_0,...,\omega_n)\in\mathcal{P}h(x,y;\{y\}^\complement)$, $\max_i(d(\omega_i,y))\leq M$.
	\end{enumerate}
\end{Prop}
\begin{proof}
We have already proven $v)\Rightarrow iii)$ in Lemma \ref{Asymp-Lem:five_to_three} from \cite{Asymp}.

We first prove the equivalences $i)\Leftrightarrow ii) \Leftrightarrow iv)$, then we prove $iii)\Leftrightarrow iv)$. Lastly we prove $non(v))\Rightarrow non(iii))$ in Proposition \ref{Prop:three_to_five}.

	It is clear from the definition of $\pi_{<\infty}:\mathcal{C}\to\mathcal{C}_{<\infty}$ that $iv)\implies i)$.	From Lemma \ref{Lem_der_de_pi}, $i)\implies ii)$.
	
	Recall from Proposition \ref{Prop_Spectrum_belonging_C} that we have $T_{u(R)}\mathcal{C}=\{0\}\times \{0_{<\infty}\}\times \mathbb{C}\cdot \nu_\infty$ where $\nu_\infty\in E_\infty$ is a real vector with strictly positive coefficients. Now since the derivative of the power series (\ref{power_expansion_1}) on the polydisk $Q$, is the sum of the derivatives of the maps 
	\begin{equation}\label{eqts4}
	 (\tau,J_1,...,J_N)\mapsto a_{n,k_1,...,k_N} \tau^n J_1^{k_1}\cdots J_N^{k_N},
	 \end{equation} it follows by non-negativity of the $a_{n,k_1,...,k_N}$ that the derivative of (\ref{eqts4}) at $u(R)$ over $T_{u(R)}\mathcal{C}$ is zero if, and only if $k_{L+1}=...=k_N=0$ or $a_{n, k_1, ..., k_n} = 0$. Hence $ii)\Longleftrightarrow iv)$.

At this point we have shown that the assertions $i)$, $ii)$ and $iv)$ are equivalent.  We now prove that $iii)$ is equivalent to $i)$, $ii)$ and $iv)$, by proving $i)\implies iii)$ and $non(iv))\implies non(iii))$.

	We have that $i)\implies iii)$, for if $\varphi_{x,y}$ stands for the factor of $f_{x,y}$ through $\pi_{<\infty}$, then $\varphi_{x,y}$ is a rational function over $\mathcal{C}_{<\infty}$ that is regular on a neighbourhood of the range of $v^{(<\infty)}$ over $\overline{\mathbb{D}(0,R)}\varsubsetneq \mathbb{D}(0,R_{<\infty})$. Hence $F_z(x,y)=\varphi_{x,y}\circ v^{(<\infty)}(z)$ is holomorphic on a neighbourhood of $\overline{\mathbb{D}(0,R)}$, thus we have $R_F(x,y)>R$. 
	
	We prove $iii)\implies iv)$ by contraposition. Recall that $E=\mathcal{F}(\Xi,\mathbb{C})^\Gamma\approx\mathcal{F}(\Gamma\backslash\Xi,\mathbb{C})$. Set $[\theta_1],...,[\theta_N]\in\Gamma\backslash\Xi$, be such that the isomorphism $\mathbb{C}^N\to E$ sends $J=(J_i)_{1\leq i\leq N}\in\mathbb{C}^N$ to the function $\sum_i J_i \mathds{1}_{[\theta_i]}$ in $E$. In particular, for any $i>L$, the orbit $[\theta_i]$ is in $\Xi_\infty$, and for any $i\leq L$, the orbit $[\theta_i]$ is in $\Xi_{<\infty}$. If for $i$ in $\{1,...,N\}$, we denote $\theta_i:=(a_i,b_i)_{y_i}$, then we have $v_z({[\theta_i]})=\sum_{n\geq 0} p^{(n)}(a_i,b_i,\mathcal{B}(y_i)^\complement)z^n$. Thus for any complex number $z\in\mathbb{D}(0,R)$, we have equality
		\begin{equation}\label{eqts10}
		\begin{split}
		F_z(x,y)=&\sum_n p^{(n)}(x,y;\{y\}^\complement)z^n
		\\
		&=f_{x,y}(v_z)
		\\
		&=\sum_{n,k_1,...,k_N\in\mathbb{N}} a_{n,k_1,...,k_N} z^n \prod_{i=1}^N \left(\sum_{m\geq 0} p^{(m)}(a_i,b_i,\mathcal{B}(y_i)^\complement)z^m\right)^{k_i} 
			\end{split}
		\end{equation}
		Expanding the last equality and using the non-negativity of the coefficients, we deduce the following: Let $j\in\{1,...,N\}$ be such that there exists $(n,k_1,...,k_N)\in\mathbb{N}^{N+1}$ that verifies $a_{n,k_1,...,k_N}\neq 0$ and $k_j>0$, then there exist a positive constant $c>0$ and a non-negative integer $l\geq n$ such that for any non-negative integer $m\in\mathbb{N}$, we have:
		\begin{equation}\label{eqts11}
		 p^{(m+l)}(x,y;\{y\}^\complement)\geq c \, a_{n,k_1,...,k_N} p^{(m)}(a_j,b_j,\mathcal{B}(y_j)^\complement).
		 \end{equation}
		
In particular if $non(iv))$ is satisfied, that is there exists $(n,k_1,...,k_N)\in\mathbb{N}^{1+N}$,  with $k_j>0$ for some $j>L$ and $a_{n,k_1,...,k_N}>0$, then the inequality (\ref{eqts11}) hold and the Cauchy-Hadamard formula gives:

$$\frac{1}{R_{F}(x,y)}=\limsup_m \left(p^{(m+l)}(x,y;\{y\}^\complement)^{\frac{1}{m+l}}\right)\geq \limsup_m \left(p^{(m)}(a_j,b_j,\mathcal{B}(y_j)^\complement)\right)^{\frac{1}{m}}=\frac{1}{R}.$$
Hence $non(iv))\implies non(iii))$.

The equivalence of the assertions $i)$ to $iv)$ with the last assertion $v)$ is given by the Lemma \ref{Asymp-Lem:five_to_three} from \cite{Asymp} and by the Proposition \ref{Prop:three_to_five} below.
\end{proof}

\begin{Rem}\label{Rem_for_g}
	If we replace the \og $f_{x,y}$ \fg{} by \og $g_{x,y}$ \fg{} in the Proposition \ref{Corollary_PSeries_A}. Then the result still holds, and the proof of the equivalence of the four assertion $i)- iv)$ is the same mutatis-mutandis, as the one given above. But since the radius of convergence of the Green's function \og $G_z(x,y)$ \fg{} is always equal to $R$, we get the following Proposition \ref{Prop_derivative_of_g}:
\end{Rem}
\begin{Prop}\label{Prop_derivative_of_g}

Let $x,y$ be two vertices of $X$ and $g_{x,y}$ be the rational function over $\mathbb{C}\times E\approx\mathbb{C}\times\mathbb{C}^N$, from Corollary \ref{Asymp-Cor_algbraicity_of_Greens_function} in \cite{Asymp}.

	The derivative $D_{u(R)}g_{x,y}$, of $g_{x,y}$ at $u(R)$ is non-zero over $T_{u(R)}\mathcal{C}$.
\end{Prop}

\begin{Prop}\label{Prop:three_to_five}
Let $x,y$ be two vertices of $X$ and $f_{x,y}$ be the rational function over $\mathbb{C}\times E\approx\mathbb{C}\times\mathbb{C}^N$, from Proposition \ref{Asymp-algebraicity_of_F} in \cite{Asymp}. Suppose that the assumption $v)$ of Proposition \ref{Corollary_PSeries_A} is not satisfied, that is, for any non-negative integer $M$, there exists a $p$-admissible path $\gamma=(\omega_0,...,\omega_n)\in\mathcal{P}h(x,y;\{y\}^\complement)$, with $\max_i(d(\omega_i,y))\geq M$.
Then the radius of convergence of the first-passage generating function $F_z(x,y)$ is $R_F(x,y)=R$.
\end{Prop}
\begin{proof}
Let $M>0$, from Proposition \ref{Stagnation_of_V_N}, be such that for any triple $(a,b)_{y'}$ in $\Xi$, if there exists a $p$-admissible path $\gamma\in\mathcal{P}h(a,b;\mathcal{B}(y')^\complement)$ satisfying $\max_{w\in\gamma}d(w,y')>M$, then $(a,b)_{y'}$ belongs to $\Xi_\infty$. Thanks to our assumptions, let $\gamma$ be a $p$-admissible path in $\mathcal{P}h(x,y;\{y\}^\complement)$ such that $\max_{w\in\gamma}d(w,y)>M+d(y,x)+2k$.
Decompose $\gamma$ with respect to all its passages in the ball $\mathcal{B}(y)$ (See figure \ref{figure where we have numbered the steps in B(y) and cut the gamma path between these numbers}). That is let $c_0,...,c_m\in\mathcal{B}(y)$ be the vertices of $\gamma$ in $\mathcal{B}(y)$.
Let $\gamma_0\in\mathcal{P}h(x,c_0;\mathcal{B}(y)^\complement)$, $\gamma_1\in\mathcal{P}h(c_0,c_1;\mathcal{B}(y)^\complement)$, ..., $\gamma_m\in\mathcal{P}h(c_{m-1},c_m;\mathcal{B}(y)^\complement)$, with $\gamma_0=(x)=(c_0)$ if $x$ is in the ball $\mathcal{B}(y)$, be $p$-admissible paths such that 
$$\gamma=\gamma_0\ast\gamma_1\ast\cdots\ast\gamma_m.$$
Among these $p$-admissible paths, one contains $w^*$ a vertex that is at distance $d(w^*,y)>M+d(x,y)+2k$ from $y$, say $\gamma_j$ for some $j\in\{0,1,...,m\}$. Then decompose $\gamma_j$ with respect to its first step: $\gamma_j=(c_{j-1},c_{j}')\ast\gamma_j'$, taking $c_{-1}=x$ if $j=0$. Note that $c_j'$ is not in the ball $\mathcal{B}(y)$, otherwise we would have $c'_j=c_j$, and the $p$-admissible path $\gamma_j$ would be $(c_{j-1},c_j)$, which contradicts our assumption on $\gamma_j$.

Lastly decompose $\gamma_j'\in\mathcal{P}h(c_j',c_{j};\mathcal{B}(y)^\complement)$ along the geodesic segment $[c_j',y]$ with respect to $\mathcal{B}$:
$$\gamma_j'=\gamma_1''\ast\cdots\ast\gamma_l''.$$
Among these $p$-admissible paths, one of them contains the vertex $w^*$, say $\gamma_i''$, for some integer $i\in\{1,...,l\}$. The $p$-admissible path $\gamma_i''$ belongs to some set $\mathcal{P}h(a,b;\mathcal{B}(y')^\complement)$, with $(a,b)_{y'}\in\Xi$ and $y'\in[c_j',y]$.

Since $d(w^*,y')\geq d(w^*,y)-d(y,y')$ and $d(y,y')\leq d(y,c_j')\leq d(y,c_{j-1})+d(c_{j-1},c_j)\leq k+ d(y,x)+k$, we have 
$$d(w^*,y')> M.$$
Thus from Proposition \ref{Stagnation_of_V_N}, the triple $(a,b)_{y'}$ belongs to $\Xi_\infty$. Finally, let $\gamma_s$ and $\gamma_b$ be the $p$-admissible path in $\mathcal{P}h(x,a;\{y\}^\complement)$ and $\mathcal{P}h(b,y;\{y\}^\complement)$, respectively, defined by:
\begin{align*}
\gamma_s&=\gamma_0\ast\cdots\ast\gamma_{j-1}\ast (c_{j-1},c_{j}')\ast\gamma_1''\ast\cdots\ast\gamma_{i-1}''\\
\gamma_b&=\gamma_{i+1}''\ast\cdots\ast\gamma_l''\ast\gamma_{j+1}\ast\cdots\ast\gamma_m.
\end{align*}
Then for any $p$-admissible path $\gamma''\in\mathcal{P}h(a,b;\mathcal{B}(y')^\complement)$, the concatenated $p$-admissible path $\gamma_s\ast\gamma''\ast\gamma_b$ is in $\mathcal{P}h(x,y;\{y\}^\complement)$. Thus for any non-negative real number $r\geq 0$,
we have:
$$w_r(\gamma_s)\underbrace{
\left(\sum_{\gamma''\in\mathcal{P}h(a,b;\mathcal{B}(y)^\complement)} w_r(\gamma'') \right)
}_{\displaystyle=G_r(a,b;\mathcal{B}(y)^\complement)} w_r(\gamma_b)
\leq \underbrace{\sum_{\gamma\in\mathcal{P}h(x,y;\{y\}^\complement)} w_r(\gamma)}_{\displaystyle=F_r(x,y)}
$$
In particular since $(a,b)_{y'}$ is in $\Xi_\infty$, the radius of convergence $R(a,b;\mathcal{B}(y)^\complement)$ of the restricted Green's function $G_z(a,b;\mathcal{B}(y)^\complement)$ is equal to $R$ (Remark \ref{Rem_R_infini_is_min}). Hence we get $R_F(x,y)\leq R$.
\end{proof}

\begin{figure}[ht]
\centering
		\includegraphics[scale=0.5]{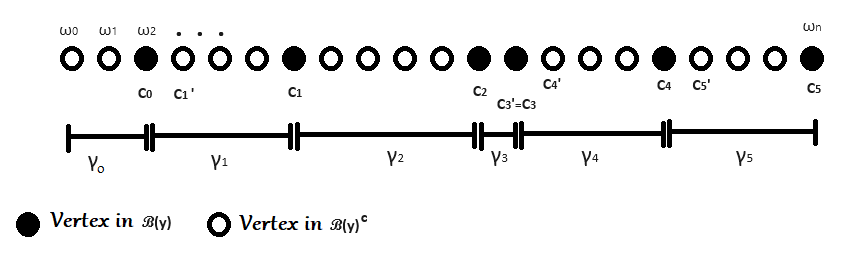}
		\caption{Example of decomposition used in the proof of Proposition \ref{Prop:three_to_five}.}
		\label{figure where we have numbered the steps in B(y) and cut the gamma path between these numbers}
	\end{figure}
\newpage
\appendix

\begin{center}
\huge{\textbf{Appendix}} 
\end{center}

\section{Examples of Dependency Digraph}\label{Section_Examples}
    Let $(\mathbb{Z}/2\mathbb{Z})^{\star 3}$ be the free product of three copies of $\mathbb{Z}/2\mathbb{Z}$. Denote by $a,b$ and $c$ the generators of each copies of $\mathbb{Z}/2\mathbb{Z}$.
    
	In the following Examples we are looking at finite range random walk on the Cayley graph $Cay(F,\{a,b,c\})$ of the free product $F=(\mathbb{Z}/2\mathbb{Z})^{\star 3}$, with respect to the finite set of generators $\{a,b,c\}$. That is, consider a measure $\mu$ of probability over $(\mathbb{Z}/2\mathbb{Z})^{\star 3}$, and set $p_\mu:F\times F\to[0,1]$ to be the transition kernel over $F$, $(x,y)\mapsto \mu(x^{-1} y)$. If $\operatorname{Supp}(\mu)$ generates $F$ as a semi-group then we get a discrete Markov chain $(F,p_\mu)$ that is irreducible. Note that in this case the group $F$ identifies as a group of automorphisms of the Cayley graph preserving the transition kernel $p_\mu$.

    The letter $e$ represents the neutral element in $F$. The graphs represented in figure \ref{figure graphes} are the associated quotient dependency digraph $\mathcal{V}$ from Section \ref{Section_The_Graph}, for three different support for the step distribution $\mu$. Namely $\operatorname{Supp}(\mu)=\{a,b,c,ab\}$, $\operatorname{Supp}(\mu)=\{a,ac,ba\}$ and $\operatorname{Supp}(\mu)=\{a,ab,ac\}$ that are all included in $\mathcal{B}_2(e)$.
%

	\begin{figure}
		\centering
		\begin{subfigure}[b]{\textwidth}
            \centering
            \includegraphics[width=0.92\textwidth]{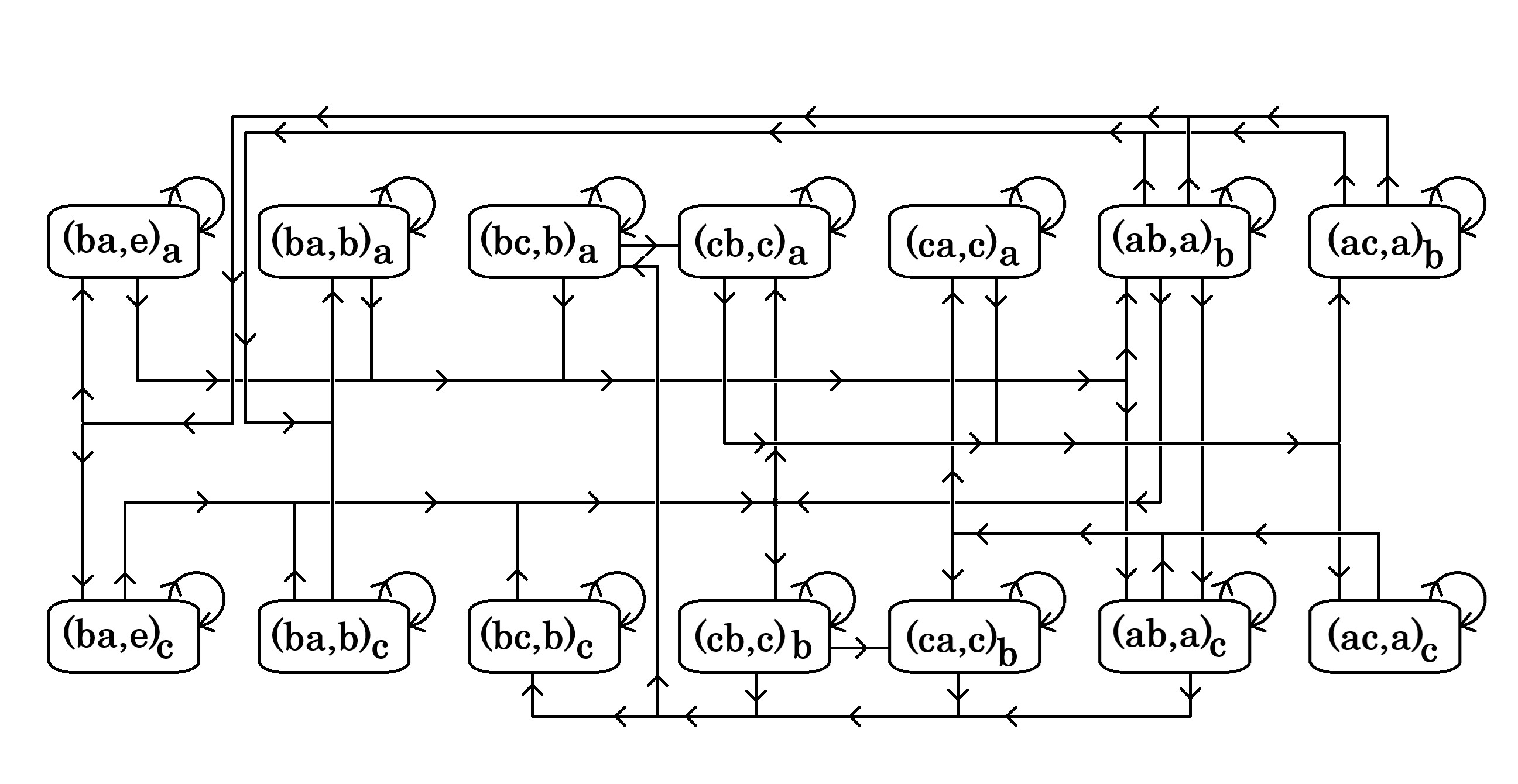}
        \end{subfigure}
        \hfill
        \begin{subfigure}[b]{\textwidth}
            \centering
           \includegraphics[width=0.92\textwidth]{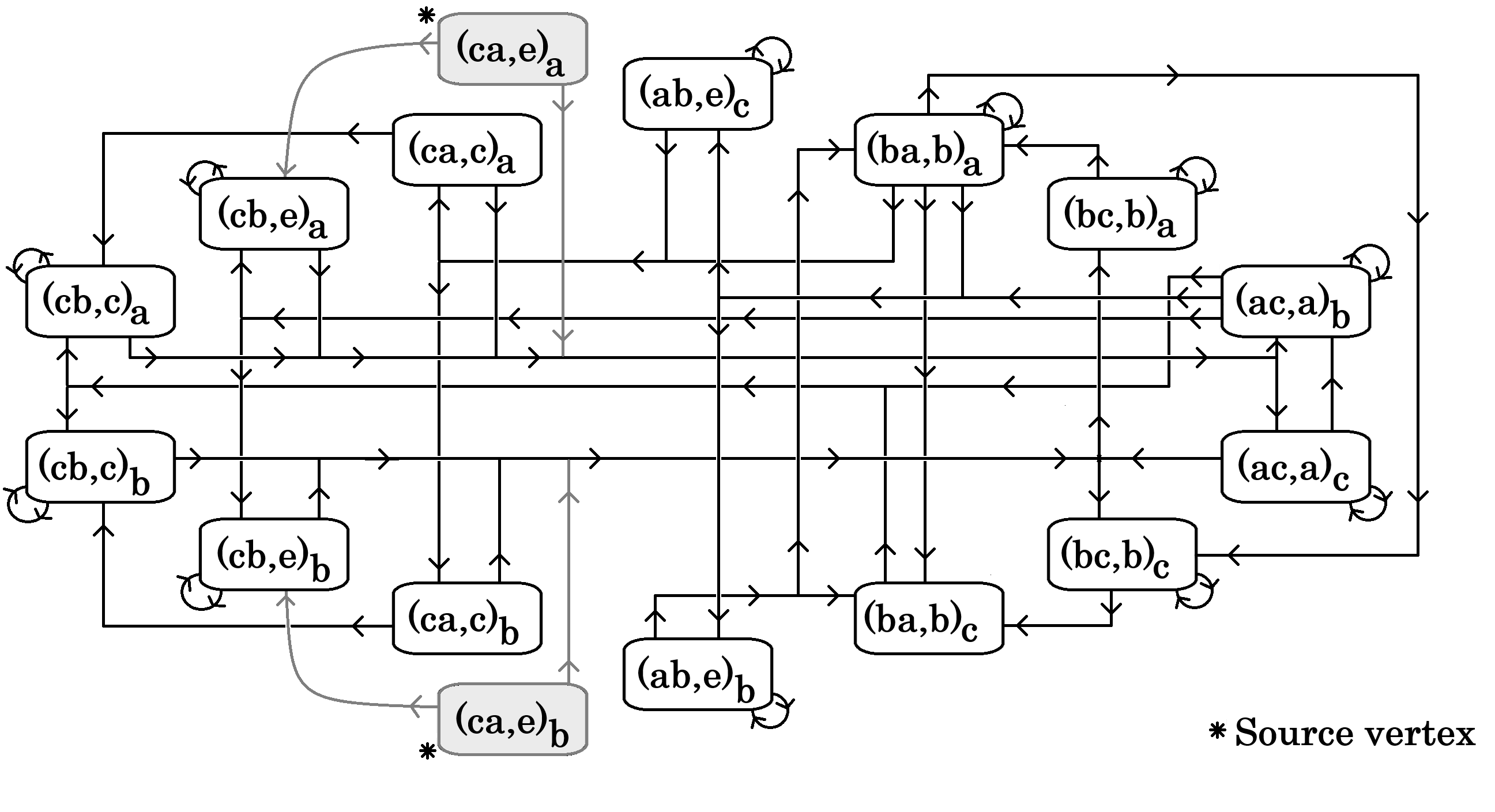}
        \end{subfigure}
        \begin{subfigure}[b]{\textwidth}
            \centering
           \includegraphics[width=0.92\textwidth]{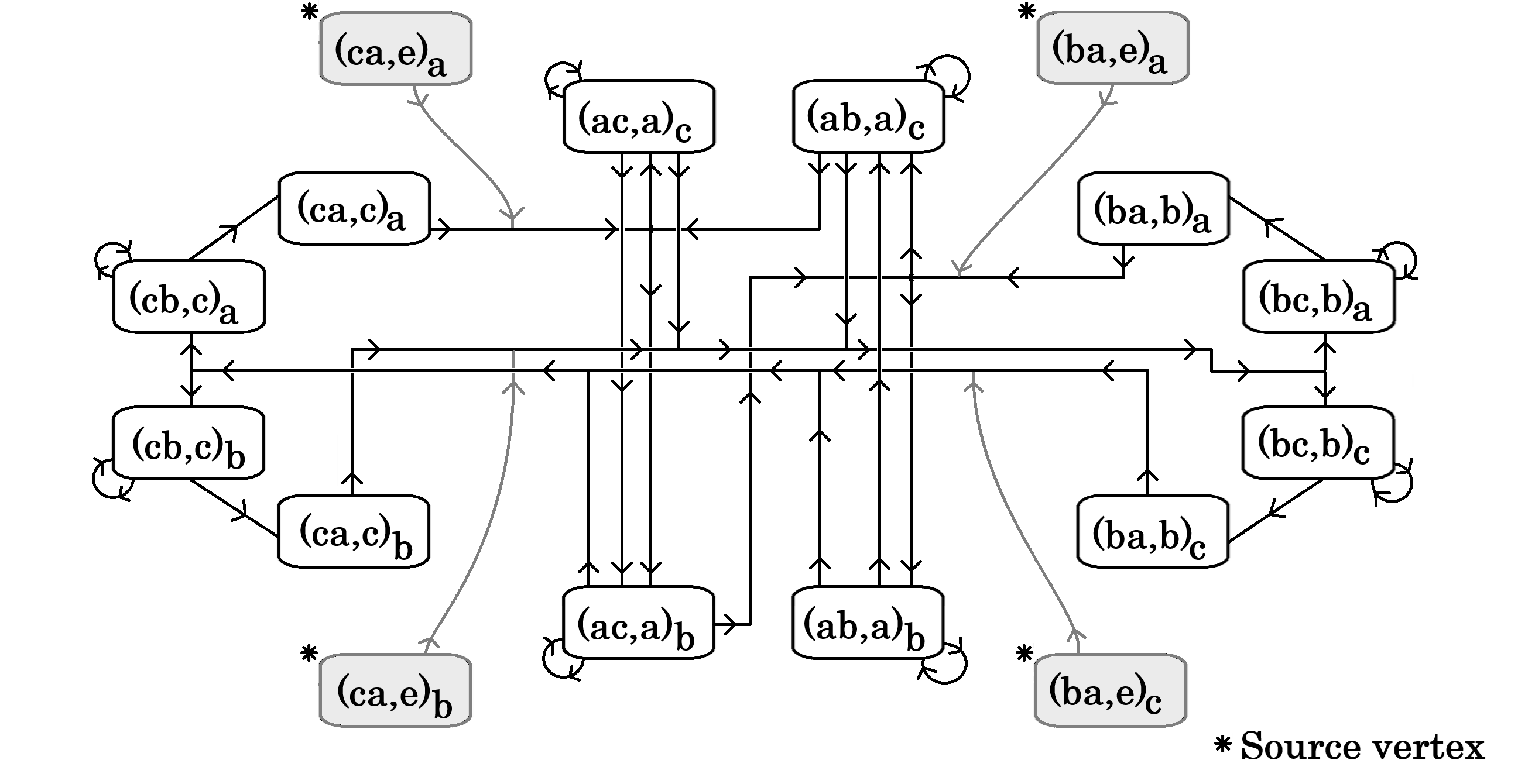}
        \end{subfigure}
       \caption{Dependancy digraph $\mathcal{V}$ for $\mu$ with support (from top to bottom) $\operatorname{Supp}(\mu)=\{a,b,c,ab\}$, $\operatorname{Supp}(\mu)=\{a,ac,ba\}$ and $\operatorname{Supp}(\mu)=\{a,ab,ac\}$}\label{figure graphes}
	\end{figure}

	\section{Rational fraction admitting power series expansion with non-negatives coefficients }
\begin{Prop}\label{Prop_rational_extension}
Let $\mathcal{C}\subset\mathbb{C}^n$ be an algebraic curve and consider a multivariable rational function $F\in\mathbb{C}(X_1,...,X_n)$ that admits a formal power series expansion with \underline{non-negative} coefficients:
	\begin{equation}\label{formula_t6}
	    \sum_{k_1,...,k_n \geq 0} a_{k_1,...,k_n} X_1^{k_1}\cdot\cdot\cdot X_n^{k_n}, \quad a_{k_1,...,k_n}\geq 0,
	\end{equation}
 converging in a neighbourhood of $0\in\mathbb{C}^n$.

Let $Z:\overline{\mathbb{D}(0,R)}\to\mathcal{C}$, $t\mapsto Z(t)=(z_1(t),...,z_n(t))$ be a holomorphic parametrisation\footnote{That is a bijective continuous map on $\overline{\mathbb{D}}(0,R)$ that is holomorphic on $\mathbb{D}(0,R)$.} of $\mathcal{C}$ such that:
\begin{enumerate}[label=\roman*)]
\item $Z(0)=0$;
\item $Z(t)$ is real whenever $t$ is a real number;
\item For any $i\in\{1,...,n\}$, $z_i:[0,R]\to\mathbb{R}$ is strictly increasing;
\item For any $t\in\mathbb{D}(0,R)$, $|z_i(t)|\leq z_i(|t|)$.
\end{enumerate}
Assume that we have a rational function $f\in\mathbb{C}(\mathcal{C})$ over $\mathcal{C}$ that is regular on\footnote{that is to say $f$ does not possess any pole on... } the set $\{Z(t): t\in\mathbb{D}(0,R)\}\subset\mathcal{C}$, and such that $F|_\mathcal{C}=f$, this is equivalent to asking that for any complex number $t$ of modulus sufficiently small, we have
\begin{equation}\label{eq_a3}
 F(Z(t))=\sum_{k_1,...,k_n} a_{k_1,...,k_n} z_1(t)^{k_1}\cdot\cdot\cdot z_n(t)^{k_n}=f(Z(t)).
\end{equation}

If $f(Z(t))$ admits a finite limit when $t$ goes to $R$ , then $f$ is regular on a neighbourhood of the closure $\overline{\{Z(t):t\in\mathbb{D}(0,R)\}}$ and the series (\ref{formula_t6}) converges on a polydisk  
$$\{ (w_1,...,w_n) \in\mathbb{C}^n : \forall 1\leq i \leq n, |w_i|\leq z_i(R)+\varepsilon\},$$
for some $\varepsilon>0$.
\end{Prop}
Note that the regularity of $f$ on a neighbourhood of $Z(R)\in\mathcal{C}$, is a consequence of the existence of the limit $\lim_{t\to R}f(Z(t))$. 

To prove this proposition we will use the two following lemmas:
\begin{Lem}\label{Lem_a1}
Let $\sum_{k_1,...,k_n} a_{k_1,...,k_n} z_1^{k_1}\cdot\cdot\cdot z_n^{k_n}$ be a multivariable power series with non-negative coefficients $a_{k_1,...,k_n}\geq 0$ that converges on $\mathbb{D}(0,1)^n$; Let $Z=(Z_1,...,Z_n):[0,1]\to (\mathbb{R}^+)^n$ be a continuous function such that $Z(1)=(1,...,1)$ and for any $1\leq i \leq n$ and $t<1$, we have $Z_i(t)<1$; Lastly, let $f$ denote the function that associates to any real number $t\in[0,1[$, the value  $f(t)=\sum_{k_1,...,k_n} a_{k_1,...,k_n} Z_1(t)^{k_1}\cdot\cdot\cdot Z_n(t)^{k_n}$.

If $\lim_{t\to 1}f(t)=:f(1^-)$ exists, then $\sum_{k_1,...,k_n} a_{k_1,...,k_n}<\infty$ and equals $f(1^-)$.
\end{Lem}
\begin{proof}
	Let $0\leq t<1$, for any non-negative integer $K\in\mathbb{N}$, we have
	\begin{equation*}
		\sum_{\substack{k_1,...,k_n \\ k_1+...+k_n\leq K}} a_{k_1,...,k_n}Z_1(t)^{k_1}\cdot\cdot\cdot Z_n(t)^{k_n}\leq f(t).
	\end{equation*}
	Letting $t$ goes to $1$, we obtain
	\begin{equation*}
		\sum_{\substack{k_1,...,k_n \\ k_1+...+k_n\leq K}} a_{k_1,...,k_n}\leq f(1^-).
	\end{equation*}
	Lastly, taking the supremum for $K\in\mathbb{N}$ we get
	\begin{equation*}
		\sum_{k_1,...,k_n} a_{k_1,...,k_n}\leq f(1^-).
	\end{equation*}
	The result follows from Abel's theorem for power series (See \cite{Ahlfors_1978} p.41).
\end{proof}

\begin{Lem}\label{Lem_a2}
	Let $F\in\mathbb{C}(X_1,...,X_n)$ be a rational function admitting a formal power series expansion with non-negative coefficients:
	\begin{equation}\label{Multivariable_power_series}
	    \sum_{k_1,...,k_n} a_{k_1,...,k_n} X_1^{k_1}\cdot\cdot\cdot X_n^{k_n}, \quad a_{k_1,...,k_n}\geq 0.
	\end{equation}
	
	If $\sum_{k_1,...,k_n} a_{k_1,...,k_n}<\infty$ then there exists $\varepsilon>0$ such that the multivariable power series (\ref{Multivariable_power_series}) converges on the polydisc
	$$\{(w_1,...,w_n)\in\mathbb{C}^n: \forall 1\leq i \leq n, \, |w_i|\leq 1+\varepsilon \}$$
\end{Lem}
\begin{proof}
	Consider the complex rational function $f:t\mapsto F(t,...,t)$ - where we see $F$ as a complex valued function $F:\mathbb{C}^n\to\mathbb{C}\cup\{\infty\}$. By assumption, for any complex number $t\in\mathbb{D}(0,1)$, we have $f(t)= \sum_{k_1,...,k_n} a_{k_1,...,k_n} t^{k_1+...+k_n}$. Besides since the series $\sum_{k_1,...,k_n} a_{k_1,...,k_n}$ is finite by Abel's theorem for power series (See \cite{Ahlfors_1978} p.41), the non-negative real number $f(t)$ goes to $\sum_{k_1,...,k_n} a_{k_1,...,k_n}$ when $t\in [0,1[$ goes to $1$. In particular, $1$ is not a pole for $f$. Since the set of poles of $f$ forms a discrete set in $\mathbb{C}$, there exists $\varepsilon>0$ such that, $f$ does not possesses any pole in the disk $\mathbb{D}(1,2\varepsilon)\subset\mathbb{C}$. Then by Vivanti–Pringsheim theorem for singularities of power series with positive coefficients (See \cite{Remmert_1991}, p.235), the rational function $f$ does not possesses any pole in the disk $\mathbb{D}(0,1+2\varepsilon)$ and in particular 
	$$\sum_{k_1,...,k_n} a_{k_1,...,k_n}(1+\varepsilon)^{k_1}\cdot\cdot\cdot (1+\varepsilon)^{k_n}<\infty.$$
	Hence the multivariable power series (\ref{Multivariable_power_series}) converges on the polydisc
	$$\{(w_1,...,w_n)\in\mathbb{C}^n: \forall 1\leq i \leq n, \, |w_i|\leq 1+\varepsilon \}.$$
\end{proof}
\begin{proof}[Proof of Proposition \ref{Prop_rational_extension}] 

	Let $0<L\leq R$ be maximal such that for any $t\in\mathbb{D}(0,L)$, we have 
 \begin{equation}\label{eqclaim_a3}
 F(Z(t))=\sum_{k_1,...,k_n} a_{k_1,...,k_n} z_1(t)^{k_1}\cdot\cdot\cdot z_n(t)^{k_n}=f(Z(t)).
 \end{equation} 
 	Applying Lemma \ref{Lem_a1} with the continuous function $W:t\mapsto (\frac{z_1(Lt)}{z_1(L)},...,\frac{z_n(Lt)}{z_n(L)})$ and the power series $$\sum_{k_1,...,k_n} \left(a_{k_1,...,k_n} z_1(L)^{k_1}\cdot\cdot\cdot z_n(L)^{k_n}\right) z_1^{k_1}\cdot\cdot\cdot z_n^{k_n},$$ we obtain that $\sum_{k_1,...,k_n} a_{k_1,...,k_n} z_1(L)^{k_1}\cdot\cdot\cdot z_n(L)^{k_n}<\infty$. 
Then applying Lemma \ref{Lem_a2} to the polynomial $F(z_1(L)X_1,...,z_n(L) X_n)\in\mathbb{C}(X_1,...,X_n)$, that is the rational function $F$ whom, for each $i\in\{1,...,n\}$, we have substituted $X_i$ by $z_i(L)X_i$, admitting the power series expansion 
$$
\sum_{k_1,...,k_n} \left(a_{k_1,...,k_n} z_1(L)^{k_1}\cdot\cdot\cdot z_n(L)^{k_n}\right) X_1^{k_1}\cdot\cdot\cdot X_n^{k_n},$$
we get the convergence of the series (\ref{formula_t6}) in a polydisc
$$C_L:=\{w\in\mathbb{C}^n : \forall 1\leq i \leq n,\, |w_i|\leq z_i(L)+\varepsilon_L\},$$
for some $\varepsilon_L>0$.
In particular $F$ does not possesses any pole on this polydisk and neither does $f$ on the intersection of this polydisk with $\mathcal{C}$. 

Beside, if $L<R$, then there exists $\delta_L>0$ such that for any complex number $t\in\mathbb{D}(0,L+\delta_L)$, we have that $Z(t)$ belongs to the polydisk $C_L$. Thus the equality (\ref{eqclaim_a3}) holds for any $t\in\mathbb{D}(0,L+\delta_L)$. And in particular $L$ can't be maximal. Thus necessarily $L=R$, and  (\ref{eqclaim_a3}) holds for any $t\in\overline{\mathbb{D}(0,R)}$. Hence $f$ is regular on a neighbourhood of the closure $\overline{\{Z(t):t\in\mathbb{D}(0,R)\}}$. 

The proposition is proven.

\end{proof}
\section{Spectral Radius Properties}
\begin{Lem}\label{inequality_spectral_radius}
Let $n>0$ be a positive integer and let $A$ and $B$ be two $n\times n$ complex matrices. Suppose that $B$ has non-negative real coefficients and that we have coefficient-wise the inequality 
\begin{equation}\label{Ineq_lem_sr_1}
|A|\leq B.
\end{equation}
Then the spectral radius of $A$ is lower than the one of $B$:
\begin{equation}\label{Ineq_lem_sr_2}
\rho(A)\leq \rho(B).
\end{equation}
Besides, if the inequality in (\ref{Ineq_lem_sr_1}) is strict coefficientwise and $\rho(B)\neq 0$, then the inequality in (\ref{Ineq_lem_sr_2}) is also strict.
\end{Lem}
\begin{proof}
If $B$ is zero, then $A$ is zero too and (\ref{Ineq_lem_sr_1}) is trivial.
If $A=(a_{i,j})$ then we denote $|A|=(|a_{i,j}|)$ the matrix whose coefficients are the absolute values of those in $A$. We have for any vector $v\in\mathbb{C}^n$, the inequality
$$ |Av|\leq |A||v|\leq B|v|.$$
where $|v|$ is the vector whose coordinates are the absolute values of the coordinates of $v$.
Let $||.||$ denote the $1$-norm on $\mathbb{C}^n$, and let $v_A\in\mathbb{C}^n$ be an eigenvector of $A$, with norm $||v_A||=1$, associated with the eigenvalue $\lambda_A$ of $A$ of maximal modulus, that is $\lambda_A=\rho(A)$. Then for any non-negative integer $n\in\mathbb{N}$, we have
$$\rho(A)^n=||A^n v_A ||\leq ||\, |A|^n|v_A|\,||\leq || B^n |v_A|\,||\leq \sup_{||v||=1}||B^n v ||=:|||B^n|||.$$
Lastly using the spectral radius formula, also knwon as Gelfand's formula (See \cite{Rudin_FR_2009}, Théorème 18.9) 
$\rho(B)=\lim_{n\to\infty} |||B^n|||^{1/n}$, we get:
$$\rho(A)\leq \rho(B).$$

For the second point note that for any $\varepsilon\in ]0,1[$, we have $\rho((1-\varepsilon)B)=(1-\varepsilon)\rho(B)$. Thus if $\rho(B)\neq 0$, it suffices to find a $\varepsilon\in ]0,1[$ such that $|A|\leq (1-\varepsilon)B$ to obtain a strict inequality in (\ref{Ineq_lem_sr_2}).
\end{proof}

\begin{Lem}\label{inequality_spectral_radius_Perron}
Let $n>0$ be a positive integer and let $A$ and $B$ be two $n\times n$ complex matrices. Suppose that $B$ has non-negative real coefficients and that we have coefficient-wise the inequality 
\begin{equation}
|A|\leq B.
\end{equation}
Suppose now that there exists a pair $(i,j)\in\{1,...,n\}^2$ such that the $(i,j)$ coefficients of $A$ and $B$ satisfy $|A_{i,j}|< B_{i,j}$
 and suppose that $B$ is a Perron irreducible matrix, then the spectral radius of $A$ is strictly lower than the one of $B$,
\begin{equation}
\rho(A) < \rho(B).
\end{equation}
\end{Lem}
\begin{proof}
Without loss of generality (thanks to Lemma \ref{inequality_spectral_radius}), we can suppose that $A$ has non-negative coefficients,
so that $\rho(A)$ is an eigenvalue of $A$, and $\rho(B)$ is an eigenvalue of $B$.

Let $P$ be a polynomial with non-negative coefficients. Then the matrices $P(A)$ and $P(B)$ satisfies  $$P(A)\leq P(B).$$ 
Take $P\neq 0$ with non-negative coefficient to be such that, $P(B)$ has all its coefficients positive. Then coefficient-wise we have
$$P(A)A P(A) < P(B) B P(B).$$ 
For the coefficient indexed $(i,j)$ of $A$ (resp. $B$) appears in the expression of any coefficients of the above left (resp. right) handside matrix.

Lastly, from Lemma \ref{inequality_spectral_radius} we get
$$P(\rho(A))^2\rho(A)\leq \rho(P(A)2P(A))<\rho(P(B)BP(B))=P(\rho(B))^2\rho(B).$$
Which by strict increasing of $x\mapsto P(x)x P(x)$ on $[0,+\infty[$ gives $\rho(A)<\rho(B)$.
\end{proof}

\bibliographystyle{alpha}
\bibliography{sample.bib}

Institut Mathématiques de Bordeaux/ Institut Montpellierain Alexander Grothendieck\\
\indent E-MAIL: gchevalier@protonmail.com
\end{document}